\documentclass[11pt]{article}

\usepackage{textcase}

\usepackage{hyperref}
\usepackage{authblk}
\usepackage{fullpage}  
\usepackage{float}
\usepackage{diagbox}
\usepackage{pb-diagram}  		
\usepackage{esvect}
\usepackage{url}
\usepackage{tabularx}
\usepackage{algorithm, algorithmicx,algpseudocode, listings} 
\usepackage{multirow} 
\usepackage{hhline}  
\usepackage{amsmath}
\usepackage{xcolor}

\usepackage{graphicx}
\usepackage{graphicx,subfigure}
\usepackage{amsthm,amscd}
\usepackage{color}
\usepackage{booktabs}
\usepackage{threeparttable}    
\usepackage{appendix}
\usepackage{enumerate}
\usepackage{longtable}
\usepackage{lscape}

\usepackage{verbatim}
\usepackage{booktabs}
\usepackage{extarrows}
\usepackage{epstopdf}
\usepackage{lineno}
\usepackage{verbatim}
\usepackage{enumitem}
\usepackage{multirow}
\usepackage{listings}
\usepackage{amsfonts}
\usepackage{bm}
\usepackage{makecell}

\allowdisplaybreaks

\newlength\myindent
\newcommand\bindent[1]{%
  \begingroup
  \setlength{\itemindent}{\myindent}
}

\newcommand{\kwcomm}[1]{{}}
\newcommand{\kw}[1]{{}}
\newcommand{\dH}{\mathbf{\dot{H}}}
\newcommand{\comm}[1]{}
\newcommand{\pacomm}[1]{{}}

\newcommand{\delete}[1]{}

\renewcommand{\theequation}{\arabic{equation}}

\definecolor{mygreen}{rgb}{0,.5,0}

\newcommand{\be}{\begin{equation}}
\newcommand{\ee}{\end{equation}}
\newcommand{\bee}{\begin{equation*}}
\newcommand{\eee}{\end{equation*}}
\newcommand{\bea}{\begin{eqnarray}}
\newcommand{\eea}{\end{eqnarray}}
\newcommand{\beaa}{\begin{eqnarray*}}
\newcommand{\eeaa}{\end{eqnarray*}}
\newcommand{\R}{\mathbb{R}}

\newcommand{\qb}{\mathbf{\dot{B}}}
\newcommand{\qG}{\mathbf{\dot{G}}}
\newcommand{\qd}{\mathbf{\dot{D}}}
\newcommand{\qs}{\mathbf{\dot{S}}}

\DeclareMathOperator*{\argmax}{arg\,max}

\newcommand{\iprod}[2]{ \left\langle #1, #2 \right\rangle}

\newcommand{\A}{\mathbf{\dot{A}}}

\newcommand{\La}{\mathbf{\dot{\Lambda}}}

\newcommand{\qL}{\mathbf{\dot{L}}}
\newcommand{\qR}{\mathbf{\dot{R}}}

\newcommand{\D}{\mathbf{\dot{D}}}
\newcommand{\W}{\mathbf{\dot{W}}}
\newcommand{\qT}{\mathbf{\dot{T}}}
\newcommand{\Z}{\mathbf{\dot{Z}}}

\newcommand{\Q}{\mathbf{\dot{Q}}}
\newcommand{\Y}{\mathbf{\dot{Y}}}

\newcommand{\diag}{\mathrm{diag}}

\newcommand{\qH}{\mathbb{H}}
\newcommand{\qi}{\mathbf{i}}
\newcommand{\qj}{\mathbf{j}}
\newcommand{\qk}{\mathbf{k}}
\newcommand{\qU}{\mathbf{\dot{U}}}

\newcommand{\qV}{\mathbf{\dot{V}}}
\newcommand{\ie}{\mathbf{\;i.e.\;}}

\newcommand{\st}{\mathrm{s.t.}}

\newcommand{\mT}{\mathrm{T}}
\newcommand{\mF}{\mathrm{F}}

\newcommand{\X}{{\mathbf{\dot{X}}}}

\newcommand{\argmin}{\mathop{\mathrm{arg\, min}}}

\newtheorem{theorem}{Theorem}[section]
\newtheorem{lemma}{Lemma}[section]
\newtheorem{corollary}{Corollary}[section]
\newtheorem{assumption}{Assumption}[section]
\newtheorem{remark}{Remark}[section]

\date{}
\begin{document}

\title{SLRQA: A Sparse Low-Rank Quaternion Model for Color Image Processing with Convergence Analysis

}

\author[]{Zhanwang Deng\thanks{Academy for Advanced Interdisciplinary Studies, Peking University, Beijing, China. (\url{dzw_opt2022@stu.pku.edu.cn})} }
\author[]{Yuqiu Su\thanks{School of Mathematical Sciences, Xiamen University, Xiamen, China. (\url{yuqiusu@stu.xmu.edu.cn})} } 
\author[]{Wen Huang\thanks{Corresponding author. School of Mathematical Sciences, Xiamen University, Xiamen, China. (\url{wen.huang@xmu.edu.cn}). This work was supported by the National Natural Science Foundation of China (No.~12371311), the National Natural Science Foundation of Fujian Province (No.~2023J06004), the Fundamental Research Funds for the Central Universities (No.~20720240151), and Xiaomi Young Talents Program.
} }

\affil[]{}

\maketitle

\begin{abstract}

In this paper,  we propose a Sparse Low-rank Quaternion Approximation (SLRQA) model for color image processing problems with noisy observations. 
The proposed SLRQA is a quaternion model that combines low-rankness and sparsity priors without an initial rank estimation. 
A proximal linearized ADMM (PL-ADMM) algorithm is proposed to solve SLRQA and the global convergence is guaranteed under standard assumptions. 
When the observation is noise-free, a limiting case of the SLRQA, called SLRQA-NF, is proposed. Subsequently, a proximal linearized ADMM (PL-ADMM-NF) algorithm for SLRQA-NF is given. Since SLRQA-NF does not satisfy a widely-used assumption for global convergence of ADMM-type algorithms, we propose a novel assumption, under which the global convergence of PL-ADMM-NF is established.
In numerical experiments,
we verify the effectiveness of quaternion representation. Furthermore, for color image denoising and color image inpainting problems, SLRQA and SLRQA-NF demonstrate superior performance both quantitatively and visually when compared with some state-of-the-art methods.

 \textbf{Key words}: Color image denoising, Quaternion matrix completion, Nonconvex linearized ADMM.
\end{abstract}
\section{Introduction}

 Low-rankness plays an important role in image processing and has been used in
 image denoising~\cite{chen2019low}, image inpainting~\cite{guo2017patch}, image deblurring~\cite{ren2016image}, and image filtering problems~\cite{chen2014removing}.
A commonly encountered approach for exploiting the low-rankness is to find a matrix with the lowest rank such that certain matching errors are minimized. Due to the existence of the rank function, such problems are known as NP-hard~\cite{recht2010guaranteed}.
 To overcome this difficulty, Candès et al. have proved that the nuclear norm is the tightest convex relaxation of the NP-hard rank minimization function~\cite{candes2009exact}. In addition to the nuclear norm, other surrogates have also been proposed, such as weighted nuclear norm (WNNM)\cite{gu2014weighted}, Schatten $ p $-norm~\cite{nie2012low}, weighted Schatten $p$-norm~\cite{xie2016weighted}, and log-determinant penalty~\cite{kang2015logdet}, which have shown competitive effectiveness in various applications. 
In addition to low-rankness, sparsity is also a significant property of natural images that has been considered in image recovery~\cite{yang2022quaternion2,cai2016image}, video recovery~\cite{yang2022quaternion}, texture repairing~\cite{liang2012repairing}, and face recognition~\cite{zou2016quaternion}.
The intuition behind sparsity is that images under certain transforms are usually sparse, such as wavelet or discrete cosine transform (DCT).
Moreover, quaternion representation has shown great potential for color image processing problems. A color image can be encoded as a pure quaternion matrix.
The main advantage of quaternion representation is that the three color
channels simultaneously are treated by quaternion representation,
which can well keep the connection between image color channels \cite{xu2023deep}.  
Moreover, compared to tensor-based methods, quaternion-based methods not only use the information among channels but also preserve the orthogonality of the coefficient matrix for the three channels, 
see detailed discussions under the framework of dictionary learning in \cite[Section~III]{xu2015vector}. 
Furthermore, it is empirically shown in \cite{xu2015vector,chen2019low} that quaternion representation consistently
provides a more accurate approximation than monochromatic representation using the
reconstructed matrix of the same rank.

The three priors, i.e. low-rankness, sparsity, and quaternion representation, all play important roles in image processing. However, models based on all these three priors are still limited. In the following subsection, we propose a quaternion-based model that combines low-rankness and sparsity properties to address problems arising in image processing.



\subsection{A novel quaternion based model}

 Let $ \Y \in \mathbb{H}^{m \times n}$ denote an observed color image satisfying $\Y = \mathcal{A}(\X) + \boldsymbol{\varepsilon}$, where $\X \in \mathbb{H}^{m \times n}$ is the desired noiseless color image, $ \qH^{m \times n}$ denotes the set of all $m \times n$ quaternion matrices, $\mathcal{A}$ denotes a linear operator, and $\boldsymbol{\varepsilon}$ is a random noise\footnote{The definition of quaternion matrices is clarified in Section~\ref{2}.}. 
According to the principle of maximizing the posterior probability, an approach to recover the image $\X$ is to maximize likelihood estimation function, i.e., 
$
    \X_* = \argmax_\X P(\X|\Y).
$
It follows from the Bayes rule that
$
P(\X|\Y) = \frac{P(\Y|\X)P(\X)}{P(\Y)} \propto P(\Y|\X)P(\X).
$
By taking the logarithm, an equivalent optimization problem is given as follows:  
\begin{align}
    \X_* &= \argmin_{\X} -\log P(\Y|\X) - \log P(\X) = \argmin_{\X} E(\Y;\X) + R(\X) \label{maxium},
\end{align}
where $E(\Y;\X)$ is the negative log-likelihood function and is also called the data fitting term, and $R(\X)$ is the prior term. It is assumed throughout this paper that the entries of the noise $\boldsymbol{\varepsilon}$ are drawn from the Gaussian distribution with mean zero and variance $\tau^2$. It follows that $E(\Y;\X) = \frac{1}{2\tau^2} \|\mathcal{A}(\X)-\Y\|_\mF^2$ if all the entries of $\Y$ are observed, and  $E(\Y;\X) = \frac{1}{2\tau^2}\|\mathcal{P}_{\Omega}(\mathcal{A}(\X)) - \mathcal{P}_{\Omega}(\Y)\|_{\mathrm{F}}^2$ if only partial entries of $\Y$ are observed, where $\Omega$ denotes the indices of the observed entries, that is, $(\mathcal{P}_{\Omega}(\Y))_{ij} = \Y_{ij} $ if $(i,j) \in \Omega$ and $(\mathcal{P}_{\Omega}(\Y))_{ij} = 0 $ if $(i,j) \notin \Omega$.

By integrating the low-rankness and sparsity into the prior term, i.e., $R(\X) = \sum_i \phi(\sigma_i( \sqrt{ {\X}\strut^{*} \X + \varepsilon^2 \mathbf{I} }  ),\gamma) +\lambda p( \mathcal{W}(\X) )$, and letting $\W = \mathcal{W}(\X)$, we propose the following Sparse Low-rank
Quaternion Approximation (SLRQA) model for color image processing:
\be \label{model:new}
\begin{aligned}
    \min_{\X,\W \in \mathbb{H}^{m \times n} } \quad & \sum_i \phi(\sigma_i( \sqrt{ {\X}\strut^{*} \X + \varepsilon^2 \mathbf{I} }  ),\gamma) +\lambda p(\W) + \frac{1}{2\tau^2}\|\mathcal{P}_{\Omega}(\mathcal{A}(\X)) - \mathcal{P}_{\Omega}(\Y)\|_{\mathrm{F}}^2,\\
    \st \quad & \mathcal{W}(\X)=\W,
\end{aligned}
\ee
where $\X \in \mathbb{H}^{m \times n}$ denotes a color image, $\W \in \mathbb{H}^{m \times n}$ is the coefficient matrix under an orthogonal transform $\mathcal{W}$ such as quaternion wavelet transform or quaternion discrete cosine transform (QDCT), 
$\mathcal{A}$ is a known operator such as a blurring operator or the Radon transform,  $p(\W)$ is a function that 
promotes sparsity and has Lipschitz continuously differentiable such as a Huber function, i.e., $(p(\W))_{ij} =\frac{1}{2\delta} |\W_{ij}|^2$ if $|\W_{ij}| < \delta$; and $(p(\W))_{ij} =|\W_{ij}| - \frac{1}{2} \delta$ otherwise, $\delta > 0$ is a prescribed constant, $\varepsilon > 0$ is a smoothing parameter, $\phi$ is a nonconvex surrogate function that substitutes the rank function, $\sigma_i(\X)$ is the $i$-th singular value of $\X$ defined in Section~\ref{2}. 
The function $\sum_i \phi(\sigma_i (\sqrt{ {\X}\strut^{*} \X + \varepsilon^2 \mathbf{I} } ),\gamma)$ promotes the low-rankness of $\X$. We note that if $\varepsilon =0$, then $\sum_i \phi(\sigma_i( \sqrt{ {\X}\strut^{*} \X + \varepsilon^2 \mathbf{I} }  ),\gamma)$ reduces to $\sum_i \phi(\sigma_i(  {\X}  ),\gamma) $. If we futher assumes that $\phi(x,\gamma) = x$, then $\sum_i \phi(\sigma_i( \sqrt{ {\X}\strut^{*} \X + \varepsilon^2 \mathbf{I} }  ),\gamma)$ reduces to the nuclear norm.  Some nonconvex surrogate functions for~$\phi$ have been given in~\cite{chen2019low,lu2014generalized} and listed in Table~\ref{nonconvex2}. 


\begin{center}
\begin{table}[h] 
\caption{Some surrogate functions $\phi(x,\gamma)$}
\label{nonconvex2}
\resizebox{\textwidth}{!}{
\begin{tabular}{|c|c|c|c|c|c|c|} 
\hline
Nonconvex function  & Log-determinant  &  Schatten-$\gamma$  & Logarithm  & Laplace \cite{chen2017denoising} & Weighted Schatten-$\gamma$ \cite{xie2016weighted} & ETP \\ \hline
$\phi(x,\gamma)$ & log$(1+x^2)$ & $x^\gamma$ & log($\gamma$+$x$) & $1 - e^{-{\frac{x}{\gamma}}}$ & $ w x^\gamma$ & $\frac{\lambda (1 - \exp(-\gamma x) )}{1 - \exp(-\gamma)} $\\ \hline
\end{tabular}}
\end{table}
\end{center}

The $\ell_1$ norm and nuclear norm have been widely used to promote sparsity and low rankness respectively. 
However, the Lipschitz continuous differentiability of $p(\W)$ is assumed in our convergence analysis\footnote{We use a variant of the ADMM method to solve SLRQA. Note that such an assumption has been made for all the ADMM methods with global convergence guaranteed for nonconvex problems as far as we know, see, e.g., a review in~\cite{liu2019linearized}.}. 
Therefore, a smoothing function such as the Huber function can be used as a substitution for $p(\W)$. For more surrogate function of $\|\cdot\|_1$, we refer the readers to \cite{lu2014generalized}.
For the subsequent convergence analysis for Model \eqref{model:new2}, it is required that the term that promotes the low-rankness is also differentiable with respect to $\X$.  To satisfy this condition and ensure numerical stability, we also use a smoothed version of $\sum_i \phi(\sigma_i(\X),\gamma)$ by adding a small positive perturbation $\varepsilon$, i.e., replacing $\sum_i \phi(\sigma_i(\X),\gamma)$ with 
$\sum_i \phi(\sigma_i( \sqrt{ {\X}\strut^{*} \X + \varepsilon^2 \mathbf{I} }  ),\gamma),$ as discussed in \cite{zhang2023linear,zhang2019low}.
It is empirically shown in Section~\ref{6} that such a substitution yields satisfactory performance. This strategy has also been used in~\cite{zhang2023linear,zhang2019low} to guarantee convergence theoretically. 



When the observed color image is noise-free, or the variance $\tau^2$ in the noise $\boldsymbol{\varepsilon}$ is zero, Model~\eqref{model:new} is undefined due to the existence of a zero in the denominator of the data fitting term. In this case, one can reformulate the data fitting term as a constraint. 
It follows that a variant of SLRQA for the noise-free problem is given by
\be \label{model:new22}
\begin{aligned}
    \min_{\X,\W \in \mathbb{H}^{m \times n} } \quad &  \sum_i \phi(\sigma_i( \sqrt{ {\X}\strut^{*} \X + \varepsilon^2 \mathbf{I} }  ),\gamma) +\lambda p(\W) ,\\
    \st \quad & \mathcal{W}(\X)=\W, \quad \mathcal{P}_{\Omega}(\mathcal{A}(\X)) = \mathcal{P}_{\Omega}(\Y).
\end{aligned}
\ee
Since $\mathcal{W}$ is an orthogonal transform such as quaternion wavelet transform or quaternion discrete cosine transform (QDCT) \cite{QDCT}, we have  $\sum_i \phi(\sigma_i(\sqrt{ \X \strut^* \X + \varepsilon^2 \mathbf{I} }),\gamma) =\sum_i \phi(\sigma_i(\sqrt{ \W \strut^* \W + \varepsilon^2 \mathbf{I} }  ),\gamma)$, where $\Q_1,\Q_2$ are unitary matrix that defined later, the superscript $*$ denotes the conjugate transpose operator. 
Let $\Z = \mathcal{W}(\X)$. We propose a model named SLRQA-NF which is equivalent to~\eqref{model:new22} for noise-free problems:
\be \label{model:new2}
\begin{aligned}
    \min_{\Z,\W \in \mathbb{H}^{m \times n}} \quad & \sum_i \phi(\sigma_i( \sqrt{ {\Z}\strut^{*} \Z + \varepsilon^2 \mathbf{I} }  ),\gamma) +\lambda p(\W),\\
    \st \quad & \mathcal{P}_\Omega(\mathcal{A} \mathcal{W}^{\#}(\W))= \mathcal{P}_\Omega(\Y),\quad \Z=\W,
\end{aligned}
\ee
where the superscript $\#$ denotes the adjoint operator which is defined in Section \ref{2}. 
 \subsection{Related work} \label{1-2}
%
  The derivations
 of SLRQA and SLRQA-NF are motivated by sparse low rank inpainting (SLRI)~\cite{liang2012repairing} and Low-Rank Quaternion
Approximation (LRQA)~\cite{chen2019low} models. To be specific,  SLRI is a convex model that uses the sparsity and low-rankness properties while quaternion representation and nonconvex surrogate rank function have not been considered. LRQA is a quaternion-based model that makes use of the low-rankness property by replacing the nuclear norm with nonconvex surrogate functions, but it has not taken advantage of sparsity. SLRQA and SLRQA-NF take advantage of SLRI and LRQA 
by using quaternion representation, low-rankness, and sparsity. 
Next, we discuss some image processing models related to SLRQA and SLRQA-NF.

\textbf{Image denoising (noisy observation and identity linear operator $\mathcal{A}$):} 

If the low-rankness prior is omitted, i.e. $\phi \equiv 0$, the Lipschitz continuous gradients assumption is dropped, and $p(\W) = \|\W\|_1$, then  the classical analysis based approach~\cite{cai2016image} with an extension to quaternion representation is presented:
\be \label{model:denoising-analysis}
\begin{aligned}
    \min_{\X \in \mathbb{H}^{m \times n}} \quad & \lambda \|\mathcal{W}(\X)\|_1 + \frac{1}{2\tau^2}\|\X-\Y\|_{\mF}^2,
\end{aligned}
\ee
where the variable $\W$ is eliminated by $\W = \mathcal{W}(\X)$. 
If the sparse term $p(\W)$ is not used, the low-rankness prior is used and $\varepsilon$ is set to 0, 
then SLRQA becomes LRQA denoising model in~\cite[Equation~(20)]{chen2019low}
$$
\min_{\X \in \mathbb{H}^{m \times n}} \sum_i \phi(\sigma_i(\X),\gamma) + \frac{1}{2\tau^2}\|\X-\Y\|_{\mF}^2.
$$
As noted in \cite{chen2019low}, LRQA is a generalization of other models such as the WNNM \cite{gu2014weighted}. 

\textbf{Image deblurring (noisy observation and blurring kernel operator $\mathcal{A}$):} 

If the low-rankness prior is not used, i.e., $\phi \equiv 0$ and let $p(\W) = \|\W \|_1$, then the standard analysis based model for image deblurring in~\cite{cai2016image}  with extensions to quaternion representation is presented
$$
\min_{\X \in \mathbb{H}^{m \times n}} \lambda \|\mathcal{W}(\X)\|_1 + \frac{1}{2\tau^2}\| \mathcal{A}(\X)-\Y\|_{\mF}^2.
$$

\textbf{Image inpainting (partially observed images):} If the observed image is noise-free, the linear operator $\mathcal{A}$ is the identity, $p(\W) = \|\W\|_1$, $\phi$ is the nuclear norm and $\varepsilon =0$, then SLRQA-NF becomes SLRI in~\cite{liang2012repairing} with an extension to quaternion representation:
\be \label{model:inpainting-SLRI}
\begin{aligned}
    \min_{\W \in \mathbb{H}^{m \times n}} \quad & \|\W\|_* + \lambda \|\W\|_1, \quad    \st\,  \mathcal{P}_\Omega(\mathcal{W}^{\#}(\mathcal{\W}))= \mathcal{P}_\Omega(\Y).
\end{aligned}
\ee
Furthermore, if the sparsity term is dropped, a nonconvex surrogate $\phi$ is used and $\varepsilon$ is set to 0, then SLRQA-NF becomes LRQA~\cite[Equation~(22)]{chen2019low} given by
\be \label{model:inpainting-LRQA}
\begin{aligned}
    \min_{\X \in \mathbb{H}^{m \times n}} \quad & \sum_i \phi(\sigma_i(\X),\gamma), \quad
    \st \,  \mathcal{P}_\Omega(\mathcal{\mathcal{\X}})= \mathcal{P}_\Omega(\Y).
\end{aligned}
\ee

The above-reviewed models do not simultaneously use the low-rankness, sparsity, and quaternion representation. 
Recently, a quaternion based model combined with low-rankness and sparsity has been proposed in~\cite{han2022low}. 
Specifically, the model is given by: 
 \begin{equation} \label{model:other1}
 \begin{aligned}
  &\min_{\qL \in \mathbb{H}^{m \times r}, \D \in \mathbb{H}^{ r \times r}, \qR \in \mathbb{H}^{r \times n}, \X \in \mathbb{H}^{m \times n}, \W \in \mathbb{H}^{m \times n}}\,  \|\D\|_* + \lambda \|\W\|_1, \\
  \st\, & \qL^{*}\qL = \mathbf{I}_r ,\, \qR\qR^{*} = \mathbf{I}_r, ~~ \mathcal{W}(\X) = \W, ~~ \X = \qL\D\qR,~~ \mathcal{P}_{\Omega}(\Y - \qL\D\qR) = \mathbf{0}.
 \end{aligned}
 \end{equation}
 where $r$ is a prescribed integer  and the superscript $*$ denotes the conjugate transpose operator, see its definition in Section~\ref{2}.
Different from SLRQAs, Model~\eqref{model:other1} uses the nuclear norm to promote low-rank instead of nonconvex surrogate functions $\phi(\sigma_i(\X),\gamma)$. 
Furthermore, it needs to estimate the rank $r$ in the initialization phase while SLRQA does not. 
Another model that uses the three priors is proposed in~\cite{yang2022quaternion,yang2022quaternion2} for color image and video recovery. 
The model is given by
 \begin{equation} \label{model:other2}
\begin{aligned}
\min_{\X \in \mathbb{R}^{m \times n} } \, & \|\X\|_* - \sum_{i=r}^{\min (m,n)} \sigma_i(\X) + \lambda \|\W\|_1, \\
 \st \;\; & \mathcal{P}_{\Omega}(\X -\Y) = 0, \, \mathcal{W}(\X) = \W.
\end{aligned}
 \end{equation}
 Note that Model~\eqref{model:other2} also needs to estimate the rank $r$ and the low-rankness is promoted by penalizing a truncated nuclear norm, which is nonconvex. Moreover, convergence analysis of the proposed algorithms for solving the existing models in~\cite{han2022low,yang2022quaternion,yang2022quaternion2} with three priors are not given. 

 In this paper, proximal linearized ADMM algorithms, i.e., PL-ADMM and PL-ADMM-NF are proposed to solve Models~\eqref{model:new} and~\eqref{model:new2} and the global convergence of these algorithms is established.  
A generalization of SLRQA is formulated as
\be \label{prob:conve}
\begin{aligned}
\min_{\X,\W}\quad & f(\X) + g(\W) + \frac{1}{2\tau^2}\|\mathcal{P}_{\Omega}(\mathcal{A}(\X)) - \mathcal{P}_{\Omega}(\Y)\|_{\mathrm{F}}^2, \\
\st \quad &   \mathcal{C}(\X) + \mathcal{B}(\W) = \mathbf{\dot{B}},
\end{aligned}
\ee
where $f$ is semi-lowercontinuous, $g$ is continuously differentiable, $\mathcal{C},\mathcal{B} : \mathbb{H}^{m \times n} \rightarrow \mathbb{H}^{p \times q}$ are the constraint matrices and $\qb \in \mathbb{H}^{p \times q}$. Specifically, if $f(\X) = \sum_i \phi(\sigma_i( \sqrt{ {\X}\strut^{*} \X + \varepsilon^2 \mathbf{I} }  ),\gamma), g(\W) = p(\W)$,  $\mathcal{B} = - \mathcal{I}$, $\mathcal{C} = \mathcal{W}$, and $\qb = 0$, then Problem~\eqref{prob:conve} reduces to  the SLRQA in~\eqref{model:new}, where $ \mathcal{I}$ denotes the identity operator.
The convergence of nonconvex ADMM and its variants has been extensively studied in literatures and the global convergence of PL-ADMM is established.

However, SLRQA-NF does not satisfy a widely-used assumption in existing ADMM-type algorithms. Specifically,
SLRQA-NF can be formulated as
\be \label{prob:conve2}
\begin{aligned}
\min_{\X,\W}\quad & f(\X) + g(\W) , \\
\st \quad &   \mathcal{C}_1(\X) + \mathcal{B}_1(\W) = \mathbf{\dot{B}}_1 , \\
& \qquad \quad ~ \mathcal{B}_2(\W) = \mathbf{\dot{B}}_2,
\end{aligned}
\ee
where $f$ and $g$ are both continuously differentiable, $\mathcal{B}_1,\mathcal{C}_1: \mathbb{H}^{m \times n} \rightarrow \mathbb{H}^{p_1 \times q_1}$, $\mathcal{B}_2: \mathbb{R}^{m \times n} \rightarrow \mathbb{R}^{p_2 \times q_2}$ are given linear maps, $\qb_1 \in \mathbb{H}^{p_1 \times q_1}$, and $\qb_2 \in \mathbb{H}^{p_2 \times q_2}$. Specifically, if $f(\X) = \sum_i \phi(\sigma_i( \sqrt{ {\X}\strut^{*} \X + \varepsilon^2 \mathbf{I} }  ),\gamma),$ $ g(\W) = p(\W)$,  $\mathcal{B}_1 = - \mathcal{I}$,  $\mathcal{C}_1 = \mathcal{I}$, $\mathcal{B}_2 = \mathcal{P}_{\Omega}\mathcal{A} \mathcal{W}^{\#}$, $\qb_1 = 0$, and $\qb_2 = \mathcal{P}_{\Omega}(\Y)$,  
then Problem~\eqref{prob:conve2} reduces to the SLRQA-NF in~\eqref{model:new2}. Note that the linear constraints in Problem \eqref{prob:conve2} can be viewed as those in \eqref{prob:conve} if the two constraints are considered as one, that is $\mathcal{C} = \begin{pmatrix} \mathcal{C}_1 \\ 0 \end{pmatrix}$, $\mathcal{B} = \begin{pmatrix} \mathcal{B}_1 \\ \mathcal{B}_2 \end{pmatrix}$, and $\mathbf{\dot{B}} = \begin{pmatrix} \mathbf{\dot{B}}_1 \\ \mathbf{\dot{B}}_2 \end{pmatrix}$. In the convergence analysis of existing ADMM-type algorithms for problems with constraints as in~\eqref{prob:conve}, the assumption 
\begin{equation} \label{eqn:03}
\text{range}(\mathcal{C}) \subseteq \text{range}(\mathcal{B}) \hbox{ and } \mathbf{\dot{B}} \in \text{range}(\mathcal{B}),
\end{equation}
which we refer to as ``range assumption", are required \cite{wang2019global,yashtini2022convergence,liu2019linearized,hong2016convergence,li2015global}. The authors in \cite{wang2019global} show the tightness of the ``range assumption'' in the sense that there exists a problem such that the ``range assumption'' does not hold while the ADMM algorithm for this problem diverges. This tightness does not prevent a weaker version of the ``range assumption'' that guarantees convergence of ADMM. Since the constraints in~\eqref{prob:conve2} do not satisfy the ``range assumption'', we propose an alternative weaker assumption that guarantees global convergence of an ADMM-type algorithm.

\subsection{Contribution}

 The main contributions of this paper are listed as follows:


\begin{itemize} \itemsep=0.3cm
\item 
We propose a novel model called SLRQA to address image processing problems with noise by maximizing the a posteriori probability estimate. Unlike the existing image processing models, the proposed SLRQA \eqref{model:new} is a quaternion model combining low rank and sparsity, providing flexibility to choose surrogate functions and not needing an initial rank estimate.  We propose a PL-ADMM algorithm to solve SLRQA and its global convergence is guaranteed.


\item  When the observations are noise-free, we propose a model named SLRQA-NF~\eqref{model:new2} and a corresponding algorithm named PL-ADMM-NF. The global convergence of PL-ADMM-NF is established under a newly proposed assumption. To the best of our knowledge, this is the first nonconvex ADMM-type algorithm without the ``range assumption'' in~\eqref{eqn:03} and still guarantees global convergence.


\item In the numerical experiments, we verify that the quaternion representation outperforms the RGB monochromatic representation. Moreover, compared with some representative state-of-the-art image processing models, the promising performance of SLRQA and SLRQA-NF demonstrates their superiority and robustness numerically and visually.
\end{itemize}

\subsection{Organization}
The rest of this paper is organized as follows. In Section~\ref{2}, the preliminaries and notations are introduced. The PL-ADMM algorithm for solving SLRQA and PL-ADMM-NF algorithm for solving SLRQA-NF are presented in Section~\ref{3}.   The global convergence of these two algorithms are established in~Section \ref{4}. In Section \ref{5}, extensive numerical experiments are conducted to demonstrate the performance of the models and algorithms by comparing them with other state-of-the-art methods. Finally, the conclusion is given in Section \ref{6}.

\section{Preliminaries} \label{2}

Throughout this paper,  scalars, real vectors, real matrices, and real tensors are respectively denoted by lowercase letters $x$, boldface lowercase letters $\mathbf{x}$, boldface capital letters $ \mathbf{X}$, and boldface capital letters with a tilde $\widetilde{\mathbf{X}}$ respectively.  $\qH$ denotes the quaternion algebra proposed in~\cite{hamilton1866elements}.  A dot above a variable $\dot{q}$ represents a quaternion number.  A quaternion number consists of one real part and three imaginary parts:
$
\dot{q}=x_s+x_i\qi+x_j \qj+x_k\qk.
$
where $x_s,x_i,x_j,x_k\in \R$, and $\qi,\qj,\qk$ are three imaginary units.
 The quaternion $\dot{q}$ is called a pure quaternion if and only if it has a zero real part, $\ie x_s=0$. The real part of $\dot{q}$ is $\text{Re}(\dot{q}) = x_s$. The conjugate  $\overline{\dot{q}}$ is defined as $\overline{\dot{q}}= x_s-x_i\qi-x_j\qj-x_k\qk$. The modulus $|\dot{q}|$ is defined as $|\dot{q}|=\sqrt{x_s^2+x_i^2+x_j^2+x_k^2}$. The inverse of $\dot{q}$ is given by $\dot{q}^{-1} = \dfrac{\overline{\dot{q}}}{|\dot{q}|^2}$. 
For a quaternion matrix $\X,$  $\X_{ij}$ denotes the $i$-th row, $j$-th coloum element of $\X_{ij}$ and $\X_{ijk}$ to denote the  $i$-th row, $j$-th and $k$-th imaginary units.
The Frobenius norm of $\X$ is defined as 
  $\|\X\|_{\mF} := \sqrt{\sum_{i=1}^{n}\sum_{j=1}^{m} |\X_{ij}|^2}$.  Accordingly, the conjugate operator $\overline{\A}$ , transpose operator $\A^\mT$ and the conjugate transpose operator $\A^* $ are defined as $\overline{\A}=(\overline{\dot{a}})_{ij} , \A^\mT=(\dot{a})_{ji} $ and $\A^*=(\overline{\dot{a}})_{ji}$, respectively. Furthermore, the multiplication rule of quaternion is given by:
 $\qi \cdot 1= \qi, \quad \qj \cdot 1= \qj, \quad \qk \cdot 1= \qk. \quad$
 $\qi^2 = \qj^2 =\qk^2=-1 ,\qi \qj=-\qj \qi=\qk,\,\qj \qk = -\qk \qj=\qi,\,\qk \qi=-\qi \qk=\qj$.
Note that the multiplication of two quaternions $\dot{q}_1$ and $\dot{q}_2$ is not commutative, i.e., $\dot{q}_1\dot{q}_2 \neq \dot{q}_2\dot{q}_1$ in general. $\Q$ is called a unitary matrix if and only if  $\Q\Q^* = \Q^*\Q = \mathbf{I},$ where $\mathbf{I}$ denotes a real identity matrix. The inner product of two quaternion matrices $\qG_1 , \qG_2 \in \mathbb{H}^{m \times n}$ is defined as $ \iprod{ \qG_1 }{\qG_2 } := \text{tr}(\qG_1^{*} \qG_2)$. The image of a given matrix $\qG_1$ is denoted by ${\rm range}(\qG_1)$.
Given an operator $\mathcal{W} : \mathbb{H}^{m \times n} \rightarrow \mathbb{H}^{m \times n}$,  the weighted norm is defined as $\|\X\|_{\mathcal{W}} := \iprod{\mathcal{W}(\X) }{\X}$. 
The operator norm is defined as $\|\mathcal{W}\| := \sup_{\X} \frac{\|\mathcal{W}( \X) \|_{\mathrm{F}} }{\|\X\|_{\mathrm{F}}}$. Let $\mathcal{W}^{\#}$ denote the adjoint operator of $\mathcal{W}$, i.e., $\iprod{\mathcal{W}(\qG_1)}{\qG_2} = \iprod{\qG_1}{\mathcal{W}^{\#}(\qG_2)}$ holds for any $\qG_1, \qG_2 \in \mathbb{H}^{m \times n}$. 
If $\mathcal{W}$ is given by a quaternion matrix, then $\mathcal{W}^{\#}$ is the conjugate transpose of the matrix.
For $\widetilde{\qG} := (\qG_1,\cdots,\qG_p) \in \mathbb{H}^{n_1 \times m_1} \times \cdots \mathbb{H}^{n_p \times m_p}, $
$|||\widetilde{\qG}\||_{\mathrm{F}}:= \sqrt{\sum_{i=1}^{p}\|\qG_i\|_{\mathrm{F}}^2 }$. Hence we have $\sum_{i=1}^p \frac{1}{\sqrt{p}} \| \qG_i \|_{\mathrm{F}}  \le \||\widetilde{\qG}\||_{\mathrm{F}} \le \sum_{i=1}^p \|\qG_i\|_{\mathrm{F}} ,$ where the first equality follows from the Cauchy-Schwarz inequality.

The quaternion singular value decomposition theorem will be used in the solving process of our algorithm. We state it as follows. 
\begin{theorem}[Quaternion singular value decomposition (QSVD)\cite{zhang1997quaternions}]
For any given quaternion matrix  $\A \in \qH^{m \times n},$ there exist two unitary quaternion matrices $\qU \in \qH ^{m \times m},\qV \in \qH^{n \times n}$ such that
$$
\A=\qU \left( \begin{matrix}
	\mathbf{\Sigma }&		\mathbf{0}\\
	\mathbf{0}&		\mathbf{0}\\
\end{matrix} \right)
 \qV^*,
$$
where $\mathbf{\Sigma} \in \mathbb{R}^{r \times r}$ is a diagonal matrix with $\mathbf{\Sigma}_{1,1} \geq \mathbf{\Sigma}_{2,2} \geq \cdots \geq  \mathbf{\Sigma}_{r,r} >0$ being the singular values of $\A.$ We denote $\sigma_i(\A)$ to be the $i$-th singular value of $\A$ and
$\boldsymbol{\sigma}_{\A} = (\mathbf{\Sigma}_{1,1},\cdots,\mathbf{\Sigma}_{r,r}, 0 ,\cdots,0) \in \mathbb{R}^{l}$ as the singular value vector	of $\A$, where $l= \min\{m,n\}$.
\end{theorem}
According to the QSVD, the rank of a quaternion matrix is defined as the number of nonzero singular values while the nuclear norm $\|\A\|_*$ is defined as the sum
of all nonzero singular values.

\section{Algorithm framework} \label{3}
In this section, two algorithms, i.e. PL-ADMM and PL-ADMM-NF, are proposed to solve Model \eqref{model:new}  and \eqref{model:new2} respectively. The details of PL-ADMM are clarified in Section \ref{3-1} and the PL-ADMM-NF algorithm is presented in Section \ref{3-2}.
\subsection{An Algorithm for SLRQA} \label{3-1}

The proposed proximal linearized alternating direction method of multipliers (PL-ADMM) is stated in Algorithm~\ref{algorithm:denoise}.

\begin{algorithm}[h]
\caption{Proixmal linearized-ADMM (PL-ADMM) for SLRQA in~\eqref{model:new}}
\label{algorithm:denoise}
\begin{algorithmic}[1]
\Require A noisy observed image $\Y$, an initial iterate $\X_0 := \Y$, two parameters $\mu \in (0,2)$ and $\beta > 0$, a tolerance $\eta > 0$, two sequences of regularization parameters $\{L_{k,1}\}$ and $\{L_{k,2}\}$ satisfying $\sup_{k \ge 0} { L_{i,k} } < \infty$ and $\inf_{k>0} {L_{i,k}} > 0$ for $i = 1, 2$.
\Ensure A recovered image $\X_*$.
\State Set $\W_0 = \boldsymbol{0}$ and $\La_0 = \boldsymbol{0}$;
\For{$k = 0,1,2,\ldots$}
    \State \label{denoise:X} $\X$-subproblem: 
    \begin{align*}
       \X_{k+1} \in &\argmin_{\X} \; \sum_i \phi(\sigma_i( \sqrt{ {\X}\strut^{*} \X + \varepsilon^2 \mathbf{I} }  ),\gamma)  \notag \\
       &+ \frac{1}{2\tau^2}\|\mathcal{P}_{\Omega} (\Y) - \mathcal{P}_{\Omega} (\mathcal{A}(\X)) \|_{\mF}^2 + \frac{\beta}{2}\|\mathcal{W}(\X) - \W_k + \frac{\La_k}{\beta}\|_{\mF}^2 + \frac{L_{k,1}}{2}\|\X - \X_k \|_{\mF}^2;
    \end{align*}
    \State \label{denoise:W} $\W$-subproblem:
    \begin{align*}
         \W_{k+1} = \argmin_\W& \; \lambda \iprod{\nabla p(\W_k)}{\W} + \frac{\beta}{2}\|\W - (\mathcal{W}(\X_{k+1}) + \frac{\La_k}{\beta})\|_{\mF}^2 + \frac{L_{k,2}}{2}\|\W - \W_k \|_{\mF}^2; 
     \end{align*}
    \State \label{denoise:La} Update Lag{\rm range} multipliers:
    $
    \La_{k+1} = \La_{k} + \mu \beta (\mathcal{W}(\X_{k+1}) - \W_{k+1}); 
    $
    \If{$\|\X_{k+1} - \X_{k}\|_{\mF} + \|\W_{k+1} - \W_{k}\|_{\mF} + \|\mathcal{W}(\X_{k+1}) - \W_{k+1}\|_{\mF} < \eta$}
        \State Return $\X_* = \X_{k+1}$;
    \EndIf
\EndFor
\end{algorithmic}
\end{algorithm}

Each iteration of PL-ADMM minimizes the augmented Lagrangian function by using the alternating direction method. Specifically, the augmented Lagrangian function of SLRQA in~\eqref{model:new} is given by
\bee
\begin{aligned}
\mathcal{L}_\beta(\X,\W,\La) =\sum_i \phi(\sigma_i( \sqrt{ {\X}\strut^{*} \X + \varepsilon^2 \mathbf{I} }  ),\gamma)+&\lambda p(\W)+ \frac{1}{2\tau^2}\| \mathcal{P}_{\Omega}(\Y) -\mathcal{P}_{\Omega}(\mathcal{A}(\X)) \|_{\mF}^2 \\
&+ \iprod{\La}{\mathcal{W}(\X)-\W} +\frac{\beta}{2}\|
\mathcal{W}(\X)-\W\|^2_\mF,
\end{aligned}
\eee
where $\La \in \mathbb{H}^{m \times n}$ denotes the Lag{\rm range} multipliers and $\beta > 0$.

In Step~\ref{denoise:X}, $\X_{k+1}$ is obtained by minimizing $\mathcal{L}_{\beta}(\X,\W,\La) + \frac{L_{k,1}}{2}\|\X-\X_k\|_{\mathrm{F}}^2$ for a fixed $\W_k$ and $\La_k$.
This subproblem does not yield a closed-form solution in general and therefore may be solved by nonsmooth algorithms such as a subgradient method or the linearized Bregman algorithm in~\cite{cai2010singular}. 
If the linear operator $\mathcal{A}$ in~\eqref{model:new} is the identity and all pixels of the noisy image $\Y$ are observed, i.e., $\Omega$ indicates the whole image, then the $\X$-subproblem can be solved efficiently. 
Specifically, the subproblem can be reformulated as
\begin{equation}
\begin{aligned}
    \X_{k+1}   
    =& \argmin_{\X} \sum_i \phi(\sigma_i( \sqrt{ {\X}\strut^{*} \X + \varepsilon^2 \mathbf{I} }  ),\gamma) + \\
    & \frac{\tau^2(\beta + L_{k,1}) + 1}{2\tau^2}\left\|\X -\frac{\tau^2}{\tau^2(\beta + L_{k,1}) +1}\left(\frac{1}{\tau^2}\Y + \beta \mathcal{W}^{\#}(\W_k-\frac{\La_k}{\beta})+L_{k,1} \X_k \right)\right\|_{\mF}^2, 
\end{aligned}
\end{equation}
If $\phi(\sigma_i(\X),\gamma) = \sigma_i(\X),$ the $\X$-subproblem  has a closed-form solution by using the soft thresholding method on the singular values~\cite{cai2010singular}. 
When a nonconvex surrogate function $\phi(\sigma_i(\X),\gamma)$ is used,
the subproblem of $\X$ can be solved via a convex-concave procedure (CCP) in~\cite{lanckriet2009convergence}. For completeness, we summarize CCP in Algorithm~\ref{algorithmf:X}.  Lemma~\ref{thm3} is crucial since it transforms the matrix optimization problem of $\X$ into an optimization problem of singular values. We note that the case of $\varepsilon =0$ has been proven in \cite[Theorem 3]{chen2019low}. We extend the theorem into the case that $\varepsilon \ge 0$.

\begin{lemma} \label{thm3}
Let $\mathbf{\hat{\X}}=\qU \boldsymbol{\Sigma} \qV^*  \in \qH^{m \times n}$ be a singular value decomposition of $\mathbf{\hat{\X}}$. The solution of
\be \label{subproblem:X}
\argmin_{\X\in \qH^{m \times n} } \, \sum_i \phi(\sigma_i( \sqrt{ {\X}\strut^{*} \X + \varepsilon^2 \mathbf{I} }  ),\gamma)+\frac{\mu}{2}\|\mathbf{\hat{\X}}-\X\|_{\mF}^2
\ee
can be represented by $\X_*=\qU \boldsymbol{\Sigma_*} \qV^*,$ where $\boldsymbol{(\Sigma_*)}_{i,i}=(\boldsymbol{\sigma_*})_i,$ for $i = 1,\cdots ,\min\{m,n\}$ where $\boldsymbol{\sigma_*}$ is obtained by
\be \label{scheme:new}
\boldsymbol{\sigma_*}=\argmin_{\boldsymbol{\sigma}\, \geq 0} \varphi(\boldsymbol{\sigma},\varepsilon ;\gamma) +\frac{\mu}{2}\|\boldsymbol{\sigma}-\boldsymbol{\sigma_{\boldsymbol{\hat{\X}}}}\|_2^2,
\ee
where $\varphi(\boldsymbol{\sigma},\varepsilon ;\gamma) := \sum_i \phi( \sqrt{\boldsymbol{\sigma}_i^2  + \varepsilon^2} ,\gamma) $.
\end{lemma}
\begin{proof}
According to the definition of QSVD, the following inequality holds.
\begin{equation} \label{eqn:proof}
\frac{\mu}{2}\|\mathbf{\hat{\X}}-\X\|_{\mF}^2  + \sum_i \phi(\sigma_i( \sqrt{ {\X}\strut^{*} \X + \varepsilon^2 \mathbf{I} }  ),\gamma) 
= \frac{\mu}{2}\|  \boldsymbol{\Sigma} -\qU^*\X \qV \|_{\mF}^2  +  \sum_i \phi(\sigma_i( \sqrt{ {\X}\strut^{*} \X + \varepsilon^2 \mathbf{I} }  ),\gamma), 
\end{equation}
where the equality is based on the fact that the Frobenius norm is unitarily variant \cite[Equation (37)]{chen2019low}. Suppose \( \A = \qU^*\X \qV \),  thus \eqref{eqn:proof} is equal to
\begin{equation}
\begin{aligned}
&= \frac{\mu}{2}\|  \boldsymbol{\Sigma} -\qU^*\X \qV \|_{\mF}^2  +  \sum_i \phi(\sigma_i( \sqrt{ {\A}\strut^{*} \A + \varepsilon^2 \mathbf{I} }  ),\gamma)   \\
&\geq \frac{\mu}{2} \left\| \boldsymbol{\Sigma} - \boldsymbol{\Sigma}_{\A} \right\|_{\mathrm{F}}^2 +  \sum_i \phi(\sigma_i( \sqrt{ {\A}\strut^{*} \A + \varepsilon^2 \mathbf{I} }  ),\gamma) \\
&= \frac{\mu}{2} \left\| \boldsymbol{\Sigma} - \boldsymbol{\Sigma}_{\X} \right\|_{\mathrm{F}}^2 +  \sum_i \phi(\sigma_i( \sqrt{ {\X}\strut^{*} \X + \varepsilon^2 \mathbf{I} }  ),\gamma) \\
&= \frac{\mu}{2} \left\| \boldsymbol{\sigma} - \boldsymbol{\sigma}(\X) \right\|_2^2 +  \varphi(\boldsymbol{\sigma}(\X),\varepsilon ;\gamma) \\
&\geq \frac{\mu}{2} \left\| \boldsymbol{\sigma} - \boldsymbol{\sigma}_* \right\|_2^2 +  \varphi(\boldsymbol{\sigma}_*,\varepsilon ;\gamma), 
\end{aligned}
\end{equation}
where \( \boldsymbol{\sigma}(\X) \)  is the singular value vectors of \( \X \), the first inequality follows from Hoffman-Wielandt inequality.  According to the definition of $\X_*$,  $\sum_i \phi(\sigma_i( \sqrt{ {\X_*}\strut^{*} \X_* + \varepsilon^2 \mathbf{I} }  ),\gamma)+\frac{\mu}{2}\|\mathbf{\hat{\X}}-\X_*\|_{\mF}^2 = \frac{\mu}{2} \left\| \boldsymbol{\sigma} - \boldsymbol{\sigma}_* \right\|_2^2 +  \varphi(\boldsymbol{\sigma}_*,\varepsilon ;\gamma) $. The proof is completed.
\end{proof}

 In the $k$-th iteration of Algorithm~\ref{algorithmf:X}, the function $\varphi(\boldsymbol{\sigma},\varepsilon;\gamma)$ is approximated by a linearized function $\varphi(\boldsymbol{\sigma}_k,\varepsilon;\gamma) +\nabla_{\boldsymbol{\sigma}} \varphi(\boldsymbol{\sigma}_k,\varepsilon;\gamma )^\mT\boldsymbol{\sigma} $. 
It follows that the update of $\boldsymbol{\sigma}$ in Problem~\eqref{scheme:new} is given by
\be \label{scheme:upsig}
\boldsymbol{\sigma}_{k+1}=\max\left\{\boldsymbol{\sigma}_{\boldsymbol{\hat{\X}}}-\frac{\nabla_{\boldsymbol{\sigma}} \varphi(\boldsymbol{\sigma}_k,\varepsilon;\gamma)}{\mu},\boldsymbol{0}\right\}.
\ee
It has been shown in~\cite[Theorem 4]{lanckriet2009convergence} that any limit points of $\{\boldsymbol{\sigma}_k \}_{k=1}^\infty$ are stationary points of \eqref{scheme:new}.

\begin{algorithm}[h]
\caption{CCP to solve Problem~\eqref{subproblem:X}}
\label{algorithmf:X}
\begin{algorithmic}[1]
  \Require A quaternion matrix $\boldsymbol{\hat{\X}} \in \mathbb{H}^{m \times n}$, a nonconvex function $\phi$, a constant $\mu > 0$, a tolerance $\eta_C$, and an initial iterate $\boldsymbol{\sigma}_0 = \mathbf{0}$.
  \Ensure $\boldsymbol{\X}_*$.
  \State Compute QSVD $\mathbf{\hat{\X}}=\qU \boldsymbol{\Sigma} \qV^*$ and let $\boldsymbol{\sigma_{\hat{\X}}}$ satisfy $\boldsymbol{\Sigma}=\diag(\boldsymbol{\sigma_{\hat{\X}}})$;
  \For{$k = 1,2,\ldots$}
    \State Compute $\boldsymbol{\sigma}_{k}$ by~\eqref{scheme:upsig};
    \If{$\|\boldsymbol{\sigma}_{k-1} - \boldsymbol{\sigma}_{k} \|_2 < \eta_C$}
      \State Set $\boldsymbol{\sigma}_* = \boldsymbol{\sigma}_k$ and return $\X_* = \qU \boldsymbol{\Sigma_*} \qV^*$, where $\boldsymbol{\Sigma_*} = \diag(\boldsymbol{\sigma_*})$;
    \EndIf
  \EndFor
\end{algorithmic}
\end{algorithm}

In Step~\ref{denoise:W},  the $\W$-subproblem approximately minimizes $\mathcal{L}_{\beta}(\X,\W,\La) + \frac{L_{k,2}}{2}\|\W-\W_k\|_{\mathrm{F}}^2$ for a fixed $\X_{k+1}$ and $\La_k$. Specifically, the objective function of the $\W$-subproblem is obtained by using a linearized strategy for $p(\W)$ in~$\mathcal{L}_{\beta}(\X,\W,\La) + \frac{L_{k,2}}{2}\|\W-\W_k\|_{\mathrm{F}}^2$. It follows that $\W_{k+1}$ has a closed form solution
\be
\W_{k+1}= \frac{1}{\beta + L_{k,2}}\left( \beta \mathcal{W}(\X_{k+1}) + \La_k +L_{k,2}\W_k \right) - \frac{\lambda}{\beta + L_{k,2} } \nabla p(\W_k).
\ee

Step \ref{denoise:La} is used to update the Lag{\rm range} multiplier. The difference between the classic ADMM for convex function and PL-ADMM is that $\mu \in (0,2)$ in PL-ADMM rather than $\mu \in (0, \frac{\sqrt{5}+1}{2})$ \cite{chen2017note}, which means that a larger stepsize can be chosen with convergence guarantee.

\subsection{An Algorithm for SLRQA-NF} \label{3-2}


The proposed algorithm PL-ADMM-NF for solving SLRQA-NF is stated in Algorithm \ref{algorithm:inpaint}. 
\begin{algorithm}[h]
\caption{PL-ADMM-NF algorithm for SLRQA-NF in \eqref{model:new2}}
\label{algorithm:inpaint}
\begin{algorithmic}[1]
    \Require Corrupted images $\Y$, two parameters $\mu \in (0,2)$ and $\beta > 0$, a tolerance $\eta > 0$, a sequence of regularization parameters $\{L_{k,1}\}, \{L_{k,2}\}$ satisfying $\sup_{k \ge 0} { L_{k,1} } < \infty$, $\sup_{k \ge 0} { L_{k,2} } < \infty$ and $\inf_{k>0} {L_{k,1}} > 0$, $\inf_{k>0} {L_{k,2}} > 0$.
    \Ensure Recovered $\Z_*$.
    \State Set $\Z_0 = \boldsymbol{0}$, $\W_0 =  \boldsymbol{0}$, $\La_{1,0} = \boldsymbol{0}$, $\La_{2,0} = \boldsymbol{0}$.
    \For{$k = 0,1,2,\ldots$}
        \State $\Z$ subproblem:
        $$
        \Z_{k+1} \in \argmin_\Z \sum_i \phi(\sigma_i( \sqrt{ {\Z}\strut^{*} \Z + \varepsilon^2 \mathbf{I} }  ),\gamma) + \iprod{\La_{k,1}}{\Z - \W_k} + \frac{\beta_1}{2} \|\Z - \W_k\|_{\mF}^2 + \frac{L_{k,1}}{2} \|\Z - \Z_k\|_{\mF}^2;
        $$ \label{inpaint:subX}
        \State $\W$ subproblem:
        $$
        \begin{aligned}
        \W_{k+1} &= \argmin_\W \lambda \iprod{\nabla p(\W_k)}{\W - \W_k} + \frac{\beta_1}{2} \|\W - \Z_{k+1} + \frac{1}{\beta_1} \La_{k,1} \|_{\mathrm{F}}^2 \\
        & + \frac{\beta_2}{2} \left\| (\mathcal{P}_\Omega(\mathcal{A} \mathcal{W}^{\#}(\W ) - \Y) + \frac{1}{\beta_2} \La_{k,2} \right\|_{\mathrm{F}}^2  + \frac{L_{k,2}}{2} \|\W - \W_k\|_{\mathrm{F}}^2;
            \end{aligned}
        $$
 \label{inpaint:subW}
        \State $\La_{k+1,1} = \La_{k,1} + \mu\beta  (\Z_{k+1} - \W_{k+1})$; \label{inpaint:subY1}
        \State $\La_{k+1,2} = \La_{k,2} + \mu\beta  \mathcal{P}_\Omega(\mathcal{A} \mathcal{W}^{\#}(\W_{k+1}) - \Y)$; \label{inpaint:subY2}
        \If{$\|\Z_{k+1} - \W_{k+1}\|_{\mF} + \|\mathcal{P}_\Omega(\mathcal{A} \mathcal{W}^{\#}(\W_{k+1})) - \mathcal{P}_\Omega(\Y)\|_{\mF} < \eta$}
            \State Return $\Z_* = \Z_{k+1}$; \textbf{break}
        \EndIf
    \EndFor
\end{algorithmic}
\end{algorithm}
The augmented Lagrangian function of SLRQA-NF \eqref{model:new2} is given by:
\be
\begin{aligned}
& \tilde{\mathcal{L}}_{\beta_1,\beta_2} (\Z,\W,\La_1,\La_2)= \sum_i \phi(\sigma_i( \sqrt{ {\Z}\strut^{*} \Z + \varepsilon^2 \mathbf{I} }  ),\gamma) +\lambda p(\W)+\iprod{\La_1}{\Z-\W}\\
& +\frac{\beta_1}{2}\|
\Z-\W\|^2_\mF+\iprod{\La_2}{\mathcal{P}_\Omega( \mathcal{A} \mathcal{W}^{\#}(\W))-\mathcal{P}_\Omega(\Y) }+\frac{\beta_2}{2}\|\mathcal{P}_\Omega( \mathcal{A} \mathcal{W}^{\#}(\W))-\mathcal{P}_\Omega(\Y)\|_{\mF}^2 \\
& =  \sum_i \phi(\sigma_i( \sqrt{ {\Z}\strut^{*} \Z + \varepsilon^2 \mathbf{I} }  ),\gamma) +\lambda p(\W)+\frac{\beta_1}{2} \| \Z -\W + \frac{1}{\beta_1}\La_1 \|_{\mathrm{F}}^2 \\
&+ \frac{\beta_2}{2}\|\mathcal{P}_\Omega( \mathcal{A} \mathcal{W}^{\#}(\W))-\mathcal{P}_\Omega(\Y) +\frac{1}{\beta_2} \La_2 \|_{\mathrm{F}}^2 - \frac{1}{2\beta_1}\|\La_1\|_{\mathrm{F}}^2 - \frac{1}{2\beta_2}\|\La_2\|^2_{\mathrm{F}},
\end{aligned}
\ee
where $\beta_1,\beta_2 >0$ are the penalty parameters, $\La_1,\La_2 \in \qH^{m \times n}$ are the Lag{\rm range} multipliers.

In Step 3 of Algorithm \ref{algorithm:inpaint}, 
since the subproblem of $\Z$ is equivalent to
$$
\min_{\Z}\, \sum_i \phi(\sigma_i( \sqrt{ {\Z}\strut^{*} \Z + \varepsilon^2 \mathbf{I} }  ),\gamma) +  \frac{\beta+L_{k,1}}{2}\left\|\Z-\frac{1}{\beta+L_{k,1}} \left(\beta\W_k+\La_{k,1} + L_{k,1}\Z_k \right) \right\|_{\mathrm{F}}^2,
$$
 Algorithm \ref{algorithmf:X} can be used to obtain $\Z_{k+1}$.  
In Step \ref{inpaint:subW} of Algorithm \ref{algorithm:inpaint}, the solution of $\W$ has an explicit solution: 

$$
\W_{k+1} = \mathcal{G} ^{-1}\left(   \beta_1 \Z_{k+1} - \La_{k,1}+ \mathcal{W} \mathcal{A}^{\#} \mathcal{P}_{\Omega}^{\#}\left(\beta_2 \mathcal{P}_{\Omega}(\Y) - \La_{k,2}   \right) + L_{k,2} \W_k- \lambda \nabla p(\W_k)  \right),
$$
where $\mathcal{G} = (\beta_1 + L_{k,2})\mathcal{I} + \beta_2 \mathcal{W}\mathcal{A}^{\#}\mathcal{P}_{\Omega}^{\#} \mathcal{P}_{\Omega} \mathcal{A} \mathcal{W}^{\#}$.
Steps \ref{inpaint:subY1} and \ref{inpaint:subY2} are used to update the Lag{\rm range} multipliers.  Finally, the recovered image $\X_* $ is obtained by $\X_* = \mathcal{W}^{\#}(\Z_*)$. 


\section{Convergence analysis} \label{4}

In this section, we present the global convergence analysis for Algorithm \ref{algorithm:denoise} and \ref{algorithm:inpaint} respectively. The convergence analysis for \ref{algorithm:denoise} is similar to that in \cite{yashtini2022convergence}. However, the algorithm we analysis is slightly different from \cite{yashtini2022convergence}. Furthermore, we present more flexible assumptions. The convergence analysis for Algorithm \ref{algorithm:inpaint} is different from the previous works since two linear constraints are considered. To the best of our knowledge, this is the first convergence analysis for ADMM with two linear constrants.
 \subsection{Global Convergence for Algorithm \ref{algorithm:denoise}}
The global convergence is established for a generalization of SLRQA in~\eqref{model:new}:
\be \label{prob:conve}
\begin{aligned}
\min_{\X,\W}\quad & f(\X) + g(\W) + \frac{1}{2\tau^2}\|\mathcal{P}_{\Omega}(\mathcal{A}(\X)) - \mathcal{P}_{\Omega}(\Y)\|_{\mathrm{F}}^2, \\
\st \quad &   \mathcal{C}(\X) + \mathcal{B}(\W) = \mathbf{\dot{B}},
\end{aligned}
\ee
where $f$ is semi-lowercontinuous and $g$ is continuously differentiable, $\mathcal{C},\mathcal{B} : \mathbb{H}^{m \times n} \rightarrow \mathbb{H}^{m \times n}$ are given linear maps and $\qb \in \mathbb{H}^{m \times n}$. Specifically, if $f(\X) = \sum_i \phi(\sigma_i(\X),\gamma), g(\W) = p(\W)$,  $\mathcal{B} = - \mathbf{I}$, $\mathcal{C} = \mathcal{W}$, and $\qb = 0$, then Problem~\eqref{prob:conve} becomes the SLRQA in~\eqref{model:new}.
The augmented Lagrangian function of \eqref{prob:conve} is
\be \label{Lag:general}
\begin{aligned}
& \mathcal{L}_{\beta}(\X,\W,\La)  =  f(\X) + g(\W) + \frac{1}{2\tau^2}\|\mathcal{P}_{\Omega}(\mathcal{A}(\X)) - \mathcal{P}_{\Omega}(\Y)\|_{\mathrm{F}}^2 + \iprod{\La}{\mathcal{C}(\X) + \mathcal{B}(\W) - \mathbf{\dot{B}} } \\
& + \frac{\beta}{2} \|\mathcal{C}(\X) + \mathcal{B}(\W) - \mathbf{\dot{B}} \|^2_{\mathrm{F}}\\
&= f(\X) + g(\W) + \frac{1}{2\tau^2}\|\mathcal{P}_{\Omega}(\mathcal{A}(\X)) - \mathcal{P}_{\Omega}(\Y)\|_{\mathrm{F}}^2 + \frac{\beta}{2}\left\|\mathcal{C}(\X) + \mathcal{B}(\W) - \mathbf{\dot{B}} + \frac{\La}{\beta} \right\|_\mF^2 - \frac{1}{2\beta}\left\|\La\right\|_\mF^2 .\\
\end{aligned}
\ee
The iteration scheme for the proposed PL-ADMM for~\eqref{prob:conve} is
\be \label{scheme2}
\begin{aligned}
\X_{k+1} & \in \argmin_{\X} f(\X) + \frac{1}{2\tau^2}\|\mathcal{P}_{\Omega}(\mathcal{A}(\X)) - \mathcal{P}_{\Omega}(\Y)\|_{\mathrm{F}}^2 \\
& + \frac{\beta}{2}\left\|\mathcal{C}(\X) + \mathcal{B}(\W_k) - \mathbf{\dot{B}} + \frac{\La_k}{\beta} \right\|_\mF^2 + \frac{L_{k,1}}{2}\|\X-\X_{k} \|_{\mathrm{F}}^2 ,\\
\W_{k+1} & \in \argmin_{\W} \iprod{\nabla g(\W_k)}{\W} + \frac{\beta}{2}\left\|\mathcal{C}(\X_{k+1}) + \mathcal{B}(\W) - \mathbf{\dot{B}} + \frac{\La_k}{\beta} \right\|_\mF^2  + \frac{L_{k,2}}{2}\left\|\W-\W_k\right\|^2_{\mathrm{F}},\\
\La_{k+1} & = \La_k + \mu \beta \cdot (\mathcal{C}(\X_{k+1}) + \mathcal{B}(\W_{k+1}) - \mathbf{\dot{B}} ),
\end{aligned}
\ee
which is a generalization of Algorithm~\ref{algorithm:denoise}.

The choices of $\beta, L_{k,1}$ and $L_{k,2}$ are crucial for the convergence of \eqref{scheme2}. Generally speaking, $\beta$ should be chosen large enough, and $L_{k,1}$, $L_{k,2}$ should be chosen in a suitable {\rm range}.
In order to give the conditions of $\beta$ and $L_{k,1}, L_{k,2}$, 
we introduce the following notations.
Let $q_i := \sup_{k \ge 0} { L_{i,k} } < \infty$ and $q_i^{-}:= \inf_{k>0} {L_{i,k}} > 0$ for $i = 1, 2$.
Define $\lambda_+^{\mathcal{B}^{\#} \mathcal{B}}$ as the smallest positive eigenvalue of $\mathcal{B}^{\#}\mathcal{B}$, and
$\rho(\mu) := 1 - |1-\mu|$. For any sequence $\{ \mathbf{\dot{U}}_k \}_{k \ge 0}$, define $\Delta$ operator by $\Delta \mathbf{\dot{U}}_{k} :=  \mathbf{\dot{U}}_{k}  - \mathbf{\dot{U}}_{k-1}.$ 

The assumptions for the convergence analysis are given in Assumption~\ref{assum1}, including the conditions of $\beta$, $L_{k,1}$, and $ L_{k,2}$.
\setlist[enumerate]{leftmargin=2em}
\begin{assumption} \label{assum1}
\,
  \begin{enumerate}[label=A\arabic*]
    \item \label{assum1a} The continuous function $f(\X)$ is coercive, i.e. $f(\X) \rightarrow \infty$ if $\|\X\|_{\mathrm{F}} \rightarrow \infty$, and $g(\W)$ is bounded from below.
    \item \label{assum1b} $g(\W)$ has $L_g$ Lipschitz continuous gradient, i.e., for every $\W_1 ,\W_2$ , it holds that 
    \begin{equation} \label{eqn:02}
    \|\nabla g(\W_1) - \nabla g(\W_2) \|_{\mathrm{F}} \le L_g \| \W_1 - \W_2 \|_{\mathrm{F}}. 
    \end{equation}
    It has been shown in~\cite[Lemma 5.7]{beck2017first} that~\eqref{eqn:02} implies that for every $\W_1$ and $\W_2$, it holds that
    \begin{equation} \label{eqn:sufficient}
    g(\W_1) \le g(\W_2) + \iprod{\nabla g(\W_2)}{\W_1 - \W_2} + \frac{L_g}{2}\|\W_1 - \W_2\|_{\mathrm{F}}^2.
    \end{equation}
    \item \label{assum1c}  The matrix $\mathcal{B}^{\#}\mathcal{B}$ is full rank or $g(\W)$ is  coercive, ${\rm range}(\mathcal{C}) \subseteq {\rm range}(\mathcal{B})$, and $\mathbf{\dot{B}} \in {\rm range}(\mathcal{B})$.
    \item \label{assum1d} The  parameters $\beta > 0$, $\mu \in (0,2)$, and   there exist three constants $a_1 > 0, a_2 > 0$, and $r > 1$ such that
    \begin{equation} \label{cond:A4}
q_1^- \mathbf{I}  \succeq a_1 \mathbf{I}\,\quad \mbox{and}\,\quad q_{2}^- \mathbf{I} + \beta \mathcal{B}^{\#}\mathcal{B} - (r\varsigma_0 + r\varsigma_1 + L_g ) \mathbf{I} \succeq a_2 \mathbf{I},
    \end{equation}
where $\mathbf{I}$ is the identity matrix with correct size, 
\begin{equation} \label{theta123}
\begin{aligned}
 \varsigma_0 := \frac{2\mu (q_2 + L_g)^2}{\beta \lambda_+^{\mathcal{B}^{\#}\mathcal{B} } (\rho(\mu))^2 },\, \hbox{ and } \varsigma_1 := \frac{2 \mu q_2^2}{\beta \lambda_+^{\mathcal{B}^{\#}\mathcal{B}} (\rho(\mu))^2}.   
\end{aligned}    
\end{equation}
    \item \label{assum1e} The parameter $\beta$ satisfies 
    $
    \beta >  \frac{2L_g} {\kappa \lambda_+^{\mathcal{B}^{\#}\mathcal{B} }\rho(\mu)}$, where $\kappa \in (0,1)$ is a given constant. 
  \end{enumerate}
\end{assumption}
One can verify that~\ref{assum1a}, \ref{assum1b}, and~\ref{assum1c} in Assumption~\ref{assum1} are satisfied for SLRQA in~\eqref{model:new}. If $\beta$ is chosen sufficiently large and the sequences $\{L_{k,1}\}$ and $\{L_{k,2}\}$ are chosen such that $q_i^-,q_i$, $i=1, 2$ in suitable {\rm range}, then~\ref{assum1d}
and~\ref{assum1e} in Assumption~\ref{assum1} are also satisfied. For example, if set $L_{k,1} = L_{k,2} = 1$ for all $k$ and let
\[
\beta > \max\left\{\frac{2L_g} {\lambda_+^{\mathcal{B}^{\#}\mathcal{B} }\rho(\mu)},  \frac{r (2\mu (1 + L_g)^2 + 2\mu)}{ (\lambda_+^{\mathcal{B}^{\#}\mathcal{B} })^2 (\rho(\mu))^2  } +   \frac{L_g  }{\lambda_+^{\mathcal{B}^{\#}\mathcal{B} } }  ,1   \right\},
\]
then \ref{assum1d} and \ref{assum1e} in Assumption \ref{assum1} hold. Therefore, Assumption~\ref{assum1} is reasonable.


Lemmas~\ref{lem:X} and \ref{lem:W} show that the augmented Lagrangian function $\mathcal{L}_{\beta}$ is sufficiently descent when updating $\W$ and $\X$. 
They are used in Lemma~\ref{lem:Y} for bounding $\mathcal{L}_{\beta}(\X_{k+1},\W_{k+1},\La_{k+1}) -\mathcal{L}_{\beta}(\X_{k},\W_{k},\La_{k})$.

\begin{lemma}[Sufficient descent of $\mathcal{L}_\beta$ for $\X$ update] \label{lem:X}
 The sequence $\{ (\X_k,\W_k,\La_k)\}_{k \ge 0}$ generated by   \eqref{scheme2} satisfies:
\[
\mathcal{L}_{\beta}(\X_{k},\W_{k},\La_k) - \mathcal{L}_{\beta}(\X_{k+1},\W_{k},\La_k) \geq \frac{L_{k,1}}{2} \|\Delta \X_{k+1}\|^2_{\mathrm{F}}. 
\]
\end{lemma}

\begin{proof}
According to the $\X_{k+1}$ update in  \eqref{scheme2}, the following inequality holds:
\be \label{Wup2}
\begin{aligned}
&f(\X_{k+1}) - f(\X_k) + \frac{\beta}{2}\|\mathcal{C}(\X_{k+1})+\mathcal{B}(\W_{k}) - \mathbf{\dot{B}} + \La_k/\beta \|_\mF^2 + \frac{L_{k,1}}{2}\| \Delta \X_{k+1} \|_{\mathrm{F}}^2 \\
&+ \frac{1}{2\tau^2}\|\mathcal{P}_{\Omega}((\mathcal{A}(\X_{k+1}) - \Y)) \|_{\mathrm{F}}^2\leq  \frac{\beta}{2}\left\|\mathcal{C}(\X_k) + \mathcal{B}(\W_k)- \mathbf{\dot{B}} + \La_k/\beta\right\|_\mF^2 + \frac{1}{2\tau^2}\|\mathcal{P}_{\Omega}((\mathcal{A}(\X_{k}) - \Y)) \|_{\mathrm{F}}^2.
\end{aligned}
\ee
It follows that
\[
\begin{aligned}
&\mathcal{L}_{\beta}(\X_{k},\W_{k},\La_k) - \mathcal{L}_{\beta}(\X_{k+1},\W_{k},\La_k)  \\
&=  f(\X_k) - f(\X_{k+1}) + \frac{1}{2\tau^2}\|\mathcal{P}_{\Omega}((\mathcal{A}(\X_{k}) - \Y)) \|_{\mathrm{F}}^2 - \frac{1}{2\tau^2}\|\mathcal{P}_{\Omega}((\mathcal{A}(\X_{k+1}) - \Y))\|_{\mathrm{F}}^2  \\
& - \frac{\beta}{2}\|\mathcal{C}(\X_{k+1}) +\mathcal{B}(\W_{k}) - \qb +\La_k/\beta \|_{\mathrm{F}}^2  + \frac{\beta}{2} \|\mathcal{C}(\X_{k}) + \mathcal{B}(\W_{k}) - \qb + \La_k/\beta \|_{\mathrm{F}}^2 \ge \frac{L_{k,1}}{2} \|\Delta \X_{k+1}\|^2,
\end{aligned}
\]
which completes the proof.
\end{proof}

\begin{lemma}[Sufficient descent of $\mathcal{L}_\beta$ for $\W$ update] \label{lem:W}
Suppose \ref{assum1b} in Assumption \ref{assum1}  holds.  The sequence $\{ (\X_k,\W_k,\La_k)\}_{k \ge 0}$ generated by~\eqref{scheme2} satisfies:
\[
\mathcal{L}_{\beta}(\X_{k+1},\W_{k},\La_k) - \mathcal{L}_{\beta}(\X_{k+1},\W_{k+1},\La_k) \geq \|\Delta \W_{k+1}\|^2_{B_k},  
\]
where $B_{k} :=L_{k,2}\mathbf{I} - L_g \mathbf{I} + \frac{\beta}{2} \mathcal{B}\mathcal{B}^{\#}$.
\end{lemma}

\begin{proof}
According to the optimality condition of $\W_{k+1},$ the following equality holds:
\begin{equation} \label{upW:1}
 \nabla g(\W_k) + \beta \mathcal{B}^{\#}(\mathcal{C}(\X_{k+1}) + \mathcal{B}(\W_{k+1}) - \qb + \La_{k}/\beta ) + L_{k,2} \Delta \W_{k+1} = 0.   
\end{equation}
It follows that
\[
\begin{aligned}
&\mathcal{L}_{\beta}(\X_{k+1},\W_{k},\La_k) - \mathcal{L}_{\beta}(\X_{k+1},\W_{k+1},\La_k)  \\
& = g(\W_k) - g(\W_{k+1}) + \frac{\beta}{2} \|\mathcal{C}(\X_{k+1}) + \mathcal{B}(\W_k)-\qb + \frac{\La_k}{\beta} \|_{\mathrm{F}}^2 - \frac{\beta}{2} \|\mathcal{C}(\X_{k+1}) + \mathcal{B}(\W_{k+1})-\qb + \frac{\La_k}{\beta} \|_{\mathrm{F}}^2 \\
& = g(\W_k) - g(\W_{k+1}) + \frac{\beta}{2}\|\mathcal{B}(\Delta \W_{k+1})\|_{\mathrm{F}}^2 - \beta \iprod{\mathcal{B}^{\#}( \mathcal{C}(\X_{k+1}) + \mathcal{B}(\W_{k+1}) -\qb + \frac{\La_k}{\beta} ) }{\Delta \W_{k+1}}  \\
& = g(\W_k) - g(\W_{k+1}) +  \iprod{\nabla g(\W_k)}{\Delta \W_{k+1}}  + L_{k,2} \|\Delta \W_{k+1}\|_{\mathrm{F}}^2  + \frac{\beta}{2}\|\mathcal{B}\Delta \W_{k+1}\|_{\mathrm{F}}^2\, \quad (\text{by~\eqref{upW:1}} ) \\
& \ge  \|\Delta \W_{k+1} \|_{B_{k}}^2, \qquad (\text{by Assumption~\ref{assum1b}} )
\end{aligned}
\]
which completes the proof.
\end{proof}

Lemma~\ref{lem:Y} is used in Lemma~\ref{lem:mon}, which is crucial for constructing a merit function.
\begin{lemma} \label{lem:Y}
Suppose \ref{assum1b} in Assumption \ref{assum1}  holds. The sequence $\{ (\X_k,\W_k,\La_k)\}_{k \ge 0}$ generated by~\eqref{scheme2} satisfies:
\[
\mathcal{L}_{\beta}(\X_{k+1},\W_{k+1},\La_{k+1})  \leq \mathcal{L}_{\beta}(\X_{k},\W_{k},\La_{k}) - \|\Delta \W_{k+1}\|^2_{B_k} - \frac{L_{k,1}}{2} \|\Delta \X_{k+1} \|^2_{\mathrm{F}} + \frac{1}{\beta \mu} \| \Delta \La_{k+1} \|^2_{\mathrm{F}}.
\]
\end{lemma}

\begin{proof}
According to the update of $\La$ in \eqref{scheme2}, Lemma \ref{lem:X}, and Lemma \ref{lem:W}, we have
\[
\begin{aligned}
\mathcal{L}_{\beta}(\X_{k+1},\W_{k+1},\La_{k+1}) & = \mathcal{L}_{\beta}(\X_{k+1},\W_{k+1},\La_{k}) + \frac{1}{\beta \mu} \| \Delta \La_{k+1} \|^2_{\mathrm{F}} \\
& \leq \mathcal{L}_{\beta}(\X_{k},\W_{k},\La_{k}) -\|\Delta \W_{k+1} \|_{B_k}^2 - \frac{L_{k,1}}{2} \|\Delta \X_{k+1} \|_{\mathrm{F}}^2 + \frac{1}{\beta \mu} \| \Delta \La_{k+1} \|^2_{\mathrm{F}}. \\
\end{aligned}
\]
\end{proof}
\begin{lemma} \label{lem:mon}
Suppose \ref{assum1b},\ref{assum1c} in Assumption \ref{assum1}  hold, the sequence $\{(\X_{k},\W_{k},\La_{k})\}_{k \ge 0}$ generated by  \eqref{scheme2} satisfies:
\begin{align}
& \frac{1}{\beta \mu}\|\Delta \La_{k+1}\|^2_{\mathrm{F}} \le \varsigma_0 \|\Delta \W_k \|^2_{\mathrm{F}} + \varsigma_1 \|\Delta \W_{k+1} \|^2_{\mathrm{F}} + \varsigma_2 \|\mathcal{B}^{\#} (\Delta \La_k) \|^2_{\mathrm{F}} - \varsigma_2 \|\mathcal{B}^{\#} (\Delta \La_{k+1}) \|^2_{\mathrm{F}}, \label{eqn:La1} \\
& \mathcal{L}_{\beta}(\X_{k+1},\W_{k+1},\La_{k+1}) + \|\Delta \W_{k+1}\|_{B_{k} - r \varsigma_1 \mathbf{I}}^2 + \frac{L_{k,1}}{2}\| \Delta \X_{k+1} \|^2_{\mathrm{F}} + \frac{r-1}{\beta \mu} \|\Delta \La_{k+1}\|^2_{\mathrm{F}} \notag \\
& + r \varsigma_2 \| \mathcal{B}^{\#} \Delta \La_{k+1} \|^2_{\mathrm{F}} \leq \mathcal{L}_{\beta}(\X_k,\W_k,\La_k) + r \varsigma_0 \| \Delta \W_k\|^2_{\mathrm{F}} + r \varsigma_2 \| \mathcal{B}^{\#} (\Delta \La_k) \|^2_{\mathrm{F}}, \label{eqn:La2}
\end{align}
where $r > 1, \varsigma_0, \varsigma_1$ are defined in \eqref{theta123}, and 
$
  \varsigma_2 := \frac{|1-\mu|}{\beta \mu \lambda_+^{\mathcal{B}^{\#}\mathcal{B}}\rho(\mu)}.  
$

\end{lemma}

\begin{proof}
 Let $k \ge 1$ be fixed and define the matrix
\begin{equation} \label{def:qw}
    \qG_{k+1} := - L_{k,2} \Delta \W_{k+1} - \nabla g(\W_k).
\end{equation}
It follows that $\Delta \qG_{k+1} = L_{2,k-1} \Delta \W_k - L_{k,2} \Delta \W_{k+1} + \nabla g(\W_{k-1}) - \nabla g(\W_k).$  By the  triangle inequality, the following inequality holds:
\[
\|\Delta \qG_{k+1}\|_{\mathrm{F}} \le \|\nabla g(\W_{k-1}) - \nabla g(\W_k) \|_{\mathrm{F}} + L_{k,2} \|\Delta \W_{k+1}\|_{\mathrm{F}} + L_{2,k-1} \|\Delta \W_k\|_{\mathrm{F}}.
\]
By \ref{assum1b} in Assumption \ref{assum1}, $\nabla g(\W)$ is $L_g$ Lipschitz continuous and $q_2 = \sup_{k \ge 0} L_{k,2} < \infty,$ we have
$
\|\Delta \qG_{k+1}\|_{\mathrm{F}} \le (L_g + q_2)\|\Delta \W_k\|_{\mathrm{F}} + q_2 \|\Delta\W_{k+1}\|_{\mathrm{F}}.
$
Hence it follows that:
\begin{equation} \label{eqn:qw}
\|\Delta \qG_{k+1}\|^2_{\mathrm{F}} \le 2(L_g + q_2)^2 \|\Delta \W_k\|^2_{\mathrm{F}} + 2q_2^2 \|\Delta \W_{k+1}\|^2_{\mathrm{F}}.
\end{equation}
Expressing the optimality condition of $\W$ subproblem using $\qG_{k+1}$, we have
$
\qG_{k+1} = \beta \mathcal{B}^{\#}(\mathcal{C}(\X_{k+1}) + \mathcal{B}(\W_{k+1})- \qb + \La_k/\beta ).$
Combining this with the $\La$ update, we obtain
\begin{equation} \label{eqn:La}
    \mathcal{B}^{\#} \La_{k+1} = \mu \qG_{k+1} + (1-\mu) \mathcal{B}^{\#}\La_k.
\end{equation}
It follows that $\mathcal{B}^{\#} \Delta \La_{k+1} = \mu \Delta \qG_{k+1} + (1-\mu) \mathcal{B}^{\#} \Delta \La_k.$ Since $\mu \in (0,2),$ we have
\[
\mathcal{B}^{\#} \Delta \La_{k+1} = \rho(\mu) \cdot \frac{\mu \Delta \qG_{k+1}}{\rho(\mu)} + |1-\mu|\cdot (\text{sign}(1-\mu)\mathcal{B}^{\#}\Delta \La_k),
\]
By the  convexity of $\|\cdot\|^2_{\mathrm{F}}$, the update of $\La$ in \eqref{scheme2} and \ref{assum1c} in Assumption \ref{assum1}, it follows that $\Delta \La_{k+1} \in \text{{\rm range}}(\mathcal{B})$,  Hence, the following inequalities hold
\begin{align} 
\lambda_+^{\mathcal{B}^{\#}\mathcal{B}}\rho(\mu) \|\Delta\La_{k+1}\|^2_{\mathrm{F}} \leq& \rho(\mu)\|\mathcal{B}^{\#} \Delta \La_{k+1}\|_{\mathrm{F}}^2 = \|\mathcal{B}^{\#} \Delta \La_{k+1}\|_{\mathrm{F}}^2 - |1-\mu| \|\mathcal{B}^{\#} \Delta \La_{k+1}\|^2_{\mathrm{F}} \nonumber \\
\le& \frac{\mu^2}{\rho(\mu)} \|\Delta \qG_{k+1}\|^2_{\mathrm{F}} + |1-\mu| \|\mathcal{B}^{\#}\Delta \La_k\|^2_{\mathrm{F}} - |1-\mu| \|\mathcal{B}^{\#} \Delta \La_{k+1}\|^2_{\mathrm{F}}, \label{eqn:4-6-1}
\end{align}
where the first inequality follows from~\ref{assum1c} in Assumption~\ref{assum1} and the definition of $\lambda_+^{\mathcal{B}^{\#}\mathcal{B}}$.
Dividing both sides of \eqref{eqn:4-6-1} by $\beta \mu \lambda_+^{\mathcal{B}^{\#}\mathcal{B}} \rho(\mu)$ and using~\eqref{eqn:qw} yield~\eqref{eqn:La1}. 
Finally, Inequality~\eqref{eqn:La2} follows by multiplying the inequality \eqref{eqn:La1} by $r >1$ and combining it with Lemma \ref{lem:Y}.
 \end{proof}

Let $\mathcal{R}(\X,\W,\La,\W',\La'):= \mathcal{L}_{\beta}(\X,\W,\La) + r\varsigma_0 \|\W-\W'\|^2_{\mathrm{F}} + r\varsigma_2\|\mathcal{B}^{\#}(\La-\La') \|^2_{\mathrm{F}}$. The merit function $\mathcal{R}_k$ is defined by
\begin{equation} \label{definition:Rk}
\begin{aligned}
&\mathcal{R}_k :=\mathcal{R}(\X_k,\W_k,\La_k,\W_{k-1},\La_{k-1}) = \mathcal{L}_{\beta}(\X_k,\W_k,\La_k)+r\varsigma_0\|\Delta \W_k\|^2_{\mathrm{F}} + r\varsigma_2 \|\mathcal{B}^{\#}\Delta\La_k\|^2_{\mathrm{F}}.
\end{aligned}
\end{equation}
If Assumption~\ref{assum1} holds, then it can be shown that the merit function is descent. Specifically, we have
\be \label{ieqn:RK}
\begin{aligned}
& \mathcal{R}_{k+1} + a\left(\|\Delta \X_{k+1} \|_{\mathrm{F}}^2 + \|\Delta \W_{k+1} \|_{\mathrm{F}}^2 + \|\Delta \La_{k+1}\|^2_{\mathrm{F}} \right) \\
& \leq \mathcal{R}_{k+1} + \|\Delta \W_{k+1}\|_{B_{k} - r (\varsigma_1+\varsigma_0) \mathbf{I}}^2 + \frac{L_{k,1}}{2}\| \Delta \X_{k+1} \|^2_{\mathrm{F}} + \frac{r-1}{\beta \mu} \|\Delta \La_{k+1}\|^2_{\mathrm{F}}  
\le \mathcal{R}_k \le \mathcal{R}_{k_0},
\end{aligned}
\ee
where $a = \min \{a_1,a_2,\frac{r-1}{\beta\mu} \}$ and $a_1,a_2$ are defined in Assumption~\ref{assum1}, the first inequality follows from \ref{assum1d}, \ref{assum1e} in  Assumption \ref{assum1}, the second inequality follows from Lemma \ref{lem:mon}, and the third inequality is due to induction of $\mathcal{R}_k \leq \mathcal{R}_{k+1}$ for any $k \ge k_0$.

\begin{theorem}[Bounded sequence of $\{(\X_k,\W_k,\La_k) \}_{k \ge 0}$] \label{thm:Xbound}
Assume Assumption \ref{assum1} holds.  The sequence $\{(\X_k,\W_k,\La_k)\}_{k \ge 0}$ generated by  \eqref{scheme2} is bounded.
\end{theorem}

\begin{proof}
According to \eqref{ieqn:RK}, there exists $k_0 \ge 1$ such that $\mathcal{R}_{k+1} \le \mathcal{R}_{k_0}$ for all $k \ge k_0.$ Hence the following inequality holds:
\begin{equation}  \label{ieqn:RK2}
\begin{aligned}
& f(\X_{k+1}) + g(\W_{k+1}) + \frac{1}{2\tau^2}\|\mathcal{P}_{\Omega}(\mathcal{A}(\X_k)) - \mathcal{P}_{\Omega}(\Y)\|_{\mathrm{F}}^2 + \frac{\beta}{2}\left\|\mathcal{C}(\X_{k+1}) + \mathcal{B}(\W_{k+1}) - \mathbf{\dot{B}} + \frac{\La_{k+1}}{\beta} \right\|_\mF^2 \\ &- \frac{1}{2\beta}\left\|\La_{k+1}\right\|_\mF^2  +(r \varsigma_0 + a) \|\Delta \W_{k+1}\|_{\mathrm{F}}^2 + a (\|\Delta \X_{k+1}\|_{\mathrm{F}}^2 + \|\Delta \La_{k+1} \|_{\mathrm{F}}^2  )+ r \varsigma_2 \|\mathcal{B}^{\#} \Delta \La_{k+1}\|_{\mathrm{F}}^2 \le \mathcal{R}_{k_0}. 
\end{aligned}
\end{equation}
  According to \eqref{eqn:La}, we have
\begin{equation} \label{eqn:muLa1}
 \mu \mathcal{B}^{\#} \La_{k+1} = \mu \qG_{k+1} + (1-\mu) \mathcal{B}^{\#}(\La_k - \La_{k+1}).
\end{equation}
Since $\mu \in (0,2),$  Equation \eqref{eqn:muLa1} can be rewritten as:
$$
\mu \mathcal{B}^{\#} \La_{k+1} = \rho(\mu) \frac{ \mu \qG_{k+1}}{\rho(\mu)} + |1-\mu|\left(\text{sign}(1-\mu) \mathcal{B}^{\#}(\La_k - \La_{k+1}) \right).
$$
It follows from the  convexity of $\|\cdot\|^2$ that
$
\lambda_+^{\mathcal{B}^{\#}\mathcal{B}} \mu^2 \|\La_{k+1}\|_{\mathrm{F}}^2 \le \frac{\mu^2}{\rho(\mu)} \|\qG_{k+1}\|_{\mathrm{F}}^2 + |1-\mu| \|\mathcal{B}^{\#}\Delta \La_{k+1}\|_{\mathrm{F}}^2.
$
According to definition \eqref{def:qw}, we have 
$
\|\qG_{k+1}\|_{\mathrm{F}}^2 \le 2 q_2^2 \|\Delta \W_{k+1}\|_{\mathrm{F}}^2 + 2\|\nabla g(\W_k)\|_{\mathrm{F}}^2 \le 2(q_2 + L_g)^2 \|\Delta \W_{k+1}\|_{\mathrm{F}}^2 + 2\|\nabla g(\W_{k+1})\|_{\mathrm{F}}^2.
$
Hence it follows that:
\begin{equation} \label{eqn:La3}
-\frac{1}{2\beta} \|\La_{k+1}\|_{\mathrm{F}}^2 \ge -\varsigma_3 \|\nabla g(\W_{k+1})\|_{\mathrm{F}}^2 - \varsigma_4 \|\Delta \W_{k+1}\|_{\mathrm{F}}^2 - \varsigma_5 \|\mathcal{B}^{\#} \Delta \La_{k+1}\|_{\mathrm{F}}^2,
\end{equation}
where 
\[
\varsigma_3 := \frac{1}{\beta \rho(\mu) \lambda_+^{\mathcal{B}^{\#}\mathcal{B}}},\quad \varsigma_4 := \frac{(q_2+L_g)^2}{\beta \rho(\mu) \lambda_+^{\mathcal{B}^{\#}\mathcal{B}}}, \quad \varsigma_5 := \frac{|1-\mu|}{2\beta \mu^2 \lambda_+^{\mathcal{B}^{\#}\mathcal{B}} }.
\]
Using  Inequalities \eqref{eqn:La3} and \eqref{ieqn:RK2}, we obtain
\begin{equation} \label{eqn:ieqn:RK3}
\begin{aligned}
&  f(\X_{k+1}) + (1-\kappa)g(\W_{k+1}) + \frac{1}{2\tau^2}\|\mathcal{P}_{\Omega}(\mathcal{A}(\X_k)) - \mathcal{P}_{\Omega}(\Y)\|_{\mathrm{F}}^2 + \frac{\beta}{2}\left\|\mathcal{C}(\X_{k+1}) + \mathcal{B}(\W_{k+1}) - \mathbf{\dot{B}} + \frac{\La_{k+1}}{\beta} \right\|_\mF^2\\
&+ a (\|\Delta \X_{k+1}\|_{\mathrm{F}}^2 + \|\Delta \La_{k+1}\|_{\mathrm{F}}^2) + (r \varsigma_0 + a -\varsigma_4) \|\Delta \W_{k+1}\|_{\mathrm{F}}^2 + (r \varsigma_2 - \varsigma_5) \|\mathcal{B}^{\#} \Delta \La_{k+1}\|_{\mathrm{F}}^2 \\
&\le \mathcal{R}_{k_0} - \inf_{\W} \left\{ \kappa g(\W) -  \varsigma_3 \|\nabla g(\W)\|_{\mathrm{F}}^2 \right\},
\end{aligned}
\end{equation}
where $\kappa \in (0,1).$ According to \eqref{eqn:sufficient} in Assumption \ref{assum1b},
 setting $\W_1 = \W - \delta \nabla g(\W)$ and $\W_2 = \W$, it follows that $\kappa g(\W_k - \delta \nabla g(\W_k)) \le \kappa g(\W_k) - \kappa(\delta - \frac{L_g \delta^2}{2})\|\nabla g(\W_k)\|_{\mathrm{F}}^2.$ Since $g$ is bounded from below, there exist $M$ such that
\begin{equation} \label{ineq:delta1}
    -M < \inf\{\kappa g(\W) - \kappa (\delta - \frac{L_g \delta^2}{2})\|\nabla g(\W)\|_{\mathrm{F}}^2: \W \in \mathbb{H}^{m \times n} \}.
\end{equation}
We choose $\delta = \frac{1}{L_g}$. According to Assumption \ref{assum1e}, we have $\varsigma_3 < \frac{\kappa}{2  L_g} = \kappa (\delta - \frac{L_g \delta^2}{2}).$  Since $r > 1$ and $\mu \in (0,2),$ according to the denifition of $\varsigma_0,\varsigma_2,\varsigma_4$ and $\varsigma_5$, it holds that $r\varsigma_2 - \varsigma_5 >0$ and $r \varsigma_0 + a - \varsigma_4 >0$. It follows from \eqref{eqn:ieqn:RK3} that
\begin{equation} \label{L:bound}
\begin{aligned}
& f(\X_{k+1}) + (1-\kappa)g(\W_{k+1}) + \frac{1}{2\tau^2}\|\mathcal{P}_{\Omega}(\mathcal{A}(\X_k)) - \mathcal{P}_{\Omega}(\Y)\|_{\mathrm{F}}^2 \\
& + \frac{\beta}{2}\left\|\mathcal{C}(\X_{k+1}) + \mathcal{B}(\W_{k+1}) - \mathbf{\dot{B}} + \frac{\La_{k+1}}{\beta} \right\|_\mF^2  + a(\|\Delta \X_{k+1}\|_{\mathrm{F}}^2 + \|\Delta \La_{k+1}\|_{\mathrm{F}}^2)< \mathcal{R}_{k_0} + M.
\end{aligned}
\end{equation}
Since $f$ is coercive, the sequence $\{\X_k\}_{k \ge k_0 }$  is  bounded and hence $\{\mathcal{C}\X_k\}_{k \ge k_0}$ is bounded. According to \eqref{L:bound}, $\Delta \La_{k+1}$ is bounded. By the $\La$ update in \eqref{scheme2}, we have $\mathcal{B}\W_{k+1} =  \frac{1}{\beta \mu}\Delta \La_{k+1} - \mathcal{C}\X_{k+1}+\qb.$ 
 Since $\{\mathcal{C}\X_k\}_{k \ge k_0}$ is bounded, if $\mathcal{B}^{\#}\mathcal{B}$ is full rank or $g$ is coercive,  $\{ \W_k \}_{k \ge k_0}$ is bounded. 
From the fact that $ \frac{\beta}{2}\left\|\mathcal{C}(\X_{k+1}) + \mathcal{B}(\W_{k+1}) - \mathbf{\dot{B}} + \frac{\La_{k+1}}{\beta} \right\|_\mF^2$ is bounded, it follows that $\{\La_k\}_{k \ge k_0}$ is bounded. As a consequence, $\{\X_k\}_{k \ge 1}, \{\W_k\}_{k \ge 1}, \{\La_k\}_{k \ge 1}$ is bounded. 
\end{proof}

Using Theorem \ref{thm:Xbound}, we have the following convergence result of  $\mathcal{R}_k$. 
\begin{lemma} \label{lem:bound}
Suppose Assumption \ref{assum1} holds. The sequence $\{\mathcal{R}_k\}_{k \ge 1}$ is bounded from below and converges.
\end{lemma}

\begin{proof}
It follows from Theorem \ref{thm:Xbound} that $\{(\X_k,\W_k,\La_k)\}_{k \ge 1}$ is bounded. 
Since $\mathcal{L}_{\beta}$ is continous respect to $\X_k,\W_k,\La_k$, it follows that $\mathcal{L}_{\beta}(\X_k,\W_k,\La_k)$ is bounded from below. Hence $\mathcal{R}_k$ is bounded from below. According to Assumption \ref{assum1} and \eqref{ieqn:RK}, $\{ \mathcal{R}_k \}_{k \ge 1}$ is monotonically decreasing. As a consequence, $\{\mathcal{R}_k\}_{k \ge 1}$ is bounded from below and converges.
\end{proof}

The following convergence result of sequence $\{\X_{k},\W_{k},\La_{k}\}$ is given by  Lemma \ref{lem:bound}. 
\begin{theorem} \label{thm:Xcov}
Suppose Assumption  \ref{assum1}  holds. It follows that
\[
\lim_{k \rightarrow \infty} \|\Delta \W_k \|_{\mathrm{F}} = 0, \lim_{k \rightarrow \infty} \|\Delta \X_k \|_{\mathrm{F}} =0, \lim_{k \rightarrow \infty}\|\Delta \La_k \|_{\mathrm{F}} = 0.
\]

\end{theorem} 

\begin{proof}
According to Theorem \ref{thm:Xbound}, summing up \eqref{ieqn:RK} from $k_0$ to $K \ge k_0$, the following inequality holds
\[
\sum_{k=k_0}^{K}\|\Delta \W_{k+1}\|_{\mathrm{F}}^2 + \|\Delta \X_{k+1}\|^2_{\mathrm{F}} + \|\Delta \La_{k+1}\|^2_{\mathrm{F}} < \frac{1}{a}(\mathcal{R}_{k_0}-\inf_{k \ge 1}\mathcal{R}_k).
\]
According to Lemma \ref{lem:bound}, it follows that $-\infty < \inf_{k \ge 1} \mathcal{R}_k$.  Let $K \rightarrow +\infty$, since $\{(\X_k,\W_k,\La_k)\}_{k \ge 0}$ is bounded, then $\mathcal{R}_{k_0}$ is bounded and hence
$
\sum_{k \ge 0} \|\Delta \W_{k+1} \|^2_{\mathrm{F}} + \|\Delta \X_{k+1}\|^2_{\mathrm{F}} + \|\Delta \La_{k+1}\|^2_{\mathrm{F}} < \infty,
$
which completes the proof.
\end{proof}

Theorem \ref{thm:Xcov} has proven the limiting behavior of $\|\Delta \W_k\|_{\mathrm{F}}, \|\Delta \X_k\|_{\mathrm{F}}$ and $\|\Delta \La_k\|_{\mathrm{F}}$, which will be used to prove the limiting behavior  of $\mathcal{R}_{k}.$ We next prove that the subgradient of $\mathcal{L}_{\beta}(\X,\W,\La)$ can be bounded by $\Delta \X_{k+1},\Delta \W_{k+1}$ and $\Delta \La_{k+1}.$

\begin{lemma}[Subgradient bound] \label{lem:sub:bound}
Suppose that Assumption \ref{assum1b} holds.
 Let $\{(\X_k,\W_k,\La_k)\}_{k \ge 0}$ be a sequence generated by~\eqref{scheme2}. Then $\widetilde{\qd}_{k+1}:= (\qd_{\X_{k+1}},\qd_{\W_{k+1}},\qd_{\La_{k+1}}) \in \partial \mathcal{L}_{\beta}(\X_{k+1},\W_{k+1},\La_{k+1})$,  where
\[
\begin{aligned}
\qd_{\X_{k+1}} &:=  \mathcal{C}^{\#} \Delta \La_{k+1} - L_{k,1} \Delta \X_{k+1}  - \beta \mathcal{C}^{\#}\mathcal{B} \Delta \W_{k+1} ,\quad \qd_{\La_{k+1}} := \frac{1}{\beta \mu} \Delta \La_{k+1} ,\, \\
\qd_{\W_{k+1}} &:= \nabla g(\W_{k+1}) - \nabla g(\W_k) + \mathcal{B}^{\#}\Delta \La_{k+1}  - L_{k,2} \Delta \W_{k+1}.
\end{aligned}
\]
Furthermore, there exists $\pi >0$ such that
\[
|\|\widetilde{\qd}_{k+1}\|| \leq \pi \left( \| \Delta \W_{k+1}\|_{\mathrm{F}} + \| \Delta \X_{k+1}\|_{\mathrm{F}}+ \| \Delta \La_{k+1}\|_{\mathrm{F}} \right),
\]
where 
\begin{equation} \label{def:rho}
 \pi := \max \left\{q_1, \beta \|\mathcal{C}\| \|\mathcal{B}\| + q_2 + L_g, \| \mathcal{C}\| + \|\mathcal{B}\| + \frac{1}{\beta\mu} \right\}.   
\end{equation}

\end{lemma}
\begin{proof}
According to the optimality condition of the update of $\X_{k+1},$ it follows that
\[
  -\frac{1}{\tau^2}\mathcal{A}^{\#}\mathcal{P}^{\#}_{\Omega}(\mathcal{P}_{\Omega}(\mathcal{A}\X_{k+1} - \Y)) - \beta \mathcal{C}^{\#}(\mathcal{C}(\X_{k+1}) + \mathcal{B}(\W_{k}) - \mathbf{\dot{B}} + \La_k/\beta) - L_{k,1} \Delta \X_{k+1} \in \partial f(\X_{k+1}),
\]
which yields
\begin{equation} \label{eqn:dX}
\qd_{\X_{k+1}} =  - \mathcal{C}^{\#} \Delta \La_{k+1} - L_{k,1} \Delta \X_{k+1} + \beta \mathcal{C}^{\#}\mathcal{B} \Delta \W_{k+1} \in \partial_{\X} \mathcal{L}_{\beta}(\X_{k+1},\W_{k+1},\La_{k+1}).
\end{equation}
 According to the optimality condition of the update of $\W_{k+1},$ the following equality holds:
 \begin{equation} \label{eqn:01}
 \nabla g(\W_{k}) + \beta\mathcal{B}^{\#}(\mathcal{C}\X_{k+1} - \mathcal{B}\W_{k+1}-\qb + \La_k/\beta) + L_{k,2} \Delta \W_{k+1}   = 0.
 \end{equation}
It follows from~\eqref{eqn:01} and the update of $\La_{k+1}$ that
\begin{equation} \label{eqn:dW}
\begin{aligned}
 \qd_{\W_{k+1}} &= \nabla g(\W_{k+1}) - \nabla g(\W_k) + \mathcal{B}^{\#}\Delta \La_{k+1}  - L_{k,2} \Delta \W_{k+1} \in \partial_{\W} \mathcal{L}_{\beta}(\X_{k+1},\W_{k+1},\La_{k+1}). \\
\qd_{\La_{k+1}} &= \frac{1}{\beta \mu} \Delta \La_{k+1} \in \partial_{\La} \mathcal{L}_{\beta}(\X_{k+1},\W_{k+1},\La_{k+1}).
\end{aligned}
\end{equation}
Combining \eqref{eqn:dX}, \eqref{eqn:dW}, we obtain
\[
\begin{aligned}
\| \qd_{\X_{k+1}}\| & \le \|\mathcal{C}\| \|\Delta \La_{k+1}\|_{\mathrm{F}} + q_1 \|\Delta \X_{k+1} \|_{\mathrm{F}} + \beta \|\mathcal{C}\| \|\mathcal{B}\| \|\Delta \W_{k+1} \|_{\mathrm{F}}, \\
\| \qd_{\W_{k+1}}\| & \le \|\mathcal{B}\| \|\Delta \La_{k+1}\|_{\mathrm{F}}   + (q_2 + L_g) \|\Delta \W_{k+1}\|_{\mathrm{F}},\quad 
\| \qd_{\La_{k+1}}\|  = \frac{1}{\beta \mu} \|\Delta \La_{k+1}\|_{\mathrm{F}}. 
\end{aligned}
\]
Therefore, the following inequality holds:
\begin{equation} \label{eqn:bound4}
| \|\widetilde{\qd}_{k+1}\| | \le \|\qd_{\W_{k+1}}\|_{\mathrm{F}} + \|\qd_{\X_{k+1}} \|_{\mathrm{F}} + \|\qd_{\La_{k+1}}\|_{\mathrm{F}} \le \rho (\|\Delta \X_{k+1} \|_{\mathrm{F}} + \|\Delta \W_{k+1} \|_{\mathrm{F}}   + \|\Delta \La_{k+1} \|_{\mathrm{F}} ),
\end{equation}
which completes the proof.
\end{proof}

%
\begin{lemma} \label{crit}
Suppose that Assumption \ref{assum1} holds.
Any limit point $(\X_*,\W_*,\La_*)$ of the sequence $\{(\X_{k},\W_{k},\La_{k})\}_{j \ge 0}$ generated by~\eqref{scheme2} is a stationary
point of $\mathcal{L}_{\beta}(\X, \W; \La)$, i.e.
\[
0 =  \partial f(\X_*) + \frac{1}{\tau^2}\mathcal{A}^{\#}\mathcal{P}^{\#}_{\Omega}(\mathcal{P}_{\Omega}(\mathcal{A}\X_{k+1} - \Y))+ \mathcal{C}^{\#}\La_*,\quad 0 \in \partial  g(\W_*) + \mathcal{B}^{\#}\La_*,\quad  \mathcal{C}\X_* + \mathcal{B}\W_* = \mathbf{\dot{B}}.
\]
\end{lemma}

\begin{proof}
Let $\{(\X_{k_j},\W_{k_j},\La_{k_j})\}_{j \ge 0}$ be a subsequence of $\{(\X_{k},\W_{k},\La_{k})\}_{k \ge 0}$ such that $(\X_*,\W_*,\La_*) = \lim_{j \rightarrow \infty}(\X_{k_j},\W_{k_j},\La_{k_j})$. By the continuity of $\mathcal{L}_{\beta}$, $\mathcal{L}_{\beta} (\X_{k_j},\W_{k_j},\La_{k_j}) \rightarrow \mathcal{L}_{\beta} (\X_*,\W_*,\La_*)$ as $j \rightarrow \infty$. Let $\qd_{k_j} \in \partial \mathcal{L}_{\beta}(\X_{k_j},\W_{k_j},\La_{k_j})$ and according to Lemma \ref{lem:sub:bound}, it follows that
$
\|\qd_{k_j}\|_{\mathrm{F}} \leq \rho \left( \| \Delta \W_{k_j}\|_{\mathrm{F}} + \| \Delta \X_{k_j}\|_{\mathrm{F}}+ \| \Delta \La_{k_j}\|_{\mathrm{F}} \right).
$
According to Theorem~\ref{thm:Xcov}, it follows that $\qd_{k_j} \rightarrow 0.$ By the closeness
criterion of the limiting sub-differential, $(\X_*,\W_*,\La_*) \in \text{crit}\, \mathcal{L}_{\beta}(\X,\W,\La)$. The proof is completed.
\end{proof}

The following Lemma states the behavior of limit points of $\{\X_{k},\W_{k},\La_{k} \}.$
\begin{lemma} \label{lem:8}
Suppose Assumption \ref{assum1} holds. If $(\X_*,\W_*,\La_*)$ is a limit point of a converging subsequence $\{\X_{k_j},\W_{k_j},\La_{k_j}\}$, then it follows that
\[
\lim_{j \rightarrow \infty} \mathcal{R}_{k_j} = \mathcal{R}(\X_*,\W_*,\La_*,\W_*,\La_*) = \mathcal{L}_{\beta}(\X_*,\W_*,\La_*)= g(\W_*)+f(\X_*) +  \frac{1}{2\tau^2}\|\mathcal{P}_{\Omega}(\mathcal{A}(\X_*)) - \mathcal{P}_{\Omega}(\Y)\|_{\mathrm{F}}^2.
\]
\end{lemma}

\begin{proof}
Let $\{(\X_{k_j},\W_{k_j},\La_{k_j})\}_{j \ge 0}$ be a subsequence generated by   \eqref{scheme2} such that $(\X_{k_j},\W_{k_j},\La_{k_j}) \rightarrow (\X_*,\W_*,\La_*)$ as $j \rightarrow \infty$. According to Lemma \ref{thm:Xcov},  $\|\Delta \X_{k_j}\|_{\mathrm{F}} \rightarrow 0$ and $\|\mathcal{C}^{\#} \Delta \La_{k_j} \|_{\mathrm{F}} \le \|\mathcal{C}\| \|\Delta \La_{k_j}\|_{\mathrm{F}} \rightarrow 0 $ when $j \rightarrow \infty$. According to the definition of $\mathcal{R}_k$, it follows that
\[
\begin{aligned}
\lim_{j \rightarrow \infty} \mathcal{R}_{k_j} = \lim_{j \rightarrow \infty} \mathcal{L}_{\beta}(\X_{k_j},\W_{k_j},\La_{k_j}).
\end{aligned}
\]
Since $\|\Delta \La_{k_j+1}  \|_{\mathrm{F}} \rightarrow 0$ as $j \rightarrow \infty$, $\|\mathcal{C}\X_{k_j} + \mathcal{B}\W_{k_j} -\qb\|_{\mathrm{F}} \rightarrow 0$.
According to the definition of $\mathcal{L}_{\beta}$ and the fact that $\{ \La_{k_j}\}_{j \ge 0}$ is a bounded sequence we also have $\iprod{\La_{k_j}}{\mathcal{C}\X_{k_j}+\mathcal{B}\W_{k_j}-b }\rightarrow 0$ as $j \rightarrow \infty$. According to the fact that $f$ is continuous, it follows that:
\[
\begin{aligned}
&\lim_{j \rightarrow \infty} \mathcal{R}(\X_{k_j},\W_{k_j},\La_{k_j},\W_{k_j},\La_{k_j}) = \lim_{j \rightarrow \infty} \mathcal{L}_{\beta}(\X_{k_j},\W_{k_j},\La_{k_j}) \\
=& f(\X_*) + g(\W_*) +  \frac{1}{2\tau^2}\|\mathcal{P}_{\Omega}(\mathcal{A}(\X_*)) - \mathcal{P}_{\Omega}(\Y)\|_{\mathrm{F}}^2.
\end{aligned}
\]
\end{proof}



\begin{theorem}[Properties of limit point set]\label{lem:limpoint}
Suppose Assumption \ref{assum1} holds. Sequence $\{ (\X_k,\W_k,\La_k)\}_{k \ge 0}$ generated by \eqref{scheme2} satisfies the following properties:
\begin{enumerate}
  \item The limit point set of the sequence $\{(\X_k,\W_k,\La_k)\}$, denoted by $\omega \left(\{ (\X_k,\W_k,\La_k) \}_{k \ge 0} \right)$ is nonempty, and compact.
  \item $\lim_{k \rightarrow \infty} \text{dist} \left[(\X_k,\W_k,\La_k),\omega \left(\{ (\X_k,\W_k,\La_k) \}_{k \ge 0} \right) \right] = 0 . $
  \item $ \omega \left(\{ (\X_k,\W_k,\La_k) \}_{k \ge 0} \right) \subseteq  \text{crit}\, \mathcal{L}_{\beta}$.
\end{enumerate}
\end{theorem}

\begin{proof}
Since $\{ (\X_k,\W_k,\La_k)\}_{k \ge 0}$ is bounded by Theorem \ref{thm:Xbound}. There exists at least one limit point of $\{ (\X_k,\W_k,\La_k)\}_{k \ge 0}$ and $\omega \left(\{ (\X_k,\W_k,\La_k) \}_{k \ge 0} \right)$ is bounded. According to \cite[Chapter 2, Exercise 6]{rudin1964principles}, it follows that $\omega \left(\{ (\X_k,\W_k,\La_k) \}_{k \ge 0} \right)$ is closed. Hence $\omega \left(\{ (\X_k,\W_k,\La_k) \}_{k \ge 0} \right)$ is nonempty and compact. As a consequence, $\lim_{k \rightarrow \infty} \text{dist} \left[(\X_k,\W_k,\La_k),\omega \left(\{ (\X_k,\W_k,\La_k) \}_{k \ge 0} \right) \right] = 0 . $  It follows from Lemma \ref{crit} that $ \omega \left(\{ (\X_k,\W_k,\La_k) \}_{k \ge 0} \right) \subseteq  \text{crit}\, \mathcal{L}_{\beta}$, which completes the proof. 

\end{proof}


\begin{lemma}\label{lma-D}
Suppose that the Assumption \ref{assum1} holds. Let $\{(\X_k,\W_k,\La_k)\}_{k\ge 0}$ be a sequence generated by \eqref{scheme2}. Define
\begin{eqnarray*}
&\qs_{\X_{k}}:=\qd_{\X_k},\quad
\qs_{\W_{k}}:=\qd_{\W_k}+2r\varsigma_0 \Delta \W_{k},\quad
\qs_{\La_{k}}:=\qd_{\La_k}+2r\varsigma_2\mathcal{B}\mathcal{B}^{\#}\Delta \La_{k},&\\
&\qs_{{\W'}_{k}}:=-2r\varsigma_0 \Delta \W_{k},\quad
\qs_{{\La'}_k}:=-2\varsigma_2\mathcal{B}\mathcal{B}^{\#}\Delta \La_{k}&
\end{eqnarray*}
where $(\qd_{\X_k},\qd_{\W_k},\qd_{\La_k})\in\partial \mathcal L_\beta(\X_k,\W_k,\La_k)$.
Then
\[
\widetilde{\qs}_{k}:=(\qs_{\X_k}, \qs_{\W_k}, \qs_{\La_k}, \qs_{{\W'}_k},\qs_{{\La'}_k} )\in
\partial \mathcal R(\X_{k},\W_{k},\La_{k},\W_{k-1},\La_{k-1})
\] for $k\ge 1$, and it holds that
\begin{eqnarray}\label{nD}
|||\widetilde{\qs}_{k}|||\le \tilde \pi \Big(\|\Delta \X_{k}\|_{\mathrm{F}} +\|\Delta \W_{k}\|_{\mathrm{F}}+\|\Delta \La_{k}\|_{\mathrm{F}}\Big),
\end{eqnarray}
where
\begin{eqnarray}\label{trho}
\tilde \pi=\sqrt{3}\pi+4r\max\{\varsigma_0, \varsigma_2\|\mathcal{B}\|^2\},
\end{eqnarray}
$r>1$, $\pi$ is given in \eqref{def:rho}.
\end{lemma}

\begin{proof}
Let $k\ge 1$ fixed, and $(\qd_{\X_{k}}, \qd_{\W_{k}}, \qd_{\La_{k}})\in \partial \mathcal L_\beta(\X_k,\W_k,\La_k)$.
By taking partial derivatives of $\mathcal R_k$ with respect to $\X,\W,\La,\W',\La'$ we obtain
\begin{eqnarray*}
\qs_{\X_{k}}&:=& \partial_\X \mathcal R(\X_{k},\W_{k},\La_{k},\W_{k-1},\La_{k-1})
=\partial_\X \mathcal L_\beta(\X_{k},\W_{k},\La_{k})=\qd_{\X_{k}},\\
\qs_{\W_{k}}&:=&\nabla_\W \mathcal R (\X_{k},\W_{k},\La_{k},\W_{k-1},\La_{k-1})
=\nabla_\W \mathcal L_\beta(\X_{k},\W_{k},\La_{k})+2r\varsigma_0 \Delta \W_k=
\qd_{\W_{k}}+2r\varsigma_0 \Delta \W_k,\\
\qs_{\La_{k}}&:=&\nabla_\La \mathcal R(\X_{k},\W_{k},\La_{k},\W_{k-1},\La_{k-1})
=\nabla_{\La} \mathcal L_\beta(\X_{k},\W_{k},\La_{k})
+2r\varsigma_2 \mathcal{B}\mathcal{B}^{\#}\Delta \La_{k}=\qd_{\La_{k}}+2r\varsigma_2 \mathcal{B}\mathcal{B}^{\#}\Delta \La_{k},\\
\qs_{{\W'}_{k}}&:=&\nabla_{\W'} \mathcal R(\X_{k},\W_{k},\La_{k},\W_{k-1},\La_{k-1})
=-2r\varsigma_0\Delta \W_k,\\
\qs_{{\La'}_k}&:=&\nabla_{\La'} \mathcal R(\X_{k},\W_{k},\La_{k},\W_{k-1},\La_{k-1})
=-2r\varsigma_2 \mathcal{B}\mathcal{B}^{\#} \Delta \La_k.
\end{eqnarray*}
By the triangle inequality, we obtain
\begin{eqnarray*}
&\|\qs_{\X_{k}}\|_{\mathrm{F}}=\|\qd_{\X_k}\|_{\mathrm{F}},\quad
\|\qs_{\W_{k}}\|_{\mathrm{F}}\le \|\qd_{\W_k}\|_{\mathrm{F}}+2r\varsigma_0\|\Delta \W_k\|_{\mathrm{F}},\\
&\|\qs_{\La_{k}}\|_{\mathrm{F}}\le\|\qd_{\La_k}\|_{\mathrm{F}}+2r\varsigma_2\|\mathcal{B}\|^2 \|\Delta \La_k\|_{\mathrm{F}},\quad
\|\qs_{{\W'}_{k}}\|_{\mathrm{F}}= 2r\varsigma_0\|\Delta \W_k\|_{\mathrm{F}},\quad
\|\qs_{{\La'}_k}\|_{\mathrm{F}}=2r\varsigma_2 \|\mathcal{B}\|^2 \|\Delta \La_k\|_{\mathrm{F}}.
\end{eqnarray*}
By Lemma \ref{lem:sub:bound}, it follows that
\begin{eqnarray*}
|||\widetilde{\qs}_{k}|||&\le&  \|\qs_{\X_{k}}\|_{\mathrm{F}}+\|\qs_{\W_{k}}\|_{\mathrm{F}}+\|\qs_{\La_{k}}\|_{\mathrm{F}}+ \|\qs_{{\W'}_{k}}\|_{\mathrm{F}}+\|\qs_{{\La'}_k}\|_{\mathrm{F}}\\
&\le& \|\qd_{\X_k}\|_{\mathrm{F}}+  \|\qd_{\W_k}\|_{\mathrm{F}}+\|\qd^{k}_{\La}\|_{\mathrm{F}}+
4r\varsigma_0\|\Delta \W_{k}\|_{\mathrm{F}}+
4r\varsigma_2 \|\mathcal{B}\|^2 \|\Delta \La_k\|_{\mathrm{F}}
\\
&\le& \sqrt{3} |||\qd_k|||+4r\varsigma_0\|\Delta \W_{k}\|_{\mathrm{F}}+
4r\varsigma_2 \|\mathcal{B}\|^2 \|\Delta \La_k\|_{\mathrm{F}}\\
&\le& \sqrt{3}\rho \|\Delta \X_{k}\|_{\mathrm{F}}
+( \sqrt{3}\rho+4r\varsigma_0)\|\Delta \W_{k}\|_{\mathrm{F}}
+( \sqrt{3}\rho+4r\varsigma_2\|\mathcal{B}\|^2)\|\Delta \La_{k}\|_{\mathrm{F}},
\end{eqnarray*}
which completes the proof.
\end{proof}

We next prove the limiting behavior of $\mathcal{R}_k$, the detailed proof is similar to Theorem \ref{lem:limpoint} and we omit it. 
\begin{lemma}\label{lma-cluster2}
Suppose the Assumption \ref{assum1} holds.
If $\{(\X_k,\W_k,\La_k)\}_{k\ge 0}$ is a sequence generated by Algorithm \ref{algorithm:denoise},
then the following statements hold.
\begin{itemize}
\item [(i)] The set
$\varGamma :=  \omega \Big(\{(\X_k,\W_k,\La_k,\W_{k-1},\La_{k-1})\}_{k\ge 1}\Big)$
is nonempty, and compact.
\item [(ii)] $\varGamma \subseteq \{(\X, \W, \La,  \W, \La)
\in\mathbb R^n\times\mathbb R^m\times\mathbb R^p\times\mathbb R^m\times\mathbb R^p:
\;(\X, \W, \La)\in {\rm crit} (\mathcal L_\beta)\Big\}.$
\item [(iii)] $\lim_{k\to\infty} {\rm dist} \Big[ (\X_k, \W_k, \La_k, \W_{k-1}, \X_{k-1}), \varGamma \Big]=0$.
\item [(iv)] The sequences $\{\mathcal R_k\}_{k\ge 0}$, $\{\mathcal L_{\beta}(\X_k,\W_k,\La_k)\}_{k\ge 0}$ approach to the same limit and if $(\X_*, \W_*, \La_*,  \W_*, \La_*) \in\varGamma$,
then
\[
\mathcal R (\X_*, \W_*, \La_*,  \W_*, \La_*)=
 \mathcal L_{\beta}(\X_*, \W_*, \La_*)
= g(\W_*)+f(\X_*) +  \frac{1}{2\tau^2}\|\mathcal{P}_{\Omega}(\mathcal{A}(\X_*)) - \mathcal{P}_{\Omega}(\Y)\|_{\mathrm{F}}^2.
\]

\end{itemize}
\end{lemma}
\begin{proof}
The results follow immediately from Theorem \ref{thm:Xbound}, \ref{thm:Xcov} and Lemma \ref{lma-D}. 
\end{proof}

The following convergence result of the whole sequence, established under the Kurdyka-Łojasiewicz(KL) property, follows from \cite[Theorem 2]{yashtini2022convergence}. The difference lies in the construction of $\mathcal{R}_{k}$. The proofs are given for completeness. Furthermore, the rationality of the KL property \eqref{eqn:KL} assumption for Model \eqref{model:new} is clarified in Corollary \ref{corollary-1}. 



\begin{theorem}[Global Convergence] \label{thm:global-con}
Suppose that Assumption \ref{assum1} holds and  $\mathcal{R}$ defined in \eqref{definition:Rk} satisfies the KL property on the limit point set $\varGamma$, 
i.e.  for every $\mathbf{\dot{V}}_* =(\X_*,\W_*,\La_*,\W_*,\La_*) \in \varGamma $, there exists $\varepsilon >0$ and desingularizing function  $\psi: [0,\eta] \rightarrow [0,\infty),$ for some $\eta \in [0,\infty )$ such that for all $\mathbf{\dot{V}} = (\X,\W,\La,\W',\La')$ in the following set:
\begin{equation} \label{eqn:def:S}
\mathcal{S}:= \{\mathbf{\dot{V}}:
\text{dist}(\mathbf{\dot{V}},\varGamma)< \varepsilon\,\, \text{and}\,\, \mathcal{R}(\mathbf{\dot{V}}^*)< \mathcal{R}(\mathbf{\dot{V}}) < \mathcal{R}(\mathbf{\dot{V}}^*) + \eta \},
\end{equation}
the inequality 
\begin{equation} \label{eqn:KL}
\psi'(\mathcal{R}(\mathbf{\dot{V}}) - \mathcal{R}(\mathbf{\dot{V}}^*) )\text{dist}(0, \partial \mathcal{R}(\mathbf{\dot{V}}) ) \ge 1    
\end{equation}
 holds, then $ \{ \qT_k \}_{k \ge 0} := \{(\X_k,\W_k,\La_k)\}_{k \ge 0}$ satisfies the finite length property:
$
\sum_{k=0}^{\infty} \| \X_k\|_{\mathrm{F}} + \|\W_k\|_{\mathrm{F}} + \|\La_k\|_{\mathrm{F}} < \infty,
$
and consequently converges to a stationary point of \eqref{prob:conve}.
\end{theorem}

\begin{proof}
According to Lemma \ref{lem:bound}, $\lim_{k \rightarrow \infty} \mathcal{R}_k := \mathcal{R}_{\infty}$. It follows that the sequence $\{\mathcal{E}_k \}_{ k \ge k_0}$ defined by 
\[
\mathcal{E}_k := \mathcal{R}_k - \mathcal{R}_{\infty},
\]
is non-negative,  monotonically decreasing, and converges to 0. We consider the following two cases:

{\it Case 1.}  There exists $k_1\ge k_0$ such that $\mathcal E_{k_1}=0$.
Hence $\mathcal E_k=0$ for all $k\ge k_1$
and, it follows that 
$ \|\Delta  \X_{k+1}\|_{\mathrm{F}}^2+\|\Delta \W_{k+1}\|_{\mathrm{F}}^2+\|\Delta \La_{k+1}\|_{\mathrm{F}}^2 \le
 \frac{1}{a}(\mathcal E_k-\mathcal E_{k+1})=0,\quad \forall k\ge k_1$ by (\ref{ieqn:RK}).
This gives rise to
\begin{eqnarray*}
 \sum_{k\ge 0}\Big(\|\Delta  \X_{k+1}\|_{\mathrm{F}}
+\|\Delta \W_{k+1}\|_{\mathrm{F}}
+\|\Delta \La_{k+1}\|_{\mathrm{F}}\Big)
\le
\sum_{k=1}^{k_1} \Big(\|\Delta \X_{k}\|_{\mathrm{F}}
+\|\Delta \W_{k}\|_{\mathrm{F}}
+\|\Delta \La_{k}\|_{\mathrm{F}}\Big)<+\infty.
\end{eqnarray*}
The latter conclusion is due to the fact that the sequence is bounded from Theorem \ref{thm:Xbound}.

{\it Case 2.} The error sequence $\mathcal E_k=\mathcal R_k-\mathcal R_{\infty}>0$ for all $k\ge k_0$.
Then by (\ref{ieqn:RK}), it follows that
\begin{eqnarray}\label{s1}
|||\Delta \qT_{k+1}\||_{\mathrm{F}}^2 =\|\Delta  \X_{k+1}\|^2_{\mathrm{F}}+\|\Delta \W_{k+1}\|^2_{\mathrm{F}}+\|\Delta \La_{k+1}\|^2_{\mathrm{F}} \le
 \frac{1}{a}(\mathcal E_k-\mathcal E_{k+1}),\quad \forall k\ge k_0.
\end{eqnarray}
By Lemma \ref{lma-cluster2}, $\varGamma$ is nonempty, compact.
Furthermore, $\mathcal R_k$ takes on a
constant value $\mathcal R_{\infty}$ on $\varGamma$.

Since the sequence $\{\mathcal R_{k}\}_{k\ge k_0}$ is monotonically decreasing to $\mathcal R_{\infty}$, there exists $ k_1\ge k_0\ge 1$ such that
\[
\mathcal R_{\infty} <\mathcal R_k< \mathcal R_{\infty}+\eta,\quad \forall k\ge k_1.
\]
By Theorem \ref{lma-cluster2}, it follows that
$\lim_{k\to\infty} {\rm dist} \Big[ (\X_k,\W_k,\La_k,\W_{k-1},\X_{k-1}), \varGamma\Big]=0.$
Thus  there exists $k_2\ge 1$ such that
\[
{\rm dist} \Big[ (\X_k,\W_k,\La_k,\W_{k-1},\X_{k-1}), \varGamma\Big]<\varepsilon, \quad  \forall k\ge k_2.
\]
Let $\tilde k =\max\{ k_1,k_2,3 \}$. It follows that
 $(\X_k,\W_k,\La_k,\W_{k-1},\X_{k-1})\in\mathcal S$
for $k\ge \tilde k$, where $\mathcal S$ defined in (\ref{eqn:def:S}).
Thus, we have
\begin{eqnarray}\label{KL}
\psi'( \mathcal E_k)\cdot {\rm dist}\Big(0,
\partial \mathcal R_k \Big)
 \ge1.
\end{eqnarray}
Since $\psi$ is concave, it holds that
$
\psi(\mathcal E_{k}) -\psi(\mathcal E_{k+1})\ge
 \psi'(\mathcal E_{k}) (\mathcal E_{k} - \mathcal E_{k+1}).
$
Together with  (\ref{s1}) and (\ref{KL}), it follows that
\begin{eqnarray*}
|||\Delta \qT_{k+1}|||^2_{\mathrm{F}} &\le&
\psi'( \mathcal E_k) |||\Delta \qT_{k+1}|||^2_{\mathrm{F}} \cdot {\rm dist}( 0, \partial \mathcal R_k)\\
&\le&
\frac{1}{a} \psi'( \mathcal E_k)  (\mathcal E_{k} - \mathcal E_{k+1}) \cdot {\rm dist}(0, \partial \mathcal R_k ) \\
&\le& \frac{1}{a} \Big(\psi(\mathcal E_{k}) -\psi(\mathcal E_{k+1})\Big)
 \cdot {\rm dist}(0, \partial \mathcal R_k ).
\end{eqnarray*}
By the arithmetic mean-geometric mean inequality, for any $\gamma>0$, we  have
\[
|\|\Delta \qT_{k+1}\||_{\mathrm{F}}
\le \frac{\gamma}{2a}  \Big(\psi(\mathcal E_{k}) -\psi(\mathcal E_{k+1})\Big)
+\frac{1}{2\gamma}{\rm dist}(0, \partial \mathcal R_k ).
\]
It follows that
\begin{eqnarray}\label{main}
\|\Delta  \X_{k+1}\|_{\mathrm{F}}
+\|\Delta \W_{k+1}\|_{\mathrm{F}}
+\|\Delta \La_{k+1}\|_{\mathrm{F}}
\le \frac{\sqrt{3}\gamma}{2a}  \Big(\psi(\mathcal E_{k}) -\psi(\mathcal E_{k+1})\Big)
+\frac{\sqrt{3}}{2\gamma}{\rm dist}(0, \partial \mathcal R_k ).
\end{eqnarray}
Then according to Lemma \ref{lma-D}, we obtain
\begin{eqnarray}\label{import1}
\|\Delta  \X_{k+1}\|_{\mathrm{F}}
+\|\Delta \W_{k+1}\|_{\mathrm{F}}
+\|\Delta \La_{k+1}\|_{\mathrm{F}}
\le
\frac{\sqrt{3}\gamma}{2a}  \Big(\psi(\mathcal E_{k}) -\psi(\mathcal E_{k+1})\Big)
+
\frac{\sqrt{3}\tilde \rho}{2\gamma}\Big(\|\Delta \X_{k}\|_{\mathrm{F}}
+\|\Delta \W_{k}\|_{\mathrm{F}}
+\|\Delta \La_{k}\|_{\mathrm{F}}\Big).
\end{eqnarray}
According to the equality
$
\sum_{k=\underline k}^K\|\Delta \X_{k}\|_{\mathrm{F}} =
\sum_{k=\underline k}^{K}
\|\Delta  \X_{k+1}\|_{\mathrm{F}}
+\|\Delta \X_{\underline k}\|-\|\Delta \X_K\|_{\mathrm{F}}
$, we choose $\gamma>0$ large enough such that $1>\sqrt{3}\tilde \rho/2\gamma$, and
let $\delta_0=1-\frac{\sqrt{3}\tilde \rho}{2\gamma}$.
Summing up (\ref{import1}) from $k=\underline k\ge \tilde k $ to $K\ge \underline k$ gives
\begin{eqnarray*}
&\delta_0 (\sum_{k=\underline k}^K\|\Delta  \X_{k+1}\|_{\mathrm{F}}
+\|\Delta \W_{k+1}\|_{\mathrm{F}}
+\|\Delta \La_{k+1}\|_{\mathrm{F}})
\le
\frac{\sqrt{3}\gamma}{2a}
\Big(\psi(\mathcal E_{\underline k}) -\psi(\mathcal E_{K+1})\Big)
&\\
&+\frac{\sqrt{3}\tilde \rho}{2\gamma}
\Big(\|\Delta \X_{\underline k}\|_{\mathrm{F}}
+\|\Delta \W_{\underline k}\|_{\mathrm{F}}
+\|\Delta \La_{\underline k}\|_{\mathrm{F}}\Big)
-\frac{\sqrt{3}\tilde \rho}{2\gamma}
\Big(\|\Delta \X_{K}\|_{\mathrm{F}}
+\|\Delta \W_{K}\|_{\mathrm{F}}
+\|\Delta \La_{K}\|_{\mathrm{F}}\Big).&
\end{eqnarray*}
Recall that $\mathcal E_k$ is monotonically decreasing and
$\psi(\mathcal E_k)\ge \psi(\mathcal E_{k+1})>0$,  it follows that%
\begin{eqnarray*}
&\sum_{k=\underline k}^K\|\Delta  \X_{k+1}\|_{\mathrm{F}}
+\|\Delta \W_{k+1}\|_{\mathrm{F}}
+\|\Delta \La_{k+1}\|_{\mathrm{F}}
\le
\frac{\sqrt{3}\gamma}{2a\delta_0}
\psi(\mathcal E_{\underline k})+\frac{\sqrt{3}\tilde \rho}{2\gamma\delta_0}
\big(\|\Delta \X_{\underline k}\|_{\mathrm{F}}
+\|\Delta \W_{\underline k}\|_{\mathrm{F}}
+\|\Delta \La_{\underline k}\|_{\mathrm{F}}\big).
\end{eqnarray*}\label{finitelength}
The right hand side of this inequality is bounded for any $K\ge \underline k$. Let $K\to\infty$ and we obtain
\begin{eqnarray}\label{lambda0}
\sum_{k\ge\underline k}\|\Delta  \X_{k+1}\|_{\mathrm{F}}
+\|\Delta \W_{k+1}\|_{\mathrm{F}}
+\|\Delta \La_{k+1}\|_{\mathrm{F}}
\le
\frac{\sqrt{3}\gamma}{2a\delta_0}
\psi(\mathcal E_{\underline k})+\frac{\sqrt{3}\tilde \rho}{2\gamma\delta_0}
\Big(\|\Delta \X_{\underline k}\|_{\mathrm{F}}
+\|\Delta \W_{\underline k}\|_{\mathrm{F}}
+\|\Delta \La_{\underline k}\|_{\mathrm{F}}\Big).
\end{eqnarray}
Since $\{(\X_k,\W_k,\La_k)\}_{k\ge 0}$ is a bounded sequence,
for any $\underline k\in\mathbb N_+$, it follows that
\begin{eqnarray}\label{lambda}
\lambda({\underline k}):=\sum_{k=1}^{\underline k}
 \|\Delta \X_{k}\|_{\mathrm{F}}
+\|\Delta \W_{k}\|_{\mathrm{F}}
+\|\Delta \La_{k}\|_{\mathrm{F}}<+\infty.
\end{eqnarray}
 By combining (\ref{lambda0}) and (\ref{lambda}),
we conclude that $\displaystyle{\sum_{k\ge 0} \|\Delta  \X_{k+1}\|_{\mathrm{F}}
+\|\Delta \W_{k+1}\|_{\mathrm{F}}
+\|\Delta \La_{k+1}\|_{\mathrm{F}}}$ is finite.

Note that for any $p,q,K\in\mathbb N_+$ where $q\ge p>0$, it follows that:
\begin{eqnarray*}
|||\qT_q-\qT_p|||_{\mathrm{F}} &=&||| \sum_{k=p}^{q-1} \Delta \qT_{k+1}|||_{\mathrm{F}}
\le \sum_{k=p}^{q-1} |||\Delta \qT_{k+1}|||_{\mathrm{F}}\\
&\le&   \sum_{k=p}^{q-1}
\Big(\| \Delta  \X_{k+1}\|_{\mathrm{F}}
+ \| \Delta \W_{k+1}\|_{\mathrm{F}}
+ \| \Delta \La_{k+1}\|_{\mathrm{F}}\Big)\\
&\le&\sum_{k\ge 0} \|\Delta  \X_{k+1}\|_{\mathrm{F}}
+\|\Delta \W_{k+1}\|_{\mathrm{F}}
+\|\Delta \La_{k+1}\|_{\mathrm{F}} < \infty.
\end{eqnarray*}
This implies that $\{\qT_k\}_{k\ge 0}=\{(\X_k,\W_k,\La_k)\}_{k\ge 0}$ is a Cauchy sequence
and hence converges. Moreover, by Theorem \ref{lem:limpoint}, it converges to a stationary point. 
\end{proof}
\begin{corollary} \label{corollary-1}
Suppose the penalty parameter $\beta$ is large enough, and $L_{k,1}, L_{k,2}$ are set as given constants for all $k$ such that \eqref{cond:A4} and Assumption \eqref{assum1e} hold, then the sequence  $\{(\X_k,\W_k,\La_k)\}_{k\ge 0}$ generated by Algorithm \ref{algorithm:denoise} converges to a  stationary point of SLRQA \eqref{model:new}. 
\end{corollary}

\begin{proof}
It follows directly from the definition of SLRQA \eqref{model:new} that Assumption \ref{assum1a},\ref{assum1b} and \ref{assum1c} are all satisfied. Since \eqref{cond:A4} and Assumption \eqref{assum1e} hold, Assumption  \ref{assum1d}, \ref{assum1e} are also satisfied. Furthermore, the Huber function and Frobenius norm are semi-algebraic \cite[Example 2]{bolte2014proximal}. From \cite[Appendix, Lemma 4]{shang2016scalable}, it follows that the nuclear norm is also a semi-algebraic function. Since $\phi$ is a semi-algebraic function and the fact that the composition and sum of semi-algebraic functions are also semi-algebraic, we can conclude that $\mathcal{R}$ is semi-algebraic and the KL-inequality \ref{eqn:KL} holds since any proper closed semi-algebraic function satisfies the KL-inequality \cite[Theorem 3]{bolte2014proximal}. It follows from Theorem \ref{thm:global-con} that the sequence  $\{(\X_k,\W_k,\La_k)\}_{k\ge 0}$ generated by Algorithm \ref{algorithm:denoise} converges to a stationary point of SLRQA \eqref{model:new}, which completes the proof.
\end{proof}

\subsection{Global Convergence for Algorithm \ref{algorithm:inpaint}} \label{4-2}

The global convergence is established for a generalization of SLRQA-NF, i.e. Model \eqref{prob:conve2}.
The augmented Lagrangian function of \eqref{prob:conve2} is
\begin{align*}
& \mathcal{L}_{\beta_1,\beta_2}(\X,\W,\La_1,\La_2)  =  f(\X) + g(\W)  + \iprod{\La_1}{\mathcal{C}_1(\X) + \mathcal{B}_1(\W) - \mathbf{\dot{B}}_1 } + \iprod{\La_2}{\mathcal{B}_2(\W) -  \mathbf{\dot{B}}_2}  \\
& \qquad \qquad \qquad \qquad \qquad \qquad \quad  + \frac{\beta_1}{2} \|\mathcal{C}_1(\X) + \mathcal{B}_1(\W) - \mathbf{\dot{B}}_1 \|^2_{\mathrm{F}}+ \frac{\beta_2}{2}\|\mathcal{B}_2(\W) - \qb_2\|_{\mathrm{F}}^2\\
& \qquad \qquad = f(\X) + g(\W)    + \frac{\beta_1}{2}\left\|\mathcal{C}_1(\X) + \mathcal{B}_1(\W) - \mathbf{\dot{B}}_1 + \frac{\La_1}{\beta_1} \right\|_\mF^2  + \frac{\beta_2}{2}\left\|\mathcal{B}_2(\W) - \qb_2 + \frac{\La_2}{\beta_2} \right\|^2 \\
& \qquad \qquad \qquad \quad - \frac{1}{2\beta_1}\left\|\La_1\right\|_\mF^2 - \frac{1}{2\beta_2}\left\|\La_2\right\|_\mF^2 .
\end{align*}
The iteration scheme for the proposed PL-ADMM-NF is
\be \label{scheme2}
\begin{aligned}
\X_{k+1} & \in \argmin_{\X} f(\X)   + \frac{\beta_1}{2}\left\|\mathcal{C}_1(\X) + \mathcal{B}_1(\W_k) - \mathbf{\dot{B}}_1 + \frac{\La_{k,1}}{\beta_1} \right\|_\mF^2 + \frac{L_{k,1}}{2}\|\X-\X_{k} \|_{\mathrm{F}}^2 ,\\
\W_{k+1} & \in \argmin_{\W} \iprod{\nabla g(\W_k)}{\W} + \frac{\beta_1}{2}\left\|\mathcal{C}_1(\X_{k+1}) + \mathcal{B}_1(\W) - \mathbf{\dot{B}}_1 + \frac{\La_{k,1} }{\beta_1} \right\|_\mF^2 \\
& \qquad\qquad\qquad + \frac{\beta_2}{2}\| \mathcal{B}_2(\W) - \qb_2 + \frac{ \La_{k,2} }{\beta_2}\|_{\mathrm{F}}^2 + \frac{L_{k,2}}{2}\left\|\W-\W_k\right\|^2_{\mathrm{F}},\\
\La_{k+1,1} & = \La_{k,1} + \mu \beta_1(\mathcal{C}_1(\X_{k+1}) + \mathcal{B}_1(\W_{k+1}) - \mathbf{\dot{B}}_1 ), \\
\La_{k+1,2} & = \La_{k,2}+ \mu \beta_2  (\mathcal{B}_2(\W_{k+1} ) - \qb_2).
\end{aligned}
\ee
The following assumptions are made for the convergence of PL-ADMM-NF. 
\begin{assumption} \label{assum2}
\,
  \begin{enumerate}[label=B\arabic*]
    \item \label{assum2b} $f(\X)$ and $ g(\W)$ are bounded from below and have $L_f$ and $L_g$ Lipschitz continuous gradients respectively, i.e., for every $\X_1,\X_2,\W_1$, and $\W_2$, it holds that 
    \begin{equation} \label{eqn:Lip}
    \begin{aligned}
    &\|\nabla g(\W_1) - \nabla g(\W_2) \|_{\mathrm{F}} \le L_g \| \W_1 - \W_2 \|_{\mathrm{F}},  \\
    &\|\nabla f(\X_1) - \nabla f(\X_2) \|_{\mathrm{F}} \le L_f \|\X_1 - \X_2\|_{\mathrm{F}}.
    \end{aligned}
    \end{equation}
    \item \label{assum2c} The matrix $\mathcal{C}_1^{\#}\mathcal{C}_1 $ is full rank or $f(\X)$ is coercive,
    $ \mathcal{B}_2^{\#}\mathcal{B}_2$ is full rank, ${\rm range}(\mathcal{B}_1) \subseteq {\rm range}(\mathcal{C}_1)$, $\mathbf{\dot{B}}_1 \in {\rm range}(\mathcal{C}_1)$, and $\qb_2 \in {\rm range}(\mathcal{B}_2)$.
    \item \label{assum2d} The  parameters $\beta_1,\beta_2 > 0$, $\mu \in (0,2)$, and   there exist three constants $a_1 > 0, a_2 > 0$, and $r > 1$ such that
    \begin{equation} \label{cond2:A4}
(q_1^- - 3r\theta_{1}) \mathbf{I}  \succeq a_1 \mathbf{I}\,\quad \mbox{and}\,\quad q_{2}^- \mathbf{I} + \beta_2 \mathcal{B}_2^{\#}\mathcal{B}_2 + \beta_1 \mathcal{B}_1^{\#}\mathcal{B}_1 - (3r\theta_{2} + L_g ) \mathbf{I} \succeq a_2 \mathbf{I},
    \end{equation}
where $\mathbf{I}$ is the identity matrix with correct size and $\theta_1 = \theta_{1,1} + \theta_{2,1}, \theta_2 = \theta_{1,2} + \theta_{2,2}$,
\begin{equation} \label{theta1232}
\begin{aligned}
    & \theta_{1,1} =  \frac{4(q_1 + L_f)^2\mu}{\rho(\mu)^2 \beta_1  \lambda_+^{\mathcal{C}_1^{\#}\mathcal{C}_1}  }, \quad  \theta_{1,2} = \frac{4\beta_1 \|\mathcal{B}_1 \|^2 \|\mathcal{C}_1 \|^2 }{\rho(\mu) \lambda_+^{\mathcal{C}_1^{\#}\mathcal{C}_1} },  \\
    & \theta_{2,1} = 
   \frac{3\mu\|\mathcal{B}_1\|^2(4q_1^2 + |1-\mu|^2(4(q_1 + L_f)^2 ))}{\beta_2 \rho(\mu)^4  \lambda_+^{\mathcal{C}_1^{\#}\mathcal{C}_1}\lambda_+^{\mathcal{B}_2^{\#}\mathcal{B}_2 }
   } ,  \\
   & \theta_{2,2} =  \frac{6(L_g + q_2)^2 \mu }{\beta_2\rho(\mu)^2 \lambda_+^{\mathcal{B}_2^{\#}\mathcal{B}_2} } +  \frac{3\mu\|\mathcal{B}_1\|^2(|1-\mu|^2+1)(4\beta_1^2 \|\mathcal{B}_1 \|^2 \|\mathcal{C}_1 \|^2   )}{\beta_2\rho(\mu)^4  \lambda_+^{\mathcal{C}_1^{\#}\mathcal{C}_1}\lambda_+^{\mathcal{B}_2^{\#}\mathcal{B}_2 }
   }.
\end{aligned}
\end{equation}
    \item \label{assum2e} The parameters $\beta_1,\beta_2$ satisfies 
    $
    \frac{\kappa}{2L_f}  >  \theta_{3,0} + \theta_{4,0}, \frac{1}{2L_g}  >  \theta_{4,0} ,$ where $\kappa \in (0,1)$ is a given constant and 
\begin{equation} \label{theta340}
    \theta_{3,0} = \frac{3}{2\beta_1 \rho(\mu)\lambda_+^{\mathcal{C}_1^{\#}\mathcal{C}_1 } },\quad \theta_{4,0} = \max \left\{ \frac{3}{2\beta_2 \rho(\mu)\lambda_+^{\mathcal{B}_2^{\#}\mathcal{B}_2 } } , \frac{ 3 \|\mathcal{B}_1 \|^2 }{2\beta_2\rho(\mu)^2 \lambda_+^{\mathcal{B}_2^{\#}\mathcal{B}_2 } \lambda_+^{\mathcal{C}_1^{\#}\mathcal{C}_1 }  } \right\}.
\end{equation}
  \end{enumerate}
\end{assumption}

\begin{remark}
For \eqref{assum2b}, the differentiability of $f(\X)$ and Lipschitz continuity of $\nabla f(\X)$ in SLRQA-NF \eqref{model:new2} comes from the differentiability of $\phi$ and Lipschitz differentiability of $\nabla \phi$. For more discussions on the Lipschitz continuity for spectral function, we refer the readers to \cite[Corollary 2.5]{lewis1995convex} and \cite[Theorem 3.3]{ding2020spectral}.

It is straightforward to verify that \eqref{assum2c} holds for SLRQA-NF \eqref{model:new2}. As mentioned in Section \ref{1-2}, the combination of two constraints in \eqref{model:new2} with respect to $\W$ and $\Z$ can be represented by $(\mathbf{I};\mathbf{0} )$ and $(\mathbf{I} ;\mathbf{T} ) $, where $\mathbf{T}$ is the matrix form of $\mathcal{P}_{\Omega}(\mathcal{A}\mathcal{W}^{\#})$. In this case, \ref{assum1c} in Assumption \ref{assum1} does not hold for SLRQA-NF. However, this assumption is required for all known proofs for the global convergence of nonconvex ADMM and its variants \cite{wang2019global,yashtini2022convergence,liu2019linearized}.  Hence, existing techniques have difficulties proving the convergence of Algorithm \ref{algorithm:inpaint}. 

For \eqref{assum2d} and \eqref{assum2e}, 
if $q_1^-$ is large enough, then the first condition in \eqref{cond2:A4} is satisfied. Since 
$\mathcal{B}_2^{\#}\mathcal{B}_2$ is full rank, if $\beta_1$ and $\beta_2$ is large enough such that

\begin{equation}
\begin{aligned}
    & \beta_1 > \max\left\{ \frac{3L_f}{\kappa\rho(\mu) \lambda_+^{\mathcal{C}_1^{\#}\mathcal{C}_1 }} ,1 \right\} + \max \left\{ \frac{3L_f}{ \rho(\mu)\lambda_+^{\mathcal{B}_2^{\#}\mathcal{B}_2 } } , \frac{ 3L_f \|\mathcal{B}_1 \|^2 }{\rho(\mu)^2 \lambda_+^{\mathcal{B}_2^{\#}\mathcal{B}_2 } \lambda_+^{\mathcal{C}_1^{\#}\mathcal{C}_1 }  }\right\},  \\
    & \beta_2 >   \max \left\{ \frac{3L_g}{ \rho(\mu)(\lambda_+^{\mathcal{B}_2^{\#}\mathcal{B}_2 })^2 } , \frac{ 3L_g \|\mathcal{B}_1 \|^2 }{\rho(\mu)^2 (\lambda_+^{\mathcal{B}_2^{\#}\mathcal{B}_2 })^2 \lambda_+^{\mathcal{C}_1^{\#}\mathcal{C}_1 }  },1\right\} +  
   \frac{6L_g\mu\|\mathcal{B}_1\|^2(4q_1^2 + |1-\mu|(4(q_1 + L_f)^2 ))}{\beta_2 \rho(\mu)^4  \lambda_+^{\mathcal{C}_1^{\#}\mathcal{C}_1}(\lambda_+^{\mathcal{B}_2^{\#}\mathcal{B}_2 })^2
   } \\
   & \qquad +   \frac{12L_g(L_g + q_2)^2 \mu }{\beta_2\rho(\mu)^2 (\lambda_+^{\mathcal{B}_2^{\#}\mathcal{B}_2})^2 } +  \frac{6L_g\mu\|\mathcal{B}_1\|^2(|1-\mu|+1)(4\beta_1^2 \|\mathcal{B}_1 \|^2 \|\mathcal{C}_1 \|^2   )}{\beta_2\rho(\mu)^4  \lambda_+^{\mathcal{C}_1^{\#}\mathcal{C}_1}(\lambda_+^{\mathcal{B}_2^{\#}\mathcal{B}_2 })^2
   }, 
\end{aligned}
\end{equation}
hold, then \eqref{assum2d} and \eqref{assum2e} are satisfied. 
\end{remark}


Lemmas \ref{lem2:X} and \ref{lem2:W} present the sufficient decrease property of $\mathcal{L}_{\beta_1,\beta_2}$ for $\X$ and $\W$ update respectively and are used in Lemma~\ref{lem2:Y}. The proofs follow the same spirit in the existing analysis~\cite{yashtini2022convergence,wang2019global}.
\begin{lemma}[Sufficient descent of $\mathcal{L}_{\beta_1,\beta_2}$ for $\X$ update] \label{lem2:X}
 The sequence $\{ (\X_k,\W_k,\La_{k,1},\La_{k,2})\}_{k \ge 0}$ generated by   \eqref{scheme2} satisfies:
$$
\mathcal{L}_{\beta_1,\beta_2}(\X_{k},\W_{k},\La_{k,1},\La_{k,2}) - \mathcal{L}_{\beta_1,\beta_2}(\X_{k+1},\W_{k},\La_{k,1},\La_{k,2}) \geq \frac{L_{k,1}}{2} \|\Delta \X_{k+1}\|^2_{\mathrm{F}}. 
$$
\end{lemma}

\begin{proof}
According to the $\X_{k+1}$ update in  \eqref{scheme2}, the following equality holds:
\be \label{Wup22}
\begin{aligned}
&f(\X_{k+1}) - f(\X_k) + \frac{\beta_1}{2}\|\mathcal{C}_1(\X_{k+1})+\mathcal{B}_1(\W_{k}) - \mathbf{\dot{B}}_1 + \La_{k,1}/\beta_1 \|_\mF^2 + \frac{L_{k,1}}{2}\| \Delta \X_{k+1} \|_{\mathrm{F}}^2 \\
\leq & \frac{\beta_1}{2}\left\|\mathcal{C}_1(\X_k) + \mathcal{B}_1(\W_k)- \mathbf{\dot{B}}_1 + \La_{k,1}/\beta_1\right\|_\mF^2.
\end{aligned}
\ee
It follows that
$$
\begin{aligned}
&\mathcal{L}_{\beta_1,\beta_2}(\X_{k},\W_{k},\La_{k,1},\La_{k,2} ) - \mathcal{L}_{\beta_1,\beta_2}(\X_{k+1},\W_{k},\La_{k,1},\La_{k,2}) \\
=&  f(\X_k) - f(\X_{k+1})   - \frac{\beta_1}{2}\|\mathcal{C}_1(\X_{k+1}) +\mathcal{B}_1(\W_{k}) - \qb_1 +\La_{k,1}/\beta_1 \|_{\mathrm{F}}^2  \\
& + \frac{\beta_1}{2} \|\mathcal{C}_1(\X_{k}) + \mathcal{B}_1(\W_{k}) - \qb_1 + \La_{k,1}/\beta_1 \|_{\mathrm{F}}^2 \ge \frac{L_{k,1}}{2} \|\Delta \X_{k+1}\|_{\mathrm{F}}^2,
\end{aligned}
$$
which completes the proof.
\end{proof}

\begin{lemma}[Sufficient descent of $\mathcal{L}_{\beta_1,\beta_2}$ for $\W$ update] \label{lem2:W}
Suppose \ref{assum2b} in Assumption \ref{assum2}  holds.  The sequence $\{ (\X_k,\W_k,\La_{k,1}, \La_{k,2})\}_{k \ge 0}$ generated by~\eqref{scheme2} satisfies:
$$
\mathcal{L}_{\beta_1,\beta_2}(\X_{k+1},\W_{k},\La_{k,1},\La_{k,2}) - \mathcal{L}_{\beta_1,\beta_2}(\X_{k+1},\W_{k+1},\La_{k,1},\La_{k,2}) \geq \|\Delta \W_{k+1}\|^2_{B_k},  
$$
where $B_{k} :=L_{k,2}\mathbf{I} - L_g \mathbf{I} + \frac{\beta_1}{2} \mathcal{B}_1\mathcal{B}_1^{\#} +  \frac{\beta_2}{2} \mathcal{B}_2\mathcal{B}_2^{\#} $.
\end{lemma}

\begin{proof}
According to the optimality condition of $\W_{k+1},$ the following equality holds:
\begin{equation} \label{upW:2}
\begin{aligned}
 & \nabla g(\W_k) + \beta_1 \mathcal{C}_1^{\#}(\mathcal{C}_1(\X_{k+1}) + \mathcal{B}_1(\W_{k+1}) - \qb_1 + \La_{k,1}/\beta_1 ) \\
 &\qquad \qquad \qquad \qquad + \beta_2 \mathcal{B}_2^{\#}(\mathcal{B}_2(\W_{k+1}) - \qb_2 + \La_{k,2}/\beta_2 ) + L_{k,2} \Delta \W_{k+1} = 0.   
 \end{aligned}
\end{equation}
It follows that
$$
\begin{aligned}
&\mathcal{L}_{\beta_1,\beta_2}(\X_{k+1},\W_{k},\La_{k,1},\La_{k,2}) - \mathcal{L}_{\beta_1,\beta_2}(\X_{k+1},\W_{k+1},\La_{k,1},\La_{k,2})  \\
=&   \frac{\beta_1}{2} \|\mathcal{C}_1(\X_{k+1}) + \mathcal{B}_1(\W_k)-\qb_1 + \frac{\La_k}{\beta_1} \|_{\mathrm{F}}^2 - \frac{\beta_1}{2} \|\mathcal{C}_1(\X_{k+1}) + \mathcal{B}_1(\W_{k+1})-\qb_1 + \frac{\La_k}{\beta_1} \|_{\mathrm{F}}^2 \\
& + g(\W_k) - g(\W_{k+1}) + \frac{\beta_2}{2} \left\|\mathcal{B}_2(\W_{k}) - \qb_2 + \frac{\La_{k,2}}{\beta_2}  \right\|_{\mathrm{F}}^2  - \frac{\beta_2}{2} \left\|\mathcal{B}_2(\W_{k+1}) - \qb_2 + \frac{\La_{k,2}}{\beta_2}  \right\|_{\mathrm{F}}^2 \\
=& g(\W_k) - g(\W_{k+1})  - \beta_1 \iprod{\mathcal{C}_1^{\#}( \mathcal{C}_1(\X_{k+1}) + \mathcal{B}_1(\W_{k+1}) -\qb_1 + \frac{\La_{k,1}}{\beta_1} ) }{\Delta \W_{k+1}}  \\
& -\beta_2 \iprod{\mathcal{B}_2^{\#}(\mathcal{B}_2(\W_{k+1}) - \qb_2 + \frac{\La_{k,2}}{\beta_2} ) }{\Delta \W_{k+1}} + \frac{\beta_1}{2}\|\mathcal{B}_1(\Delta \W_{k+1})\|_{\mathrm{F}}^2 + \frac{\beta_2}{2} \|\mathcal{B}_2(\Delta \W_{k+1}) \|_{\mathrm{F}}^2 \\
\overset{\eqref{upW:2}}{=}& g(\W_k) - g(\W_{k+1}) +  \iprod{\nabla g(\W_k)}{\Delta \W_{k+1}}  + L_{k,2} \|\Delta \W_{k+1}\|_{\mathrm{F}}^2 \\
&\quad\quad\quad\quad\quad\quad\quad\quad\quad\quad\quad\quad\quad\quad\quad\quad\quad + \frac{\beta_1}{2}\|\mathcal{B}_1\Delta \W_{k+1}\|_{\mathrm{F}}^2 + \frac{\beta_2}{2}\|\mathcal{B}_2 \Delta \W_{k+1} \|_{\mathrm{F}}^2 \, \\
\ge & \|\Delta \W_{k+1} \|_{B_{k}}^2, \qquad (\text{by Assumption~\ref{assum2b}} )
\end{aligned}
$$
which completes the proof.
\end{proof}

Lemma~\ref{lem2:Y} is used in Lemma~\ref{lem2:mon} for constructing a merit function.
\begin{lemma} \label{lem2:Y}
Suppose \ref{assum2b} in Assumption \ref{assum2}  holds. The sequence $\{ (\X_k,\W_k,\La_{k,1}, \La_{k,2})\}_{k \ge 0}$ generated by~\eqref{scheme2} satisfies:

$$
\begin{aligned}
&\mathcal{L}_{\beta_1,\beta_2}(\X_{k+1},\W_{k+1},\La_{k+1,1},\La_{k+1,2} ) \leq \mathcal{L}_{\beta_1,\beta_2}(\X_{k},\W_{k},\La_{k,1},\La_{k,2}) \\
& \qquad\qquad\qquad\qquad - \|\Delta \W_{k+1}\|^2_{B_k} - \frac{L_{k,1}}{2} \|\Delta \X_{k+1} \|^2_{\mathrm{F}} + \frac{1}{\beta_1 \mu} \| \Delta \La_{k+1,1} \|^2_{\mathrm{F}} +  \frac{1}{\beta_2 \mu} \| \Delta \La_{k+1,2} \|^2_{\mathrm{F}}.
\end{aligned}
$$
\end{lemma}

\begin{proof}
According to the update of $\La_1$ and $\La_2$ in \eqref{scheme2}, Lemma \ref{lem2:X}, and Lemma \ref{lem2:W}, we have
$$
\begin{aligned}
& \mathcal{L}_{\beta_1,\beta_2}(\X_{k+1},\W_{k+1},\La_{k+1,1},\La_{k+1,2}) \\
=& \mathcal{L}_{ \beta_1,\beta_2}(\X_{k+1},\W_{k+1},\La_{k,1},\La_{k,2}) + \frac{1}{\beta_1 \mu} \| \Delta \La_{k+1,1}  \|^2_{\mathrm{F}}  + \frac{1}{\beta_2 \mu}\|\Delta \La_{k+1,2} \|_{\mathrm{F}}^2\\
\leq& \mathcal{L}_{\beta_1,\beta_2}(\X_{k},\W_{k},\La_{k,1},\La_{k,2}) -\|\Delta \W_{k+1} \|_{B_k}^2 - \frac{L_{k,1}}{2} \|\Delta \X_{k+1} \|_{\mathrm{F}}^2 \\
& + \frac{1}{\beta_1 \mu} \| \Delta \La_{k+1,1} \|^2_{\mathrm{F}}+ \frac{1}{\beta_2 \mu} \| \Delta \La_{k+1,2} \|^2_{\mathrm{F}}. \\
\end{aligned}
$$
\end{proof}
The proofs in Lemma~\ref{lem2:mon} and Theorem~\ref{thm2:Xbound} are the key differences to the existing convergence analyses of ADMM-type algorithms. Lemma~\ref{lem2:mon} establishes upper bounds of $\|\Delta \La_{k+1,1} \|$ and $\| \Delta \La_{k+1,2} \|$ for constructing a merit function and Theorem~\ref{thm2:Xbound} proves the boundedness of the sequence generated by~\eqref{scheme2}.
Lemma~\ref{lem2:mon} and Theorem~\ref{thm2:Xbound} are proven without using the widely-used ``range assumption'' but the newly proposed assumption in \ref{assum2c} of Assumption \ref{assum2}.

\begin{lemma} \label{lem2:mon}
Suppose \ref{assum2b} and \ref{assum2c} in Assumption \ref{assum2}  hold, the sequence $\{(\X_{k},\W_{k},\La_{k,1},\La_{k,2})\}_{k \ge 0}$ generated by  \eqref{scheme2} satisfies:
\begin{equation} \label{eqn2:La1}
\begin{aligned}
\frac{1}{\beta_1 \mu}\|\Delta \La_{k+1,1}\|^2_{\mathrm{F}} \le& \theta_{1,1} (\|\Delta \X_{k+1} \|^2_{\mathrm{F}} +  \|\Delta \X_{k} \|^2_{\mathrm{F}})+ \theta_{1,2}(\|\Delta \W_{k+1} \|_{\mathrm{F}}^2 +  \|\Delta \W_{k} \|_{\mathrm{F}}^2 )  \\
& + \theta_{1,3} (\|\mathcal{C}_1^{\#} (\Delta \La_{k,1}) \|^2_{\mathrm{F}} - \|\mathcal{C}_1^{\#} (\Delta \La_{k+1,1}) \|^2_{\mathrm{F}}),  \\
\frac{1}{\beta_2 \mu} \|\Delta \La_{k+1,2}\|^2_{\mathrm{F}} \le&  \theta_{2,1} (\|\Delta \X_{k+1} \|_{\mathrm{F}}^2 +  \|\Delta \X_{k} \|_{\mathrm{F}}^2 + \|\Delta \X_{k-1} \|_{\mathrm{F}}^2 ) \\
+ \theta_{2,2}( \|\Delta \W_k& \|_{\mathrm{F}}^2 + \|\Delta \W_{k+1} \|_{\mathrm{F}}^2  + \|\Delta \W_{k-1} \|_{\mathrm{F}}^2)  - \theta_{2,3}( \| \mathcal{C}_1^{\#} \Delta \La_{k+1,1}\|_{\mathrm{F}}^2 - \| \mathcal{C}_1^{\#} \Delta \La_{k,1}\|_{\mathrm{F}}^2) \\
 - \theta_{2,3}( \| \mathcal{C}_1^{\#} & \Delta \La_{k,1}\|_{\mathrm{F}}^2 - \| \mathcal{C}_1^{\#} \Delta \La_{k-1,1}\|_{\mathrm{F}}^2)  - \theta_{2,4}(\|\mathcal{B}_2^{\#}\Delta \La_{k+1,2} \|_{\mathrm{F}}^2 - \|\mathcal{B}_2^{\#}\Delta \La_{k,2} \|_{\mathrm{F}}^2 ) .
\end{aligned}
\end{equation}
Furthermore, the following inequality holds:
\begin{equation}
\begin{aligned}
& \mathcal{L}_{\beta_1,\beta_2}(\X_{k+1},\W_{k+1},\La_{k+1,1},\La_{k+1,2} )  + \| \Delta \X_{k+1} \|^2_{ \frac{L_{k,1}}{2} - r\theta_{1} \mathbf{I} }  + r\theta_1 \|\Delta \X_k \|_{\mathrm{F}}^2 + \|\Delta \W_{k+1}\|_{B_{k} - r \theta_{2} \mathbf{I}}^2  \\
& + r\theta_{2} \|\Delta \W_k \|_{\mathrm{F}}^2   + r \theta_3 \| \mathcal{C}_1^{\#} \Delta \La_{k+1,1} \|^2_{\mathrm{F}} + r \theta_4 \| \mathcal{C}_1^{\#} \Delta \La_{k,1} \|^2_{\mathrm{F}} + r\theta_5 \|\mathcal{B}_2^{\#}\Delta \La_{k+1,2} \|_{\mathrm{F}}^2 \\
& + \frac{r-1}{\beta_1 \mu} \|\Delta \La_{k+1,1}\|^2_{\mathrm{F}} + \frac{r-1}{\beta_2 \mu} \|\Delta \La_{k+1,2}\|^2_{\mathrm{F}} \\
\leq & \mathcal{L}_{\beta_1,\beta_2}(\X_k,\W_k,\La_{k,1}, \La_{k,2}) + 2r\theta_1\|\Delta \X_{k} \|_{\mathrm{F}}^2 + r\theta_1 \|\Delta \X_{k-1} \|_{\mathrm{F}}^2 + 2r \theta_2 \| \Delta \W_k\|^2_{\mathrm{F}} \\
& +   r \theta_2 \| \Delta \W_{k-1}\|^2_{\mathrm{F}}   + r \theta_3 \| \mathcal{C}_1^{\#} \Delta \La_{k,1}  \|^2_{\mathrm{F}} + r \theta_4 \| \mathcal{C}_1^{\#} \Delta \La_{k-1,1}  \|^2_{\mathrm{F}} + r \theta_5 \| \mathcal{B}_2^{\#} \Delta \La_{k,2}  \|^2_{\mathrm{F}}  , \label{eqn:L}
\end{aligned}
\end{equation}
where $r > 1,   \theta_3 = \theta_{1,3} + \theta_{2,3}, \theta_4 = \theta_{2,4}, \theta_5 = \theta_{2,5}  ,\theta_{1}, \theta_{2}$ are defined in \eqref{theta1232}, and 
\begin{equation} \label{thete11}
\begin{aligned}
& \theta_{1,3} = \frac{|1-\mu| }{\beta_1 \mu \lambda_+^{\mathcal{C}_1^{\#}\mathcal{C}_1}\rho(\mu)},   \theta_{2,3} = \frac{3|1-\mu| \|\mathcal{B}_1 \|^2 }{\beta_2\rho(\mu)^3 \mu \lambda_+^{\mathcal{B}_2^{\#}\mathcal{B}_2} \lambda_+^{\mathcal{C}_1^{\#}\mathcal{C}_1} },  \\
& \theta_{2,4}= \frac{3|1-\mu|^3 \|\mathcal{B}_1 \|^2 }{\beta_2\rho(\mu)^3 \mu \lambda_+^{\mathcal{B}_2^{\#}\mathcal{B}_2} \lambda_+^{\mathcal{C}_1^{\#}\mathcal{C}_1} }
    ,\theta_{2,5} = \frac{|1-\mu| }{\beta_2 \rho(\mu) \mu \lambda_+^{\mathcal{B}_2^{\#}\mathcal{B}_2 } }. 
\end{aligned}
\end{equation}
\end{lemma}

\begin{proof}
 Let $k \ge 1$ be fixed and define the matrix
\begin{equation} \label{def2:qw}
    \dH_{k+1,1} := -\nabla f(\X_{k+1}) - L_{k,1}\Delta \X_{k+1} + \beta_1 \mathcal{C}_1^{\#}\mathcal{B}_1\Delta \W_{k+1}, \quad \dH_{k+1,2} := - L_{k,2} \Delta \W_{k+1} - \nabla g(\W_k) .   
\end{equation}
Hence the following inequality holds:
\begin{align*}
& \Delta \dH_{k+1,1} = L_{k-1,1} \Delta \X_{k} - L_{k,1} \Delta \X_{k+1} + \nabla f(\X_{k}) - \nabla f(\X_{k+1}) + \beta_1  \mathcal{C}_1^{\#}\mathcal{B}_1 (\Delta \W_{k+1} - \Delta \W_{k}) , \\
& \Delta \dH_{k+1,2} = L_{k-1,2} \Delta \W_k - L_{k,2} \Delta \W_{k+1} + \nabla g(\W_{k-1}) - \nabla g(\W_k).
\end{align*}
It follows from the triangle inequality that:
$$
\begin{aligned}
& \|\Delta \dH_{k+1,1}\|_{\mathrm{F}} \le  (L_{k,1} + L_f)\| \Delta \X_{k+1} \|_{\mathrm{F}} + L_{k-1,1}\| \Delta \X_{k} \|_{\mathrm{F}} + \beta_1  \|\mathcal{B}_1\| \|\mathcal{C}_1\|(\| \Delta \W_{k+1}\|_{\mathrm{F}} + \| \Delta \W_k\|_{\mathrm{F}} ) , \\
& \|\Delta \dH_{k+1,2}\|_{\mathrm{F}} \le \|\nabla g(\W_{k-1}) - \nabla g(\W_k) \|_{\mathrm{F}} + L_{k,2} \|\Delta \W_{k+1}\|_{\mathrm{F}} + L_{k-1,2} \|\Delta \W_k\|_{\mathrm{F}}.
\end{aligned}
$$
By \ref{assum2b} in Assumption \ref{assum2}, $\nabla g(\W)$ and $\nabla f(\X)$ are $L_g$ and $L_f$ Lipschitz continuous respectively. According to $q_1 = \sup_{k \ge 0} L_{k,1} < \infty, q_2 = \sup_{k \ge 0} L_{k,2} < \infty$, we have
$
\|\Delta \dH_{k+1,2}\|_{\mathrm{F}} \le (L_g + q_2)\|\Delta \W_k\|_{\mathrm{F}} + q_2 \|\Delta\W_{k+1}\|_{\mathrm{F}}.
$
Hence it follows that
\begin{equation} \label{eqn2:qw}
\begin{aligned}
&\|\Delta \dH_{k+1,1}\|^2_{\mathrm{F}} \le 4(q_1 + L_f)^2\|\Delta \X_{k+1} \|_{\mathrm{F}}^2 + 4q_1^2 \|\Delta \X_k \|_{\mathrm{F}}^2 \\
&\qquad\qquad\qquad\qquad\qquad\qquad\qquad+ 4\beta_1^2 \|\mathcal{B}_1\|^2 \|\mathcal{C}_1 \|^2(\| \Delta \W_{k+1}\|_{\mathrm{F}}^2 + \|\Delta \W_k \|_{\mathrm{F}}^2 ), \\
&\|\Delta \dH_{k+1,2}\|^2_{\mathrm{F}} \le 2(L_g + q_2)^2 \|\Delta \W_k\|^2_{\mathrm{F}} + 2q_2^2 \|\Delta \W_{k+1}\|^2_{\mathrm{F}}.
\end{aligned}
\end{equation}
Expressing the optimality condition of $\X$ and $\W$ subproblems using $\dH_{k+1,1}$ and $\dH_{k+1,2}$ respectively, we have
$$
\begin{aligned}
& \dH_{k+1,1} = \beta_1 \mathcal{C}_1^{\#}(\mathcal{C}_1(\X_{k+1}) + \mathcal{B}_1(\W_{k+1})- \qb_1 + \La_{k,1}/\beta_1 ) , \\
& \dH_{k+1,2} = \beta_1 \mathcal{B}_1^{\#}(\mathcal{C}_1(\X_{k+1}) + \mathcal{B}_1(\W_{k+1})- \qb_1 + \La_{k,1}/\beta_1 ) + \beta_2 \mathcal{B}_2^{\#}( \mathcal{B}_2(\W_{k+1}) - \qb_2 + \La_{k,2}/\beta_2 ).
\end{aligned}
$$
Combining this with the $\La_1$ and $\La_2$ updates, it follows that
\begin{equation} \label{eqn2:La}
\begin{aligned}
    \mathcal{C}_1^{\#} \La_{k+1,1} &= \mu \dH_{k+1,1} + (1-\mu) \mathcal{C}_1^{\#}\La_{k,1}, \\
    \mathcal{B}_2^{\#} \La_{k+1,2} &= \mu \dH_{k+1,2} + (1-\mu)\mathcal{B}_2^{\#}\La_{k,2} - \mathcal{B}_1^{\#}(\Delta \La_{k+1,1}) - \mu\mathcal{B}_1^{\#} ( \La_{k,1}).
    \end{aligned}
\end{equation}
Since $\mu \in (0,2),$ we have
\begin{equation} \label{eqn:CLa}
\begin{aligned}
\mathcal{C}_1^{\#} \Delta \La_{k+1,1} &=   \rho(\mu)\frac{\mu}{\rho(\mu)} \Delta \dH_{k+1,1} + |1-\mu| (\text{sign}(1-\mu)\mathcal{C}_1^{\#}\Delta \La_{k,1})   , \\
\mathcal{B}_2^{\#} \Delta \La_{k+1,2} &= \frac{\rho(\mu)}{3}\frac{3\mu}{\rho(\mu)} \Delta \dH_{k+1,2} + |1-\mu|\text{sign}(1-\mu) \mathcal{B}_2^{\#} \Delta \La_{k,2} \\
& - \frac{\rho(\mu)}{3}\frac{3}{\rho(\mu)}\mathcal{B}_1^{\#}(\Delta \La_{k+1,1}) +\frac{\rho(\mu)}{3}\frac{3|1-\mu|\text{sign}(1-\mu) }{\rho(\mu)} (\mathcal{B}_1^{\#}(\Delta \La_{k,1})  ).
\end{aligned}
\end{equation}
By the  convexity of $\|\cdot\|^2_{\mathrm{F}}$, the update of $\La_1$ and $\La_2$ in \eqref{scheme2} and \ref{assum2c} in Assumption \ref{assum2}, it follows that $\Delta \La_{k+1,1} \in \text{{\rm range}}(\mathcal{C}_1)$, $\Delta \La_{k+1,2} \in \text{{\rm range}}(\mathcal{B}_2)$.   The following inequalities hold:
\begin{equation} \label{eqn2:4-6-1}
\begin{aligned}
&\lambda_+^{\mathcal{C}_1^{\#}\mathcal{C}_1}\rho(\mu) \|\Delta\La_{k+1,1}\|^2_{\mathrm{F}} \overset{\eqref{assum2c}}{\leq} \rho(\mu)\|\mathcal{C}_1^{\#} \Delta \La_{k+1,1}\|_{\mathrm{F}}^2 = \|\mathcal{C}_1^{\#} \Delta \La_{k+1,1}\|_{\mathrm{F}}^2 - |1-\mu| \|\mathcal{C}_1^{\#} \Delta \La_{k+1,1}\|^2_{\mathrm{F}} \nonumber \\
\overset{\eqref{eqn:CLa}}{\le}& \frac{\mu^2}{\rho(\mu)} \|\Delta \dH_{k+1,1}\|^2_{\mathrm{F}} + |1-\mu| \|\mathcal{C}_1^{\#}\Delta \La_{k,1}\|^2_{\mathrm{F}} - |1-\mu| \|\mathcal{C}_1^{\#} \Delta \La_{k+1,1}\|^2_{\mathrm{F}},
\end{aligned}
\end{equation}
and 
\begin{equation} \label{eqn2:4-6-2}
\begin{aligned}
& \lambda_+^{\mathcal{B}_2^{\#}\mathcal{B}_2}\rho(\mu) \|\Delta \La_{k+1,2} \|_{\mathrm{F}}^2 \overset{\eqref{assum2c} }{\le} \rho(\mu)\|\mathcal{B}_2^{\#} \La_{k+1,2} \|_{\mathrm{F}}^2 
\overset{\eqref{eqn2:4-6-1}}{\le} \frac{3\mu^2}{\rho(\mu)} \|\Delta \dH_{k+1,2}\|^2_{\mathrm{F}} + |1-\mu| \|\mathcal{B}_2^{\#}\Delta \La_{k,2}\|^2_{\mathrm{F}}  \\
& +  \frac{3}{\rho(\mu)} \|\mathcal{B}_1^{\#}\Delta \La_{k+1,1} \|^2 + \frac{3|1-\mu|^2}{\rho(\mu)}  \|\mathcal{B}_1^{\#} \Delta \La_{k,1} \|_{\mathrm{F}}^2 - |1-\mu| \|\mathcal{B}_2^{\#} \Delta \La_{k+1,2}\|^2_{\mathrm{F}} \\
 \overset{\eqref{eqn:CLa}}{\le} & \frac{3\mu^2}{\rho(\mu)} \|\Delta \dH_{k+1,2}\|^2_{\mathrm{F}} + |1-\mu| \|\mathcal{B}_2^{\#}\Delta \La_{k,2}\|^2_{\mathrm{F}} - |1-\mu| \|\mathcal{B}_2^{\#} \Delta \La_{k+1,2}\|^2_{\mathrm{F}}\\
 & + \frac{3\mu^2}{\rho(\mu)^3 \lambda_+^{\mathcal{C}_1^{\#}\mathcal{C}_1}}\|\mathcal{B}_1\|^2\|\Delta \dH_{k+1,1} \|_{\mathrm{F}}^2 + \frac{3|1-\mu| \|\mathcal{B}_1 \|^2 }{\rho(\mu)^2 \lambda_+^{\mathcal{C}_1^{\#}\mathcal{C}_1 } } \|\mathcal{C}_1^{\#}\Delta \La_{k,1} \|_{\mathrm{F}}^2  -  \frac{3|1-\mu|\|\mathcal{B}_1\|^2 }{\rho(\mu)^2\lambda_+^{\mathcal{C}_1^{\#}\mathcal{C}_1 } } \|\mathcal{C}_1^{\#}\Delta \La_{k+1,1} \|_{\mathrm{F}}^2 \\
 & + \frac{3\mu^2|1-\mu|^2\|\mathcal{B}_1\|^2}{\rho(\mu)^3 \lambda_+^{\mathcal{C}_1^{\#}\mathcal{C}_1 } } \|\Delta \dH_{k,1} \|^2_{\mathrm{F}} + \frac{3|1-\mu|^3 \|\mathcal{B}_1\|^2 }{\rho(\mu)^2\lambda_+^{\mathcal{C}_1^{\#}\mathcal{C}_1 } }\|\mathcal{C}_1^{\#} \Delta \La_{k-1,1}\|_{\mathrm{F}}^2 -  \frac{3|1-\mu|^3 \|\mathcal{B}_1\|^2 }{\rho(\mu)^2 \lambda_+^{\mathcal{C}_1^{\#}\mathcal{C}_1 } }\|\mathcal{C}_1^{\#} \Delta \La_{k,1}\|^2 \\
  \overset{\eqref{eqn2:qw}}{\le} & \frac{3\mu^2}{\rho(\mu)} \left( 2(L_g + q_2)^2 \|\Delta  \W_k\|_{\mathrm{F}}^2 + 2q_2^2 \|\Delta \W_{k+1} \|_{\mathrm{F}}^2  \right) \\
  & + \frac{3\mu^2\|\mathcal{B}_1\|^2}{\rho(\mu)^3 \lambda_+^{\mathcal{C}_1^{\#}\mathcal{C}_1}}( 4(q_1 + L_f)^2\|\Delta \X_{k+1} \|_{\mathrm{F}}^2 + 4q_1^2 \|\Delta \X_k \|_{\mathrm{F}}^2 + 4\beta_1^2 \|\mathcal{B}_1\|^2 \|\mathcal{C}_1 \|^2(\| \Delta \W_{k+1}\|_{\mathrm{F}}^2 + \|\Delta \W_k \|_{\mathrm{F}}^2 )) \\
  & + \frac{3\mu^2|1-\mu|^2\|\mathcal{B}_1\|^2}{\rho(\mu)^3 \lambda_+^{\mathcal{C}_1^{\#}\mathcal{C}_1 } } ( 4(q_1 + L_f)^2\|\Delta \X_{k} \|_{\mathrm{F}}^2 + 4q_1^2 \|\Delta \X_{k-1} \|_{\mathrm{F}}^2) \\
  &+ 4\beta_1^2 \|\mathcal{B}_1\|^2 \|\mathcal{C}_1 \|^2(\| \Delta \W_{k}\|_{\mathrm{F}}^2  + \|\Delta \W_{k-1} \|_{\mathrm{F}}^2 )  +   \frac{3|1-\mu| \|\mathcal{B}_1\|^2 }{\rho(\mu)^2\lambda_+^{\mathcal{C}_1^{\#}\mathcal{C}_1 } } (\|\mathcal{C}_1^{\#}\Delta \La_{k,1} \|_{\mathrm{F}}^2 - \|\mathcal{C}_1^{\#}\Delta \La_{k+1,1} \|_{\mathrm{F}}^2 )  \\
   & +|1-\mu| \left(\|\mathcal{B}_2^{\#}\Delta \La_{k,2}\|^2_{\mathrm{F}} -  \|\mathcal{B}_2^{\#} \Delta \La_{k+1,2}\|^2_{\mathrm{F}} \right)  + \frac{3|1-\mu|^3 \|\mathcal{B}_1\|^2 }{\rho(\mu)^2\lambda_+^{\mathcal{C}_1^{\#}\mathcal{C}_1 } } \left(\|\mathcal{C}_1^{\#} \Delta \La_{k-1,1}\|_{\mathrm{F}}^2 -  \|\mathcal{C}_1^{\#} \Delta \La_{k,1}\|_{\mathrm{F}}^2 \right).
\end{aligned}
\end{equation}
Consequently, Inequality \eqref{eqn:L}  follows from \eqref{eqn2:qw},~\ref{assum2c} in Assumption~\ref{assum2}, and the definition of $\lambda_+^{\mathcal{C}_1^{\#}\mathcal{C}_1}$ and $\lambda_+^{\mathcal{B}_2^{\#}\mathcal{B}_2}$.
Finally, Inequality~\eqref{eqn:L} follows from multiplying Inequality \eqref{eqn2:La1} by $r >1$ and combining it with Lemma \ref{lem2:Y}.
 \end{proof}

Let $\mathcal{T}(\X,\W,\La_1,\La_2,\X',\W',\La_1',\La_2',\X'',\W'',\La_1''):= \mathcal{L}_{\beta_1,\beta_2}(\X,\W,\La_1,\La_2)  + 2r\theta_1 \|\X-\X' \|_{\mathrm{F}}^2+ r\theta_1 \|\X'-\X'' \|_{\mathrm{F}}^2 + 2r\theta_2 \|\W-\W'\|^2_{\mathrm{F}} + r\theta_2 \|\W' -\W'' \|_{\mathrm{F}}^2  + r\theta_3\|\mathcal{C}_1^{\#}(\La_1-\La_1') \|^2_{\mathrm{F}} + r\theta_4\|\mathcal{C}_1^{\#}(\La_1'-\La_1'') \|^2_{\mathrm{F}} + r\theta_5\|\mathcal{B}_2^{\#}(\La_2-\La_2') \|^2_{\mathrm{F}} $. The merit function $\mathcal{T}_k$ is defined by
\begin{equation} \label{definition2:Rk}
\begin{aligned}
&\mathcal{T}_k :=\mathcal{T}(\X_k,\W_k,\La_{k,1},\La_{k,2},\X_{k-1}, \W_{k-1},\La_{k-1,1},\La_{k-1,2},\X_{k-2},\W_{k-2},\La_{k-2,1} )  \\
& = \mathcal{L}_{\beta_1,\beta_2}(\X_k,\W_k,\La_{k,1}, \La_{k,2}) +2r\theta_1\|\Delta \X_k\|^2_{\mathrm{F}} +  r\theta_1\|\Delta \X_{k-1}\|^2_{\mathrm{F}}  +2r\theta_2\|\Delta \W_k\|^2_{\mathrm{F}} +r\theta_2\|\Delta \W_{k-1}\|^2_{\mathrm{F}}  \\
& + r\theta_3 \|\mathcal{C}_1^{\#}\Delta\La_{k,1}\|^2_{\mathrm{F}} + r \theta_4 \|\mathcal{C}_1^{\#}\Delta \La_{k-1,1} \|_{\mathrm{F}}^2 +  r\theta_5 \|\mathcal{B}_2^{\#}\Delta\La_{k,2}\|^2_{\mathrm{F}}.
\end{aligned}
\end{equation}
It follows from \eqref{eqn:L}, the definition of $\mathcal{T}_k$, and  Assumption~\ref{assum2} that  
\be \label{ieqn2:RK}
\begin{aligned}
& \mathcal{T}_{k+1} + a\left(\|\Delta \X_{k+1} \|_{\mathrm{F}}^2 + \|\Delta \W_{k+1} \|_{\mathrm{F}}^2 + \|\Delta \La_{k+1,1}\|^2_{\mathrm{F}}  + \|\Delta \La_{k+1,2}\|^2_{\mathrm{F}} \right)  \leq \mathcal{T}_{k+1} + \|\Delta \W_{k+1}\|_{B_{k} - 3r \theta_2 \mathbf{I}}^2 \\
&+ \| \Delta \X_{k+1} \|^2_{\frac{L_{k,1}}{2} - 3r\theta_1 \mathbf{I} } + \frac{r-1}{\beta_1 \mu} \|\Delta \La_{k+1,1}\|^2_{\mathrm{F}}  + \frac{r-1}{\beta_2 \mu} \|\Delta \La_{k+1,2}\|^2_{\mathrm{F}}
\le \mathcal{T}_k \le \mathcal{T}_{k_0},
\end{aligned}
\ee
where $k \ge k_0, a = \min \{a_1,a_2,\frac{r-1}{\beta_1\mu}, \frac{r-1}{\beta_2 \mu} \}$ and $a_1,a_2$ are defined in Assumption~\ref{assum2}, the first inequality follows from \ref{assum2d} and \ref{assum2e} in  Assumption \ref{assum2}, the second inequality follows from Lemma \ref{lem2:mon}, and the third inequality is due to the induction of $\mathcal{T}_{k+1} \leq \mathcal{T}_{k}$ for any $k \ge k_0$.

\begin{theorem}[Bounded sequence of $\{(\X_k,\W_k,\La_{k,1}, \La_{k,2}) \}_{k \ge 0}$] \label{thm2:Xbound}
Assume Assumption \ref{assum2} holds.  The sequence $\{(\X_k,\W_k,\La_{k,1}, \La_{k,2})\}_{k \ge 0}$ generated by  \eqref{scheme2} is bounded. Consequently, the limiting point set of $\{(\X_k,\W_k,\La_{k,1}, \La_{k,2})\}_{k \ge 0}$ is nonempty.
\end{theorem}

\begin{proof}
According to \eqref{ieqn2:RK}, there exists $k_0 \ge 1$ such that $\mathcal{T}_{k+1} \le \mathcal{T}_{k_0}$ for all $k \ge k_0.$ Therefore, it holds that
\begin{equation}  \label{ieqn2:RK2}
\begin{aligned}
& f(\X_{k+1}) + g(\W_{k+1}) + \frac{\beta_1}{2}\left\|\mathcal{C}_1(\X_{k+1}) + \mathcal{B}_1(\W_{k+1}) - \mathbf{\dot{B}}_1 + \frac{\La_{k+1,1}}{\beta_1} \right\|_\mF^2 \\
&+ \frac{\beta_2}{2} \left\|\mathcal{B}_2(\W_{k+1}) - \qb_2 + \frac{\La_{k+1,2}}{\beta_2} \right\|_{\mathrm{F}}^2  - \frac{1}{2\beta_1}\left\|\La_{k+1,1}\right\|_\mF^2 - \frac{1}{2\beta_2}\left\|\La_{k+1,2}\right\|_\mF^2  \\
& +(r \theta_2 + a)  (\|\Delta \W_{k+1}\|_{\mathrm{F}}^2  ) +  (r\theta_1 +a) \left(\|\Delta \X_{k+1}\|_{\mathrm{F}}^2 + \|\Delta \X_{k}\|_{\mathrm{F}}^2 \right)  + a( \|\Delta \La_{k+1,1} \|_{\mathrm{F}}^2 + \|\Delta \La_{k+1,2} \|_{\mathrm{F}}^2   ) \\
& + r \theta_3 \|\mathcal{C}_1^{\#} \Delta \La_{k+1,1}\|_{\mathrm{F}}^2 + r\theta_4 \|\mathcal{C}_1^{\#} \Delta \La_{k,1}\|_{\mathrm{F}}^2
+ r \theta_5 \|\mathcal{B}_2^{\#} \Delta \La_{k+1,2} \|_{\mathrm{F}}^2 \le \mathcal{T}_{k_0}. 
\end{aligned}
\end{equation}
It follows from \eqref{eqn2:La} that 
\begin{equation} \label{eqn:muLa2}
\begin{aligned}
    \mathcal{C}_1^{\#} \La_{k+1,1} &= \mu \dH_{k+1,1} + (1-\mu) \mathcal{C}_1^{\#}\La_{k,1}, \\
    \mathcal{B}_2^{\#} \La_{k+1,2} &= \mu \dH_{k+1,2} + (1-\mu)\mathcal{B}_2^{\#}\La_{k,2} - (1-\mu)\mathcal{B}_1^{\#}(\Delta \La_{k+1,1}) - \mu\mathcal{B}_1^{\#} ( \La_{k+1,1}).
    \end{aligned}
\end{equation}
Since $\mu \in (0,2),$  Equation \eqref{eqn:muLa2} can be rewritten as:
\begin{equation}
\begin{aligned}
\mu \mathcal{C}_1^{\#} \La_{k+1,1} =& \rho(\mu) \frac{ \mu \dH_{k+1}}{\rho(\mu)} + |1-\mu|\left(\text{sign}(1-\mu) \mathcal{C}_1^{\#}(\La_k - \La_{k+1}) \right), \\
 \mu \mathcal{B}_2^{\#} \La_{k+1,2} =& \frac{\rho(\mu)}{3}\frac{3\mu}{\rho(\mu)}  \dH_{k+1,2} + |1-\mu|\text{sign}(1-\mu)\mathcal{B}_2^{\#} \Delta \La_{k+1,2} \\
&- \frac{\rho(\mu)}{3}\frac{3(1-\mu)}{\rho(\mu)}\mathcal{B}_1^{\#}(\Delta \La_{k+1,1})  -\frac{\rho(\mu)}{3}\frac{3\mu}{\rho(\mu)} (\mathcal{B}_1^{\#}( \La_{k+1,1})  ).
\end{aligned}
\end{equation}
It follows from the  convexity of $\|\cdot\|^2$ that
$$
\begin{aligned}
\lambda_+^{\mathcal{C}_1^{\#}\mathcal{C}_1} \mu^2 \|\La_{k+1,1}\|_{\mathrm{F}}^2 \le& \frac{\mu^2}{\rho(\mu)} \|\dH_{k+1,1}\|_{\mathrm{F}}^2 + |1-\mu| \|\mathcal{C}_1^{\#}\Delta \La_{k+1,1}\|_{\mathrm{F}}^2, \\
 \lambda_+^{\mathcal{B}_2^{\#}\mathcal{B}_2} \mu^2 \|\La_{k+1,2}\|_{\mathrm{F}}^2 \le& \frac{3\mu^2}{\rho(\mu)} \|\dH_{k+1,2} \|_{\mathrm{F}}^2 + |1-\mu|\|\mathcal{B}_2^{\#}\Delta \La_{k+1,2}\|_{\mathrm{F}}^2 \\
& + \frac{3|1-\mu|^2}{\rho(\mu)}\|\mathcal{B}_1^{\#} \Delta \La_{k+1,1} \|_{\mathrm{F}}^2 + \frac{3\mu^2}{\rho(\mu)}\|\mathcal{B}_1^{\#}\La_{k+1,1} \|_{\mathrm{F}}^2.
\end{aligned}
$$
According to \eqref{def2:qw}, it holds that
$$
\begin{aligned}
& \|\dH_{k+1,1}\|_{\mathrm{F}}^2 \le 3 q_1^2 \|\Delta \X_{k+1}\|_{\mathrm{F}}^2 + 3\|\nabla f(\X_{k+1})\|_{\mathrm{F}}^2 + 3\beta_1^2 \|\mathcal{C}_1\|^2 \|\mathcal{B}_1\|^2 \|\Delta \W_{k+1} \|_{\mathrm{F}}^2, \\
& \|\dH_{k+1,2}\|_{\mathrm{F}}^2 \le 2(q_2 + L_g)^2 \|\Delta \W_{k+1}\|_{\mathrm{F}}^2 + 2\|\nabla g(\W_{k+1})\|_{\mathrm{F}}^2.
\end{aligned}
$$
Hence it follows that
\begin{equation} \label{eqn2:La3}
\begin{aligned}
& -\frac{1}{2\beta_1} \|\La_{k+1,1}\|_{\mathrm{F}}^2 \ge -\theta_{3,0} \| \nabla f (\X_{k+1})\|_{\mathrm{F}}^2 - \theta_{3,1} \|\Delta \X_{k+1}\|_{\mathrm{F}}^2  - \theta_{3,2} \|\Delta \W_{k+1} \|^2 - \theta_{3,3} \|\mathcal{C}_1^{\#} \Delta \La_{k+1,1}\|_{\mathrm{F}}^2, \\
&-\frac{1}{2\beta_2} \|\La_{k+1,2}\|_{\mathrm{F}}^2 \ge -\theta_{4,0} \|\nabla g(\W_{k+1})\|_{\mathrm{F}}^2 -\theta_{4,0} \|\nabla f(\X_{k+1})\|_{\mathrm{F}}^2   - \theta_{4,2} \|\Delta \W_{k+1}\|_{\mathrm{F}}^2 - \theta_{4,2} \|\Delta \W_{k}\|_{\mathrm{F}}^2 \\
& - \theta_{4,5} \|\mathcal{B}_2^{\#} \Delta \La_{k+1,2}\|_{\mathrm{F}}^2 - \theta_{4,3} \| \mathcal{C}_1^{\#}\Delta \La_{k+1,1} \|_{\mathrm{F}}^2 + \theta_{4,4} \| \mathcal{C}_1^{\#}\Delta \La_{k,1} \|_{\mathrm{F}}^2 - \theta_{4,1}\|\Delta \X_k \|^2 - \theta_{4,1} \|\Delta \X_{k+1} \|_{\mathrm{F}}^2 , 
\end{aligned}
\end{equation}
where 
$$
\begin{aligned}
& \theta_{3,0} = \frac{3}{2\beta_1 \rho(\mu)\lambda_+^{\mathcal{C}_1^{\#}\mathcal{C}_1 } }, ~~\theta_{3,1} =  \frac{3q_1^2}{2\beta_1 \rho(\mu) \lambda_+^{\mathcal{C}_1^{\#}\mathcal{C}_1}} ,~~ \theta_{3,2} = \frac{3\beta_1 \|\mathcal{C}_1 \|^2 \|\mathcal{B}_1 \|^2 }{2 \rho(\mu) \lambda_+^{\mathcal{C}_1^{\#}\mathcal{C}_1 } }, \theta_{3,3} = \frac{|1-\mu|}{ 2 \beta_1 \mu^2 \lambda_+^{\mathcal{C}_1^{\#}\mathcal{C}_1}}, \\
& \theta_{4,0} = \max \left\{ \frac{3}{\beta_2 \rho(\mu)\lambda_+^{\mathcal{B}_2^{\#}\mathcal{B}_2 } } , \frac{ 9 \|\mathcal{B}_1 \|^2 }{2\beta_2\rho(\mu)^2 \lambda_+^{\mathcal{B}_2^{\#}\mathcal{B}_2 } \lambda_+^{\mathcal{C}_1^{\#}\mathcal{C}_1 }  } \right\}, \\
& \theta_{4,1} = \frac{9q_1^2\|\mathcal{B}_1\|^2 }{2\beta_2 \rho(\mu)^2\lambda_+^{\mathcal{C}_1^{\#}\mathcal{C}_1 } \lambda_+^{\mathcal{B}_2^{\#}\mathcal{B}_2 } } + \frac{6|1-\mu|^2 \|\mathcal{B}_1 \|^2 (q_1 + L_f)^2}{\beta_2 \lambda_+^{\mathcal{C}_1^{\#}\mathcal{C}_1} \lambda_+^{\mathcal{B}_2^{\#}\mathcal{B}_2} \rho(\mu)^3 } ,  \\
&\theta_{4,2} = \frac{3(L_g + q_2)^2 }{\beta_2\rho(\mu) \lambda_+^{\mathcal{B}_2^{\#}\mathcal{B}_2} } +  \frac{6 \|\mathcal{B}_1\|^2(|1-\mu|^2)(\beta_1^2 \|\mathcal{B}_1 \|^2 \|\mathcal{C}_1 \|^2   )}{\beta_2\rho(\mu)^3  \lambda_+^{\mathcal{C}_1^{\#}\mathcal{C}_1}\lambda_+^{\mathcal{B}_2^{\#}\mathcal{B}_2 } 
   } + \frac{9 \beta_1^2 \|\mathcal{B}_1 \|^4  \|\mathcal{C}_1 \|^2   }{2 \beta_2 \rho(\mu)^2 \lambda_+^{\mathcal{C}_1^{\#}\mathcal{C}_1}\lambda_+^{\mathcal{B}_2^{\#}\mathcal{B}_2} } ,  \\
  & \theta_{4,3} = \frac{3|1-\mu| \|\mathcal{B}_1\|^2 }{2\beta_2\rho(\mu)^2 \mu^2 \lambda_+^{\mathcal{B}_2^{\#}\mathcal{B}_2 }\lambda_+^{\mathcal{C}_1^{\#}\mathcal{C}_1 } }, \theta_{4,4} =  \frac{3|1-\mu|^3 \|\mathcal{B}_1\|^2 }{2\beta_2\rho(\mu)^2 \mu^2 \lambda_+^{\mathcal{B}_2^{\#}\mathcal{B}_2 }\lambda_+^{\mathcal{C}_1^{\#}\mathcal{C}_1 } }  ,\theta_{4,5} = \frac{|1-\mu|}{ 2 \beta_2 \mu^2 \lambda_+^{\mathcal{B}_2^{\#}\mathcal{B}_2}}.
\end{aligned}
$$
Define $\theta_6 = \theta_{3,0} + \theta_{4,0}, \theta_7 = \theta_{3,1} + \theta_{4,1}, \theta_8 = \theta_{3,2} + \theta_{4,2}, \theta_9 = \theta_{3,3} + \theta_{4,3}, \theta_{10} = \theta_{4,4}, \theta_{11} = \theta_{4,5}$.
Using  Inequalities \eqref{eqn2:La3} and \eqref{ieqn2:RK2}, we obtain
\begin{align}
& (1-\kappa)f(\X_{k+1})  + \frac{\beta_1}{2}\left\|\mathcal{C}_1(\X_{k+1}) + \mathcal{B}_1(\W_{k+1}) - \mathbf{\dot{B}}_1 + \frac{\La_{k+1,1}}{\beta_1} \right\|_\mF^2 + \frac{\beta_2}{2} \left\|\mathcal{B}_2(\W_{k+1}) - \qb_2 + \frac{\La_{k+1,2}}{\beta_2} \right\|_{\mathrm{F}}^2 \nonumber \\
&+ a (\|\Delta \La_{k+1,1}\|_{\mathrm{F}}^2 + \|\Delta \La_{k+1,2}\|_{\mathrm{F}}^2) + (r \theta_2 + a -\theta_8) \|\Delta \W_{k+1}\|_{\mathrm{F}}^2 + (r \theta_1 + a -\theta_7) \left(\|\Delta \X_{k+1}\|_{\mathrm{F}}^2  \right) \nonumber \\
& + (r \theta_3 - \theta_9) \|\mathcal{C}_1^{\#} \Delta \La_{k+1,1}\|_{\mathrm{F}}^2  + (r \theta_4 - \theta_{10} )\|\mathcal{C}_1^{\#}\Delta \La_{k,1} \|_{\mathrm{F}}^2 + (r \theta_5 - \theta_{11} ) \|\mathcal{B}_2^{\#} \Delta \La_{k+1,2}\|_{\mathrm{F}}^2 \nonumber \\
& \le \mathcal{T}_{k_0} - \inf_{\W} \left\{  g(\W) -  \theta_{4,0} \|\nabla g(\W)\|_{\mathrm{F}}^2 \right\} -  \inf_{\X} \left\{ \kappa f(\X) -  \theta_{6} \|\nabla f(\X)\|_{\mathrm{F}}^2 \right\},  \label{eqn:ieqn2:RK3}
\end{align}
where $\kappa \in (0,1)$.
According to \eqref{eqn:Lip} in \ref{assum2b} of Assumption \ref{assum2},
 setting $\X_1 = \X - \delta_2 \nabla f(\X)$ and $\X_2 = \X $, it follows that $\kappa f(\X_k - \delta_1 \nabla f(\X_k)) \le \kappa f(\X_k) - \kappa(\delta_2 - \frac{L_f \delta_2^2}{2})\|\nabla f(\X_k)\|_{\mathrm{F}}^2  .$ Since $g$ and $f$ are bounded from below, there exists $M$ such that
\begin{equation} \label{ineq2:delta1}
\begin{aligned}
    & -M < \inf\{ g(\W) - (\delta_1 - \frac{L_g \delta_1^2}{2})\| \nabla g(\W)\|_{\mathrm{F}}^2: \W \in \mathbb{H}^{m \times n} \}, \\
    & -M < \inf\{\kappa f(\X) - \kappa(\delta_2 - \frac{L_f \delta_2^2}{2})\|\nabla f(\X)\|_{\mathrm{F}}^2: \X \in \mathbb{H}^{m \times n} \}.
    \end{aligned}
\end{equation}
We choose $\delta_1 = \frac{1}{L_g}, \delta_2 = \frac{1}{L_f}$. According to \ref{assum2e} in Assumption \ref{assum2}, we have $\theta_{4,0} < \frac{1 }{2 L_g} =  (\delta_1 - \frac{L_g \delta_1^2}{2}), \theta_{6} < \frac{\kappa}{2 L_f} = \kappa(\delta_2 - \frac{L_f \delta_2^2}{2}).$  Since $r > 1$ and $\mu \in (0,2),$ according to the denifition of $\theta_i, i =1,\cdots,9$, it holds that $r \theta_1 + a - \theta_7 >0$,$r\theta_2 + a - \theta_8 >0, r\theta_3 - \theta_9 > 0, r \theta_4 - \theta_{10} > 0, r \theta_5 - \theta_{11} > 0$. It follows from \eqref{eqn:ieqn2:RK3} that
\begin{equation} \label{L2:bound}
\begin{aligned}
& (1 - \kappa) f(\X_{k+1}) + \frac{\beta_2}{2} \|\mathcal{B}_2(\W_{k+1}) - \qb_2 + \frac{\La_{k+1,2}}{\beta_2} \|_{\mathrm{F}}^2 \\
& + \frac{\beta_1}{2}\left\|\mathcal{C}_1(\X_{k+1}) + \mathcal{B}_1(\W_{k+1}) - \mathbf{\dot{B}}_1 + \frac{\La_{k+1,1}}{\beta_1} \right\|_\mF^2  + a( \|\Delta \La_{k+1,1}\|_{\mathrm{F}}^2 + \|\Delta \La_{k+1,2}\|_{\mathrm{F}}^2 )< \mathcal{T}_{k_0} + 2M.
\end{aligned}
\end{equation}
 According to \eqref{L2:bound}, $\Delta \La_{k+1,1}, \Delta \La_{k+1,2} $ are bounded. By the $\La_2$ update in \eqref{scheme2}, we have $\mathcal{B}_2\W_{k+1} = \frac{1}{\beta_2 \mu} \Delta\La_{k+1,2} + \qb_2$. If $\mathcal{B}_2^{\#}\mathcal{B}_2$ is full rank, $\{ \W_k \}$ is bounded.
By the $\La_1$ update in \eqref{scheme2}, we have $ \mathcal{C}_1\X_{k+1}=  \frac{1}{\beta_1 \mu}\Delta \La_{k+1,1} - \mathcal{B}_1\W_{k+1} +\qb_1.$ 
 Since $\{\mathcal{B}_1\W_k\}_{k \ge k_0}$ is bounded, if  $\mathcal{C}_1^{\#}\mathcal{C}_1$ is full rank or $f(\X_{k+1})$ is coercive, it follows that $\{ \X_k \}_{k \ge k_0}$ is bounded. 
From the fact that $ \frac{\beta_1}{2}\left\|\mathcal{C}_1(\X_{k+1}) + \mathcal{B}_1(\W_{k+1}) - \mathbf{\dot{B}}_1 + \frac{\La_{k+1,1}}{\beta_1} \right\|_\mF^2$ and $ \frac{\beta_2}{2}\left\|\mathcal{B}_2(\W_{k+1}) - \mathbf{\dot{B}}_2 + \frac{\La_{k+1,2} }{\beta_2} \right\|_\mF^2$  are bounded, it follows that $\{\La_{k_1} \}_{k \ge k_0}$ and $\{\La_{k_2} \}_{k \ge k_0}$  are bounded. As a consequence, $\{\X_k\}_{k \ge 1}, \{\W_k\}_{k \ge 1}, \{\La_{k,1}\}_{k \ge 1},\{\La_{k,2}\}_{k \ge 1}$ are bounded. 
\end{proof}

Lemma \ref{lem2:sub:bound} is used in Lemma \ref{crit2} to show the limiting behavior of sequence $\{(\X_{k},\W_{k},\La_{k,1},\La_{k,2} )\}_{k \ge 0}$.
\begin{lemma}[Subgradient bound] \label{lem2:sub:bound}
Suppose that \ref{assum2b} in Assumption holds.
 Let $\{(\X_k,\W_k,\La_{k,1},\La_{k,2} )\}_{k \ge 0}$ be a sequence generated by~\eqref{scheme2}. Then 
 \[
 \widetilde{\qd}_{k+1}:= (\qd_{\X_{k+1}},\qd_{\W_{k+1}},\qd_{\La_{k+1,1}},\qd_{\La_{k+1,2}}) \in \partial \mathcal{L}_{\beta_1,\beta_2}(\X_{k+1},\W_{k+1},\La_{k+1,1},\La_{k+1,2} ), 
 \]
  where
$$
\begin{aligned}
\qd_{\X_{k+1,1}} &:=  \mathcal{C}_1^{\#} \Delta \La_{k+1} - L_{k,1} \Delta \X_{k+1}  - \beta_1 \mathcal{C}_1^{\#}\mathcal{B}_1 \Delta \W_{k+1} ,\quad \qd_{\La_{k+1}} := \frac{1}{\beta_1 \mu} \Delta \La_{k+1,1} ,\, \\
\qd_{\W_{k+1}} &:= \nabla g(\W_{k+1}) - \nabla g(\W_k) + \mathcal{C}_1^{\#}\Delta \La_{k+1}  - L_{k,2} \Delta \W_{k+1},\quad \qd_{\La_{k+1,2}} := \frac{1}{\beta_2 \mu} \Delta  \La_{k+1,2}.
\end{aligned}
$$
Furthermore, it holds that
$$
|\|\widetilde{\qd}_{k+1}\|| \leq \pi \left( \| \Delta \W_{k+1}\|_{\mathrm{F}} + \| \Delta \X_{k+1}\|_{\mathrm{F}}+ \| \Delta \La_{k+1,1}\|_{\mathrm{F}} +  \| \Delta \La_{k+1,2}\|_{\mathrm{F}} \right),
$$
where 
$
 \pi = \max \left\{q_1, \beta_1 \|\mathcal{C}_1\| \|\mathcal{B}_1\| + q_2 + L_g, \| \mathcal{C}_1\| + \|\mathcal{B}_1\| + \frac{1}{\beta_1\mu}, \|\mathcal{B}_2 \| + \frac{1}{\beta_2 \mu} \right\}.   
$

\end{lemma}
\begin{proof}
According to the optimality condition of the update of $\X_{k+1},$ it follows that
$$
 - \beta_1 \mathcal{C}_1^{\#}(\mathcal{C}_1(\X_{k+1}) + \mathcal{B}_1(\W_{k}) - \mathbf{\dot{B}}_1 + \La_{k,1}/\beta_1) - L_{k,1} \Delta \X_{k+1} = \nabla f(\X_{k+1}),
$$
which yields
\begin{equation} \label{eqn2:dX}
\qd_{\X_{k+1}} =  - \mathcal{C}_1^{\#} \Delta \La_{k+1,1} - L_{k,1} \Delta \X_{k+1} + \beta_1 \mathcal{C}_1^{\#}\mathcal{B}_1 \Delta \W_{k+1} \in \partial_{\X} \mathcal{L}_{\beta_1,\beta_2}(\X_{k+1},\W_{k+1},\La_{k+1,1},\La_{k+1,2}).
\end{equation}
 According to the optimality condition of the update of $\W_{k+1},$ the following equality holds:
 
 \begin{equation} \label{eqn:02W}
 \nabla g(\W_{k}) + \beta_1\mathcal{C}_1^{\#}(\mathcal{C}_1\X_{k+1} - \mathcal{B}_1\W_{k+1}-\qb_1  + \frac{\La_{k,1}}{\beta_1})  + \beta_2 \mathcal{B}_2^{\#}(\mathcal{B}_2(\W_{k+1}) - \qb_2 + \frac{\La_{k,2}}{\beta_2} ) + L_{k,2} \Delta \W_{k+1}   = 0.
 \end{equation}
It follows from~\eqref{eqn:02W} and the update of $\La_{k+1,1}, \La_{k+1,2}$ that
\begin{equation} \label{eqn:dW2}
\begin{aligned}
 \qd_{\W_{k+1}} &= \nabla g(\W_{k+1}) - \nabla g(\W_k) + \mathcal{C}_1^{\#}\Delta \La_{k+1,1} + \mathcal{B}_2^{\#} \Delta \La_{k+1,2}  - L_{k,2} \Delta \W_{k+1} , \\
\qd_{\La_{k+1,1}} &= \frac{1}{\mu\beta_1 } \Delta \La_{k+1,1} \in \partial_{\La_1} \mathcal{L}_{\beta_1,\beta_2}(\X_{k+1},\W_{k+1},\La_{k+1,1},\La_{k+1,2}), \\
\qd_{\La_{k+1,2}} &= \frac{1}{\mu\beta_2 } \Delta \La_{k+1,2} \in \partial_{\La_2} \mathcal{L}_{\beta_1,\beta_2}(\X_{k+1},\W_{k+1},\La_{k+1,1},\La_{k+1,2}).
\end{aligned}
\end{equation}
Hence, it follows that $\qd_{\W_{k+1}}  \in \partial_{\W} \mathcal{L}_{\beta_1,\beta_2}(\X_{k+1},\W_{k+1},\La_{k+1,1},\La_{k+1,2}). $
Combining \eqref{eqn2:dX}, \eqref{eqn:dW2}, we obtain
$$
\begin{aligned}
\| \qd_{\X_{k+1}}\| & \le \|\mathcal{C}_1\| \|\Delta \La_{k+1,1}\|_{\mathrm{F}}  + q_1 \|\Delta \X_{k+1} \|_{\mathrm{F}} + \beta_1 \|\mathcal{C}_1\| \|\mathcal{B}_1\| \|\Delta \W_{k+1} \|_{\mathrm{F}}, \\
& \le \|\mathcal{B}_1\| \|\Delta \La_{k+1,1}\|_{\mathrm{F}} +  \|\mathcal{B}_2\| \|\Delta \La_{k+1,2}\|_{\mathrm{F}}   + (q_2 + L_g) \|\Delta \W_{k+1}\|_{\mathrm{F}}, 
 \\
\| \qd_{\La_{k+1,1}}\|  & = \frac{1}{\beta_1 \mu} \|\Delta \La_{k+1,1}\|_{\mathrm{F}}, \| \qd_{\La_{k+1,2}}\|  = \frac{1}{\beta_2 \mu} \|\Delta \La_{k+1,2}\|_{\mathrm{F}}. \\
\end{aligned}
$$
Therefore, the following inequality holds:
\begin{equation} \label{eqn2:bound4}
\begin{aligned}
& | \|\widetilde{\qd}_{k+1}\| | \le \|\qd_{\W_{k+1}}\|_{\mathrm{F}} + \|\qd_{\X_{k+1}} \|_{\mathrm{F}} + \|\qd_{\La_{k+1,1}}\|_{\mathrm{F}} + \|\qd_{\La_{k+1,2}}\|_{\mathrm{F}} \\
& \le \pi (\|\Delta \X_{k+1} \|_{\mathrm{F}} + \|\Delta \W_{k+1} \|_{\mathrm{F}}   + \|\Delta \La_{k+1,1} \|_{\mathrm{F}}  + \|\Delta \La_{k+1,2} \|_{\mathrm{F}} ),
\end{aligned}
\end{equation}
which completes the proof.
\end{proof}

\begin{lemma} \label{crit2}
Suppose that Assumption \ref{assum2} holds.
Any limit point $(\X_*,\W_*,\La_{*,1},\La_{*,2})$ of the sequence $\{(\X_{k},\W_{k},\La_{k,1},\La_{k,2} )\}_{k \ge 0}$ generated by~\eqref{scheme2} is a stationary
point of \eqref{prob:conve2}, i.e.
$$
0 =  \nabla f(\X_*) + \mathcal{C}_1^{\#}\La_{*,1},\; 0 = \nabla g(\W_*) + \mathcal{B}_1^{\#}\La_{*,1} + \mathcal{B}_2^{\#}\La_{*,2},\;  \mathcal{C}_1\X_* + \mathcal{B}_1\W_* = \mathbf{\dot{B}}_1, \; \mathcal{B}_2\W_* = \qb_2.
$$
\end{lemma}

\begin{proof}
Let $\{(\X_{k_j},\W_{k_j},\La_{k_j,1},\La_{k_j,2} )\}_{j \ge 0}$ be a subsequence of $\{(\X_{k},\W_{k},\La_{k,1},\La_{k,2})\}_{k \ge 0}$ such that $(\X_*,\W_*,\La_{*,1},\La_{*,2} ) = \lim_{j \rightarrow \infty}(\X_{k_j},\W_{k_j},\La_{k_j,1},\La_{k_j,2} )$. By the continuity of $\mathcal{L}_{\beta_1,\beta_2}$, $\mathcal{L}_{\beta_1,\beta_2} (\X_{k_j},\W_{k_j},\La_{k_j,1},\La_{k_j,2} ) \rightarrow \mathcal{L}_{\beta_1,\beta_2} (\X_*,\W_*,\La_{*,1},\La_{*,2})$ as $j \rightarrow \infty$. Let $\qd_{k_j} \in \partial \mathcal{L}_{\beta_1,\beta_2}(\X_{k_j},\W_{k_j},\La_{k_j,1},\La_{k_j,2} )$ and according to Lemma \ref{lem2:sub:bound}, it follows that
$
\|\qd_{k_j}\|_{\mathrm{F}} \leq \rho \left( \| \Delta \W_{k_j}\|_{\mathrm{F}} + \| \Delta \X_{k_j}\|_{\mathrm{F}}+ \| \Delta \La_{k_j,1}\|_{\mathrm{F}} + \| \Delta \La_{k_j,2}\|_{\mathrm{F}}\right).
$
According to Theorem~\ref{thm2:Xbound}, it follows that $\qd_{k_j} \rightarrow 0.$ By the closeness
criterion of the limiting subdifferential, $(\X_*,\W_*,\La_{*,1},\La_{*,2}) \in \text{crit}\, \mathcal{L}_{\beta_1,\beta_2}(\X,\W,\La_1,\La_2)$. The proof is completed.
\end{proof}

Since we have established the key properties, i.e. sufficient decrease property (Lemma \ref{lem2:mon}), the existence of limiting point (Theorem \ref{thm2:Xbound} and Lemma \ref{crit2}) and subgradient bound (Lemma \ref{lem2:sub:bound}), the global convergence of the algorithm is similar to the proof method in the literatures \cite{attouch2009convergence,attouch2013convergence,yashtini2022convergence}. We present the convergence result in Theorem \ref{thm2:global}, but omit its proof. 
\begin{theorem} \label{thm2:global}
Let $\tilde{\varGamma}$ denote the set of the limit points of the sequence $ \{(\X_k,\W_k,\La_{k,1},\La_{k,2} )\}_{k \ge 0}$. Suppose that Assumption \ref{assum2} holds and  $\mathcal{T}$ defined in \eqref{definition2:Rk} satisfies the KL property on $(\X_+,\W_+,\La_{+,1},\La_{+,2},\X_+,\W_+,\La_{+,1},\La_{+,2}, \X_+,\W_+,\La_{+,1})$ for any $(\X_+,\W_+,\La_{+,1},\La_{+,2} ) \in \tilde{\varGamma}$, 
then $ \{(\X_k,\W_k,\La_{k,1},\La_{k,2} )\}_{k \ge 0}$ satisfies the finite length property:
$
\sum_{k=0}^{\infty} \|\Delta \X_k\|_{\mathrm{F}} + \|\Delta \W_k\|_{\mathrm{F}} + \|\Delta \La_{k,1}\|_{\mathrm{F}}  + \|\Delta \La_{k,2}\|_{\mathrm{F}} < \infty,
$
and consequently converges to a stationary point of \eqref{prob:conve2}.
\end{theorem}

According to Theorem \ref{thm2:global}, we have the following convergence corollary of Algorithm \ref{algorithm:inpaint} for SLRQA-NF.
\begin{corollary} \label{corollary-2}
Suppose the penalty parameter $\beta_1$ and $\beta_2$ are large enough, and $L_{1,k}, L_{2,k}$ are set as given constants for all $k$ such that  \ref{assum2d} and \ref{assum2e} in Assumption~\ref{assum2} hold, then the sequence  $\{(\X_k,\W_k,\La_{k,1},\La_{k,2})\}_{k\ge 0}$ generated by Algorithm \ref{algorithm:inpaint} converges to a  stationary point of SLRQA-NF in \eqref{model:new2}. 
\end{corollary}

\begin{proof}
    By the definition of SLRQA \eqref{model:new}, \ref{assum2b} and \ref{assum2c} in  Assumption \ref{assum2} are satisfied. Furthermore, it follows from the proof of Corollary \ref{corollary-1} that $\mathcal{T}$ defined in \eqref{definition2:Rk} satisfy the KL property. Hence, the final result follows from Theorem \ref{thm2:global}.
\end{proof}

\section{Numerical experiment} \label{5}
 In this section, color image denoising and image inpainting problems are respectively used to demonstrate the performance of SLRQA in~\eqref{model:new} with Algorithm~\ref{algorithm:denoise} and SLRQA-NF in~\eqref{model:new2} with Algorithm~\ref{algorithm:inpaint}. 
In Section~\ref{sec:ParaSetting}, we give the parameter setting for SLRQA and SLRQA-NF.  In Section~\ref{sec:QuaVSRGB}, we compare SLRQA and SLRQA-NF based on quaternion and RGB representation to demonstrate the rationality of quaternion representation. Furthermore, in Section \ref{sec:ColorImagDe} and \ref{sec:ColorImagInp},  comparisons of SLRQA and SLRQA-NF with other state-of-the-art methods are also presented to respectively show the superiority of SLRQA and SLRQA-NF.
All the experiments are performed using MATLAB R2022b running on a desktop with an Intel Core R7-5800H CPU (3.90 GHz) and 16 GB of RAM. The codes are available at {\url{https://github.com/dengzhanwang/SLRQA/tree/main}}.

\subsection{Problem description, parameter setting, and testing environment} \label{sec:ParaSetting}

The models for color image denoising and image inpainting are respectively given by
\be \label{model:imagedenoising}
\begin{aligned}
    \min_{\X,\W \in \mathbb{H}^{m \times n} } \quad & \sum_i \phi(\sigma_i( \sqrt{ {\X}\strut^{*} \X + \varepsilon^2 \mathbf{I} }  ),\gamma)+\lambda p(\W) + \frac{1}{2\tau^2}\|\X - \Y\|_{\mathrm{F}}^2, \quad
    \st \quad  \mathcal{W}(\X)=\W,
\end{aligned}
\ee
and
\be \label{model:imageinpainting}
\begin{aligned}
    \min_{\Z,\W \in \mathbb{H}^{m \times n}} \quad & \sum_i (\sigma_i( \sqrt{ {\Z}\strut^{*} \Z + \varepsilon^2 \mathbf{I} }  ),\gamma) +\lambda p(\W),\quad
    \st \quad  \mathcal{P}_\Omega(\mathcal{W}^{\#}(\W))= \mathcal{P}_\Omega(\Y),\quad \Z=\W,
\end{aligned}
\ee
where $\mathcal{W}$ is the QDCT, see details in~\cite{QDCT}, $p(\W)$ is chosen as the Huber function, and the function $\phi$ is chosen as the Schatten-$\gamma$, Laplace, or Weighted Schatten-$\gamma$ functions. 
Note that the color image denoising problem~\eqref{model:imagedenoising} is SLRQA in~\eqref{model:new} with $\mathcal{A}$ being identity, $\Omega$ being the entire indices of pixels, and $p$ being the Huber function. The color image inpainting problem~\eqref{model:imageinpainting} is SLRQA-NF in~\eqref{model:new2} with $\mathcal{A}$ being identity and $p$ being the Huber function. Problem~\eqref{model:imagedenoising} with the Schatten-$\gamma$, Laplace, and Weighted Schatten-$\gamma$ functions for $\phi$ is respectively denoted by SLRQA-1, SLRQA-2, and SLRQA-3. Likewise, Problem~\eqref{model:imageinpainting} using the three options of $\phi$ is denoted by SLRQA-NF-1, SLRQA-NF-2, and SLRQA-NF-3. 

The ten images in Figure~\ref{picdn} are used to generate 
simulation problems. Specifically, for image denoising, the noisy images $\Y$ are obtained by adding additive white Gaussian noise with zero mean and variance $\tau^2$ to the color images in Figure~\ref{picdn}, and for image inpainting, each pixel of an image is chosen to be in the indices set $\Omega$ with the same probability $\chi \in (0, 1)$. Therefore, $\chi$ is also called the missing rate. The parameters $\tau$, $\lambda$, and $\gamma$ are specified later when reporting numerical results. Note that the parameters $\gamma$ and $\lambda$ depend on the choice of the surrogate function $\phi$, the noise level, and whether the non-local self-similarity is used.
The parameter $\delta$ is set depending on the accuracy of the current iterate. Specifically, $\delta = 1$ if $\epsilon_k \geq 10^{-2}$, $\delta = 10^{-2}$ if $10^{-3} \leq \epsilon_k < 10^{-2}$, and $\delta = 10^{-4}$ otherwise, where $\epsilon_k = \|\X_{k+1} - \X_{k}\|_{\mF} + \|\W_{k+1} - \W_{k}\|_{\mF} + \|\mathcal{W}(\X_{k+1}) - \W_{k+1}\|_{\mF}$ in Algorithm~\ref{algorithm:denoise} and $\epsilon_k = \|\Z_{k+1} - \W_{k+1}\|_{\mF} + \|\mathcal{P}_\Omega( \mathcal{W}^{\#}(\W_{k+1})) - \mathcal{P}_\Omega(\Y)\|_{\mF}$ in Algorithm~\ref{algorithm:inpaint}. Furthermore, for PL-ADMM-NF, the objective function of the subproblem of $\W$ is differentiable, we solve the subproblem using the gradient descent method.

The parameters in Algorithm \ref{algorithm:denoise}, \ref{algorithmf:X} and  \ref{algorithm:inpaint} are set as $\mu = 1.1$, $\beta = 10$, $L_{1, k}= L_{2, k}= 1, \forall k$, $\eta_C = 10^{-10}$, and $\eta = 10^{-4}$. Note that $\eta_C = 10^{-10}$ is for the high accuracy of the $\X$ or $\Z$-subproblem. The threshold $\varepsilon$ of $\phi$ is set to $10^{-2}$ in our experiments.

The $\texttt{qtfm}$ \cite{qtfm} package is used for the quaternion computations except the QSVD. 
The QSVD of a quaternion matrix $\X \in \mathbb{H}^{m \times n}$ is computed by taking the SVD of the complex adjoint form of $\X$, which needs an SVD of a $2m \times 2n$ complex matrix and is more efficient. For the details, we refer to~\cite{chen2019low}.



Peak Signal-to-Noise Ration (PSNR) and Structure Similarity (SSIM) \cite{wang2004image}  are chosen to measure the quality of the images recovered by all methods.  MATLAB commands $\texttt{psnr}$ and $\texttt{ssim}$ are used to compute the PSNR and SSIM for color images.  In general, the higher values the PSNR and
SSIM are, the better the denoising quality is.

\subsection{Quaternion and RGB representation} \label{sec:QuaVSRGB}


Simulation problems of image denoising in Section~\ref{sec:QuaVSRGBImagDe} and image inpainting in Section~\ref{sec:QuaVSRGBImagIn} are used to compare the performance of quaternion and RGB representations. 

\subsubsection{Image denoising} \label{sec:QuaVSRGBImagDe}

The parameter $\gamma$ is set as 0.5 in SLRQA-1, SLRQA-2, and SLRQA-3. The parameter $w$ in SLRQA-3 is chosen depending on the number of iteration $k$ and the index $i$, i.e., $w_{i,k}= 20 /(\sigma_i(\X_{k-1}) + 10^{-4})$. 
Note that the definition of $\phi$ for SLRQA-3 depends on the index $i$. Such dependency has been used in e.g.,~\cite{chen2019low}.
The multiple values of $(\lambda, \tau)$ are used, i.e., $(\lambda, \tau) = (0.01, 10), (0.3, 30)$, and $(0.5, 50)$. 

The PSNR and SSIM values of SLRQA with quaternion and RGB representation are reported in Table~$\ref{TAB8}$. Furthermore, the comparison of the two representations for SLRQA-1 with different noise levels is shown in Figure~\ref{fig: compare}. The figure and table show that SLRQA with quaternion representation achieves higher PSNR and SSIM compared with RGB representation in all cases. This observation coincides with the intuition that quaternion representation uses the information between channels to achieve better performance.

\begin{figure}[htbp]
\centering
\begin{minipage}{0.17\linewidth}
\centering
\includegraphics[width=1\linewidth]{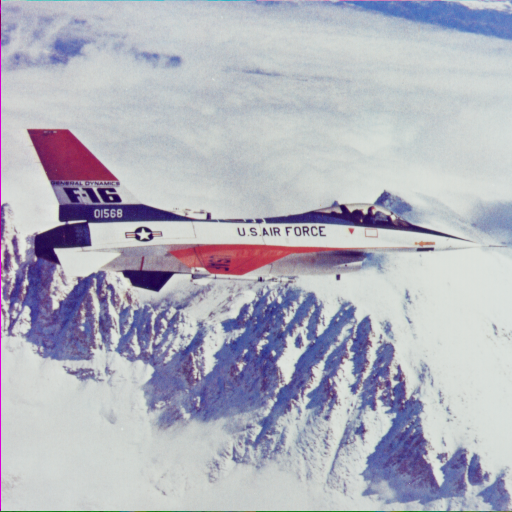}
\end{minipage}
 \hspace{0.2cm}
\begin{minipage}{0.17\linewidth}
\centering
\includegraphics[width=1\linewidth]{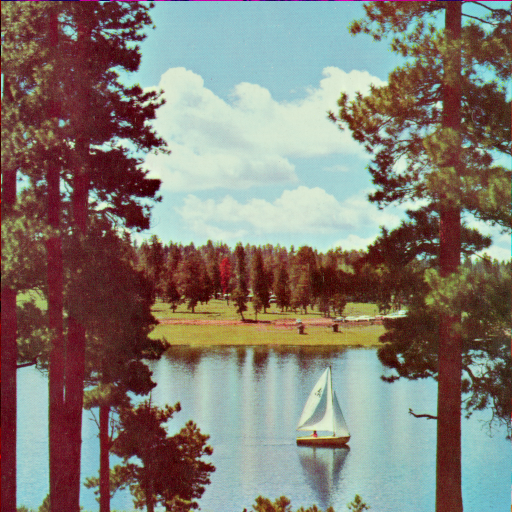}
\end{minipage}
 \hspace{0.2cm}
\begin{minipage}{0.17\linewidth}
\centering
\includegraphics[width=1\linewidth]{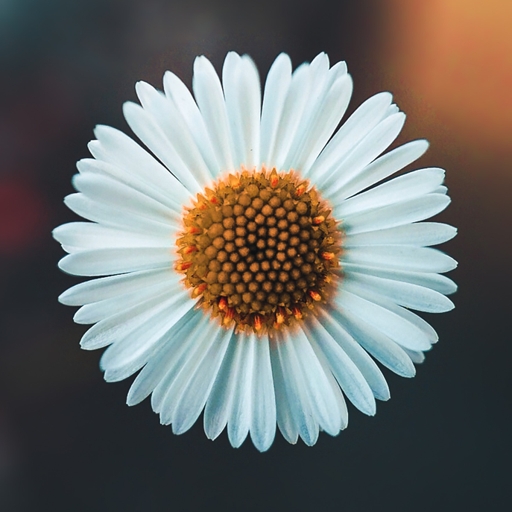}
\end{minipage}
 \hspace{0.2cm}
\begin{minipage}{0.17\linewidth}
\centering
\includegraphics[width=1\linewidth]{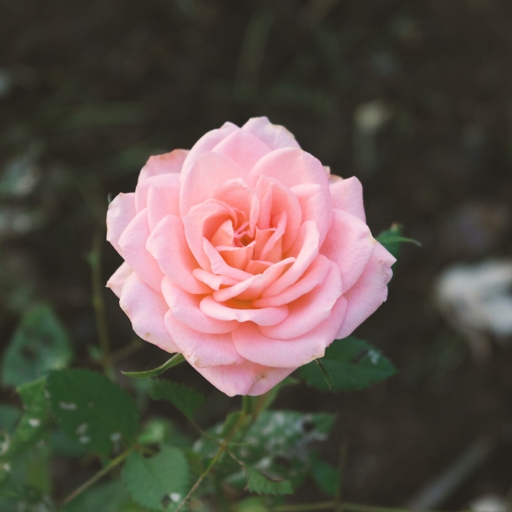}
\end{minipage}
 \hspace{0.2cm}
 \vspace{0.2cm}
\begin{minipage}{0.17\linewidth}
\centering
\includegraphics[width=1\linewidth]{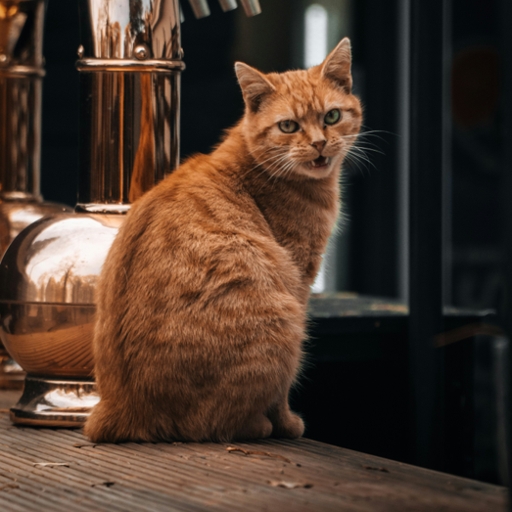}
\end{minipage}

\begin{minipage}{0.17\linewidth}
	\centering
	\includegraphics[width=1\linewidth]{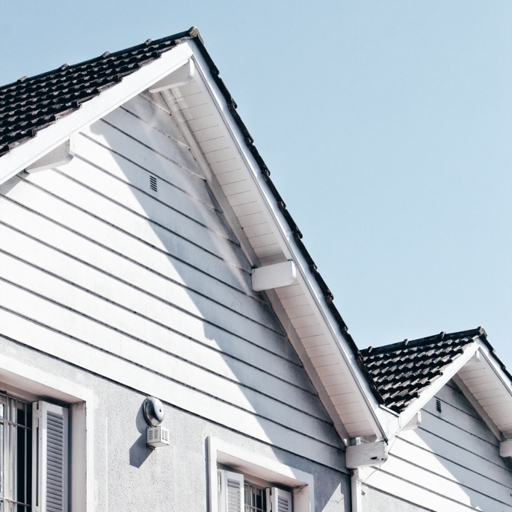}
\end{minipage}
 \hspace{0.2cm}
\begin{minipage}{0.17\linewidth}
	\centering
	\includegraphics[width=1\linewidth]{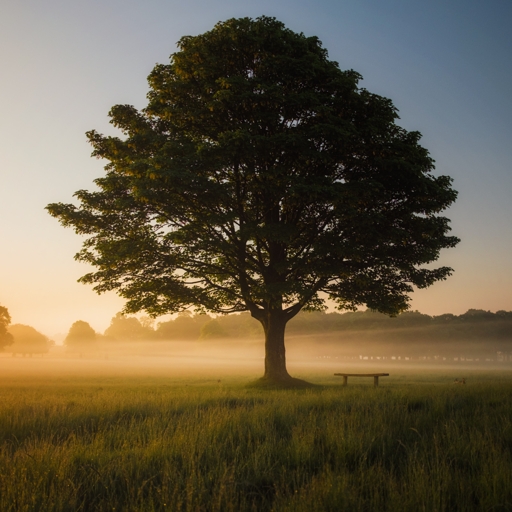}
\end{minipage}
 \hspace{0.2cm}
\begin{minipage}{0.17\linewidth}
	\centering
	\includegraphics[width=1\linewidth]{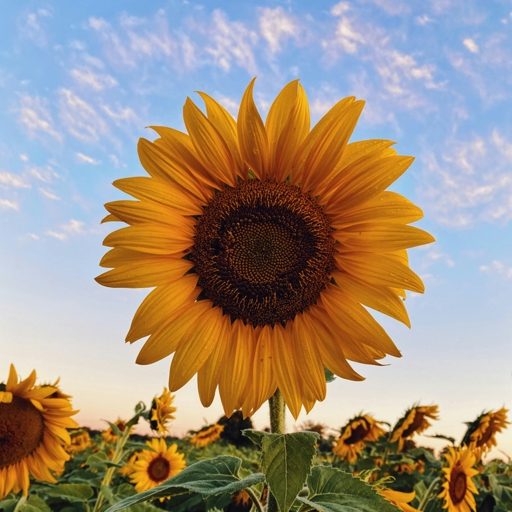}
\end{minipage}
 \hspace{0.2cm}
\begin{minipage}{0.17\linewidth}
	\centering
	\includegraphics[width=1\linewidth]{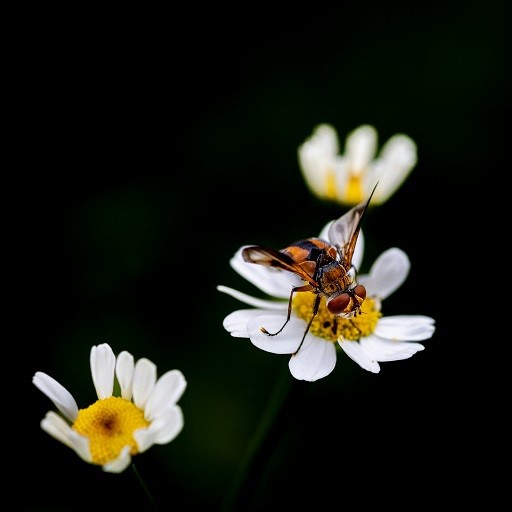}
\end{minipage}
 \hspace{0.2cm}
\begin{minipage}{0.16\linewidth}
	\centering
	\includegraphics[width=1\linewidth]{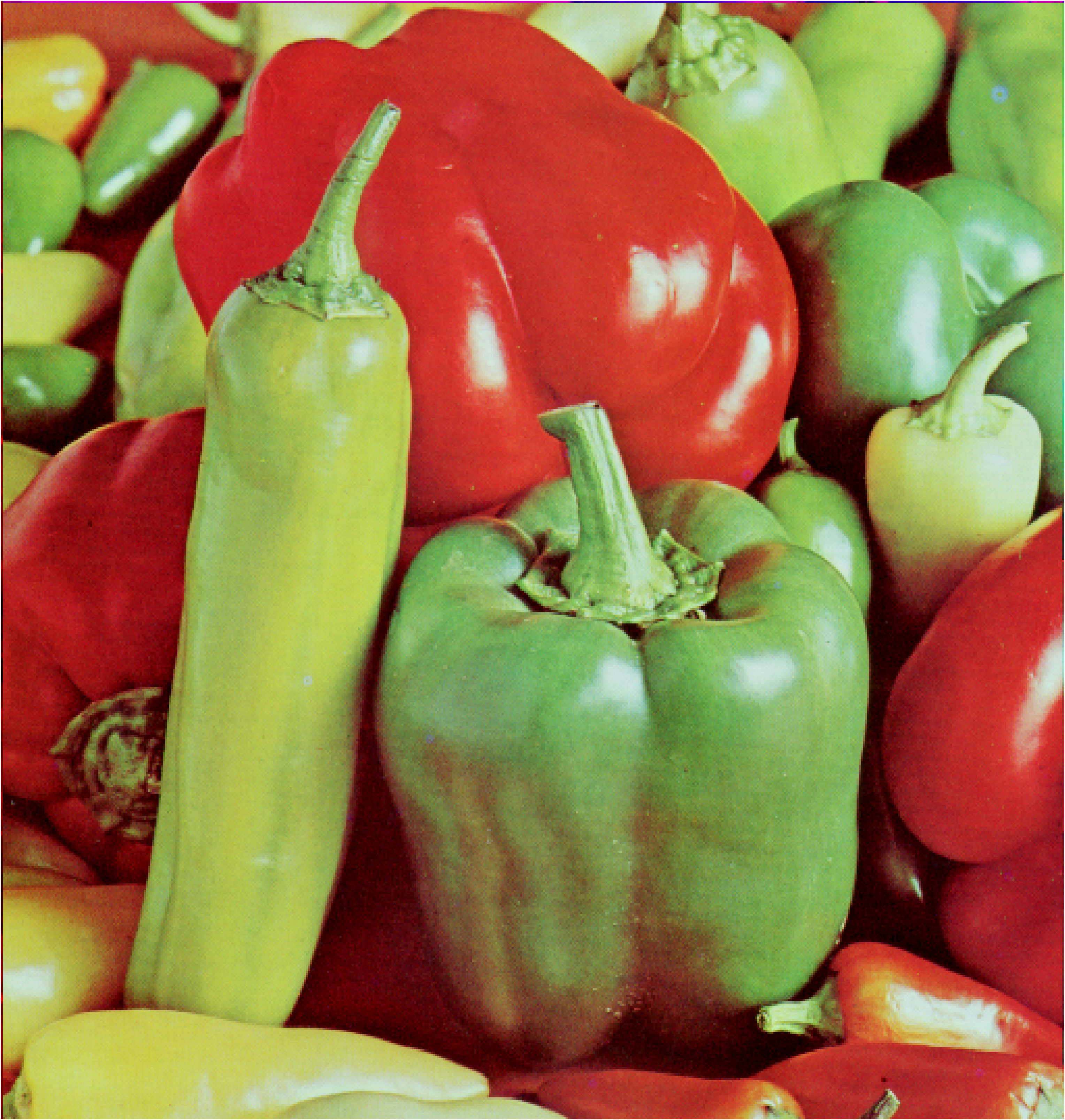}
\end{minipage}
\caption{The 10 color images (512 $\times$ 512 $\times$ 3) for numerical experiments.}\label{picdn}
\end{figure}

\begin{table*}[h]
	\setlength{\abovecaptionskip}{0.cm}
	\setlength{\belowcaptionskip}{-0.cm}
	\caption{PSNR/SSIM values of SLRQA using the RGB and quaternion representations. ``QR'' denotes that the quaternion representation is used. \textbf{Bold} fonts denote the best performance on the same conditions.}
	\label{TAB8}
	\centering
	\resizebox{\textwidth}{!}{
		\begin{tabular}{|c|c|c|c|c|c|c|c|c|c|c|c|}
			\hline
			\multicolumn{2}{|c|}{\diagbox[innerwidth=2.5cm]{Model}{Image}}&Image1&Image2&Image3&Image4&Image5&Image6&Image7&Image8&Image9&Image10\\
			\hline
			\multirow{2}*{\makecell[c]{SLRQA-1 \\ $\tau = 10$}}&QR&\textbf{32.38/0.82}&\textbf{30.43/0.81}&\textbf{31.89/0.81}&\textbf{33.86/0.87}&\textbf{33.27/0.86}&\textbf{31.21/0.81}&\textbf{31.20/0.78}&\textbf{31.20/0.78}&\textbf{33.53/0.70}&\textbf{31.54/0.79}\\
			\cline{2-12}
			\multicolumn{1}{|c|}{}&RGB&31.90/0.80&30.27/0.78&31.50/0.77&33.60/0.79&32.56/0.82&30.88/0.77&30.10/0.75&30.34/0.75&32.89/0.68&31.02/0.72\\
			\hline
			\multirow{2}*{\makecell[c]{SLRQA-1 \\ $\tau = 30$}}&QR&\textbf{24.86/0.70}&\textbf{23.91/0.66}&\textbf{24.97/0.72}&\textbf{26.68/0.67}&\textbf{25.65/0.69}&\textbf{22.85/0.69}&\textbf{24.70/0.68}&\textbf{24.19/0.66}&\textbf{25.30/0.64}&\textbf{24.97/0.66}\\
			\cline{2-12}
			\multicolumn{1}{|c|}{}&RGB&23.55/0.67&20.71/0.57&23.65/0.70&25.10/0.63&23.26/0.63&21.29/0.66&22.74/0.67&23.91/0.63&24.02/0.60&23.45/0.62\\
			\hline
			\multirow{2}*{\makecell[c]{SLRQA-1 \\ $\tau = 50$}}&QR&\textbf{23.33/0.50}&\textbf{22.04/0.48}&\textbf{23.01/0.51}&\textbf{22.23/0.46}&\textbf{21.84/0.52}&\textbf{21.24/0.50}&\textbf{23.57/0.48}&\textbf{22.14/0.48}&\textbf{22.78/0.56}&\textbf{22.53/0.57}\\
			\cline{2-12}
			\multicolumn{1}{|c|}{}&RGB &22.67/0.49&21.56/0.46&22.66/0.49&21.73/0.44&21.48/0.51&20.94/0.48&23.18/0.46&21.55/0.47&21.54/0.55&21.67/0.54\\
			\hline
   \hline
			\multirow{2}*{\makecell[c]{SLRQA-2 \\ $\tau = 10$}}&QR&\textbf{31.45/0.78}&\textbf{30.19/0.80}&\textbf{31.79/0.78}&\textbf{32.45/0.75}&\textbf{31.18/0.80}&\textbf{30.92/0.80}&\textbf{30.82/0.75}&\textbf{30.75/0.74}&\textbf{32.87/0.70}&\textbf{31.56/0.76}\\
			\cline{2-12}
			\multicolumn{1}{|c|}{}&RGB&31.20/0.77&30.72/0.80&31.46/0.78&31.91/0.72&31.04/0.80&30.57/0.79&30.20/0.74&30.35/0.72&31.40/0.67&30.32/0.72\\
			\hline
			\multirow{2}*{\makecell[c]{SLRQA-2 \\ $\tau = 30$}}&QR&\textbf{24.30/0.61}&\textbf{22.77/0.64}&\textbf{25.21/0.64}&\textbf{22.74}/0.57&\textbf{25.32/0.58}&\textbf{21.72/0.64}&\textbf{24.39/0.68}&\textbf{24.06/0.62}&\textbf{25.45/0.68}&\textbf{25.02/0.70}\\
			\cline{2-12}
			\multicolumn{1}{|c|}{}&RGB&24.12/0.60&22.47/0.63&24.98/0.62&22.52/0.56&25.02/0.58&21.59/0.62&24.26/0.68&23.75/0.60&24.82/0.65&24.30/0.64\\
			\hline
			\multirow{2}*{\makecell[c]{SLRQA-2 \\ $\tau = 50$}}&QR&\textbf{23.50/0.51}&\textbf{22.77/0.52}&\textbf{24.91/0.58}&\textbf{22.34/0.47}&\textbf{22.62/0.67}&\textbf{21.02/0.51}&\textbf{23.09/0.45}&\textbf{23.26/0.52}&\textbf{23.02/0.62}&\textbf{22.60/0.70}\\
			\cline{2-12}
			\multicolumn{1}{|c|}{}&RGB &23.27/0.51&21.96/0.52&24.76/0.57&22.03/0.46&21.88/0.65&20.94/0.51&22.88/0.45&22.90/0.51&22.32/0.60&22.12/0.65\\
			\hline
   \hline
			\multirow{2}*{\makecell[c]{SLRQA-3 \\ $\tau = 10$}}&QR&\textbf{32.86/0.81}&\textbf{31.20/0.82}&\textbf{32.03/0.84}&\textbf{30.36/0.75}&\textbf{32.50/0.87}&\textbf{30.77/0.76}&\textbf{31.10/0.80}&\textbf{31.92/0.76}&\textbf{33.89/0.73}&\textbf{31.87/0.82}\\
			\cline{2-12}
			\multicolumn{1}{|c|}{}&RGB&32.76/0.80&30.80/0.81&31.30/0.84&30.01/0.74&32.04/0.86&30.17/0.74&30.48/0.78&31.65/0.75&32.20/0.68&32.02/0.78\\
			\hline
			\multirow{2}*{\makecell[c]{SLRQA-3 \\ $\tau = 30$}}&QR&\textbf{26.40/0.67}&\textbf{24.02/0.64}&\textbf{26.22/0.68}&\textbf{24.21/0.64}&\textbf{26.34/0.68}&\textbf{22.50/0.63}&\textbf{26.42/0.68}&\textbf{26.12/0.67}&\textbf{25.80/0.68}&\textbf{25.68/0.70}\\
			\cline{2-12}
			\multicolumn{1}{|c|}{}&RGB&25.32/0.66&23.78/0.63&26.18/0.67&23.82/0.64&25.89/0.67&22.13/0.61&24.26/0.68&26.05/0.66&24.79/0.62&25.03/0.67\\
			\hline
			\multirow{2}*{\makecell[c]{SLRQA-3 \\ $\tau = 50$}}&QR&\textbf{23.90/0.51}&\textbf{22.06/0.45}&\textbf{24.99/0.60}&\textbf{23.51/0.49}&\textbf{22.72/0.65}&\textbf{22.00/0.53}&\textbf{23.92/0.53}&\textbf{24.32/0.56}&\textbf{23.04/0.63}&\textbf{23.10/0.65}\\
			\cline{2-12}
			\multicolumn{1}{|c|}{}&RGB &23.57/0.50&21.46/0.43&24.76/0.60&22.63/0.50&22.08/0.63&21.94/0.53&23.18/0.52&23.58/0.55&22.23/0.56&22.30/0.62\\
			\hline
		\end{tabular}
	}
\end{table*}

 \begin{figure}[h]
	\subfigure{
		\includegraphics[width=0.45\textwidth]{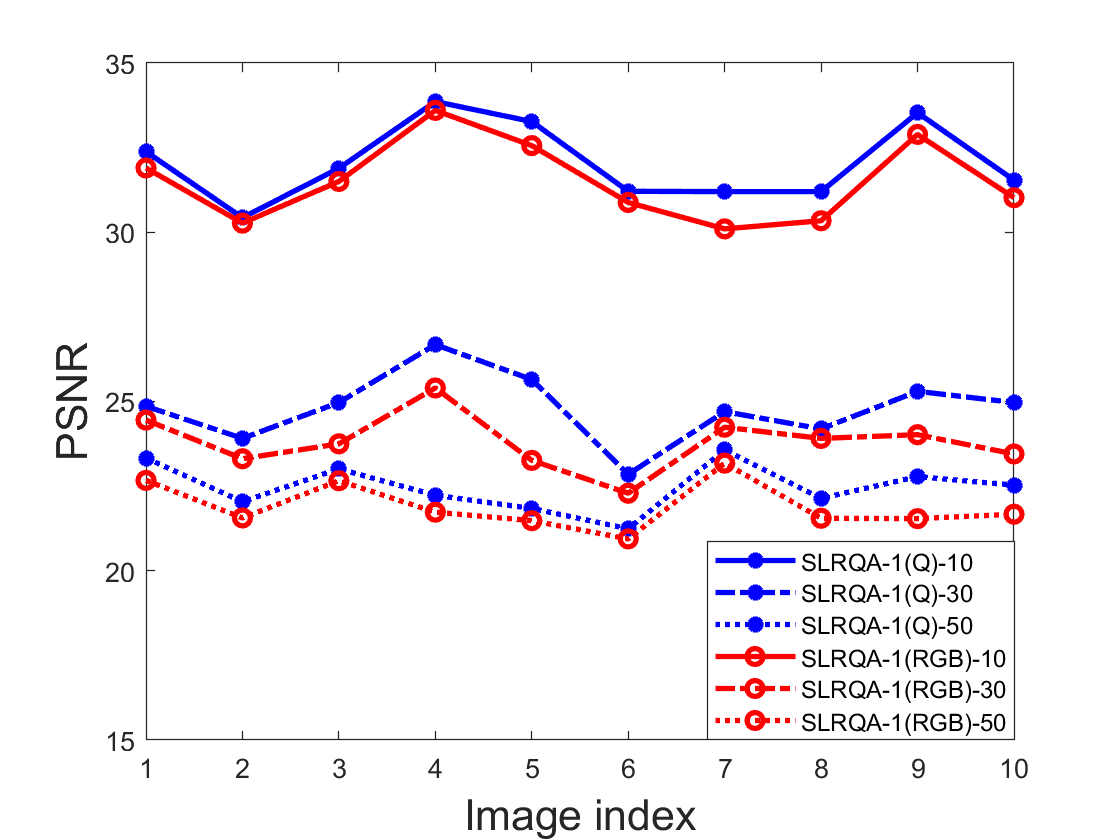}
	}
	\subfigure{
		\includegraphics[width=0.45\textwidth]{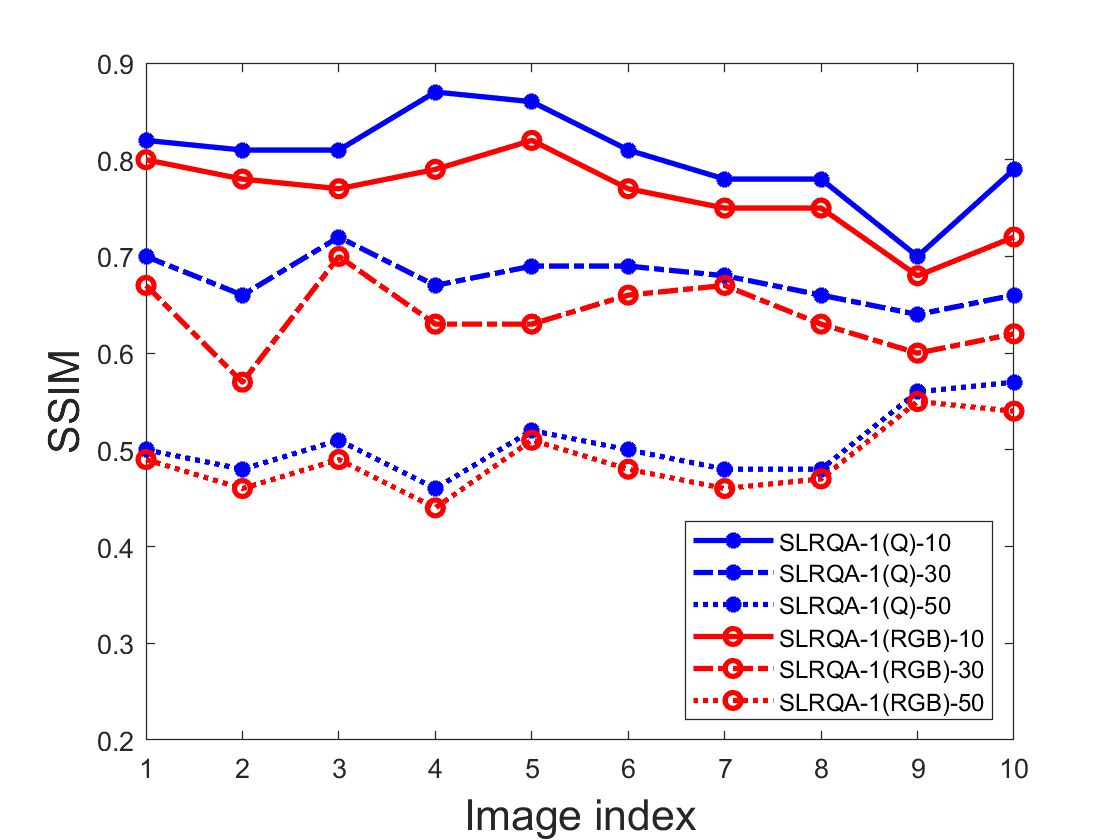}
	}
 \caption{PSNR and SSIM comparison of SLRQA-1 using quaternion and RGB representation. In the legend, the suffix (Q) means that quaternion representation is used, and (RGB) means that RGB representation is used. The numbers 10, 30, and 50 denote the variance of the noise respectively.}\label{fig: compare}
\end{figure}


\subsubsection{Image inpainting} \label{sec:QuaVSRGBImagIn}

\begin{table*}[h]
\caption{PSNR/SSIM values of SLRQA-NF using the RGB and quaternion representations. ``QR'' denotes that the quaternion representation is used.}
\label{TAB1}
\centering
\resizebox{\textwidth}{!}{
\begin{tabular}{|c|c|c|c|c|c|c|c|c|c|c|c|}
\hline
\multicolumn{2}{|c|}{\diagbox[innerwidth=4cm]{Model}{Image}}&Image1&Image2&Image3&Image4&Image5&Image6&Image7&Image8&Image9&Image10\\
\hline
\multirow{2}*{SLRQA-NF-1}&QR&\textbf{35.58/0.96}&\textbf{31.37/0.97}&\textbf{34.56/0.98}&\textbf{42.48/0.99}&\textbf{39.38/0.98}&\textbf{32.73/0.99}&\textbf{32.45/0.97}&\textbf{35.08/1.00}&\textbf{34.93/0.99}&\textbf{34.41/0.99}\\
\cline{2-12}
\multicolumn{1}{|c|}{($\chi$=0.3)}&RGB&34.66/0.93&30.83/0.96&33.66/0.97&40.19/0.98&38.63/0.99&30.10/0.90&30.85/0.92&33.91/0.98&34.11/0.99&33.61/0.99\\
\hline
\multirow{2}*{SLRQA-NF-1}&QR&\textbf{32.03/0.94}&\textbf{28.04/0.94}&\textbf{30.40/0.96}&\textbf{37.56/0.98}&\textbf{34.82/0.98}&\textbf{25.75/0.98}&\textbf{27.88/0.92}&\textbf{31.06/0.97}&\textbf{31.58/0.98}&\textbf{31.28/0.98}\\
\cline{2-12}
\multicolumn{1}{|c|}{($\chi$=0.5)}&RGB&31.19/0.90&27.59/0.93&29.77/0.93&36.42/0.98&34.28/0.97&25.10/0.77&27.35/0.86&30.33/0.96&30.91/0.98&30.63/0.98\\
\hline
\multirow{2}*{SLRQA-NF-1}&QR&\textbf{28.29/0.88}&\textbf{25.20/0.89}&\textbf{26.35/0.93}&\textbf{33.16/0.99}&\textbf{31.00/0.96}&\textbf{23.78/0.96}&\textbf{26.63/0.81}&\textbf{27.12/0.94}&\textbf{28.90/0.97}&\textbf{28.26/0.97}\\
\cline{2-12}
\multicolumn{1}{|c|}{($\chi$=0.7)}&RGB&25.49/0.71&22.53/0.79&23.23/0.75&29.27/0.89&28.42/0.86&20.59/0.59&24.48/0.77&24.72/0.86&25.85/0.77&25.53/0.93\\
\hline
\multirow{2}*{SLRQA-NF-1}&QR&\textbf{25.88/0.84}&\textbf{23.14/0.86}&\textbf{24.91/0.90}&\textbf{30.74/0.95}&\textbf{28.63/0.93}&\textbf{21.61/0.93}&\textbf{25.29/0.74}&\textbf{25.30/0.92}&\textbf{25.54/0.96}&\textbf{26.67/0.96}\\ \cline{2-12}
\multicolumn{1}{|c|}{($\chi$=0.8)}&RGB&22.41/0.54&20.63/0.69&21.12/0.67&26.51/0.81&24.92/0.78&18.14/0.46&22.82/0.71&22.38/0.79&23.97/0.61&22.54/0.91\\
\hline
\hline
\multirow{2}*{SLRQA-NF-2}&QR&\textbf{35.33/0.96}&\textbf{30.97/0.97}&\textbf{33.56/0.98}&\textbf{40.86/0.99}&\textbf{38.09/0.99}&\textbf{31.26/0.97}&\textbf{32.12/0.97}&\textbf{34.13/0.99}&\textbf{34.49/0.99}&\textbf{34.18/0.99}\\
\cline{2-12}
\multicolumn{1}{|c|}{($\chi$=0.3)}&RGB&34.18/0.95&30.83/0.97&33.15/0.98&39.19/0.99&37.81/0.99&26.62/0.99&26.85/0.95&33.54/0.98&33.97/0.99&33.84/0.99\\
\hline
\multirow{2}*{SLRQA-NF-2}&QR&\textbf{31.25/0.94}&\textbf{27.53/0.94}&\textbf{29.65/0.96}&\textbf{36.27/0.99}&\textbf{33.75/0.98}&\textbf{26.62/0.91}&\textbf{28.91/0.94}&\textbf{30.15/0.97}&\textbf{31.08/0.98}&\textbf{30.83/0.98}\\
\cline{2-12}
\multicolumn{1}{|c|}{($\chi$=0.5)}&RGB&30.65/0.92&27.49/0.94&29.46/0.95&35.60/0.97&33.63/0.97&24.62/0.96&26.91/0.90&29.92/0.96&30.78/0.98&30.63/0.98\\
\hline
\multirow{2}*{SLRQA-NF-2}&QR&\textbf{27.30/0.88}&\textbf{24.05/0.89}&\textbf{25.87/0.93}&\textbf{31.87/0.99}&\textbf{29.95/0.96}&\textbf{22.70/0.80}&\textbf{26.19/0.89}&\textbf{26.41/0.94}&\textbf{27.56/0.97}&\textbf{27.46/0.97}\\
\cline{2-12}
\multicolumn{1}{|c|}{($\chi$=0.7)}&RGB&25.90/0.86&23.19/0.87&25.63/0.90&31.27/0.96&27.39/0.94&20.26/0.77&24.52/0.85&25.17/0.93&25.53/0.96&25.25/0.96\\
\hline
\multirow{2}*{SLRQA-NF-2}&QR&\textbf{24.54/0.79}&\textbf{22.24/0.84}&\textbf{23.61/0.89}&\textbf{28.74/0.94}&\textbf{27.09/0.91}&\textbf{20.11/0.71}&\textbf{24.52/0.86}&\textbf{23.89/0.90}&\textbf{25.56/0.95}&\textbf{25.33/0.95}\\
\cline{2-12}
\multicolumn{1}{|c|}{($\chi$=0.8)}&RGB&21.90/0.76&21.84/0.83&21.13/0.86&26.35/0.93&24.92/0.91&17.85/0.81&22.35/0.85&21.23/0.89&23.07/0.94&22.75/0.95\\
\hline
\hline
\multirow{2}*{SLRQA-NF-3}&QR&\textbf{35.91/0.96}&\textbf{31.41/0.97}&\textbf{34.46/0.98}&\textbf{41.86/0.99}&\textbf{39.00/0.99}&\textbf{32.80/0.97}&\textbf{32.52/0.97}&\textbf{34.61/1.00}&\textbf{35.77/0.99}&\textbf{34.63/0.99}\\
\cline{2-12}
\multicolumn{1}{|c|}{($\chi$=0.3)}&RGB&34.66/0.93&30.83/0.96&33.66/0.97&40.19/0.99&38.63/0.99&28.79/0.95&30.33/0.95&33.91/0.98&34.11/0.99&33.61/0.99\\
\hline
\multirow{2}*{SLRQA-NF-3}&QR&\textbf{32.02/0.93}&\textbf{28.21/0.94}&\textbf{30.52/0.96}&\textbf{37.69/0.99}&\textbf{34.80/0.98}&\textbf{27.97/0.92}&\textbf{29.33/0.94}&\textbf{30.92/0.97}&\textbf{32.14/0.98}&\textbf{31.60/0.98}\\
\cline{2-12}
\multicolumn{1}{|c|}{($\chi$=0.5)}&RGB&31.23/0.90&27.59/0.93&29.77/0.93&36.42/0.98&34.28/0.97&25.46/0.90&27.37/0.91&30.33/0.96&30.91/0.98&30.63/0.98\\
\hline
\multirow{2}*{SLRQA-NF-3}&QR&\textbf{28.40/0.88}&\textbf{25.05/0.89}&\textbf{26.87/0.93}&\textbf{33.37/0.97}&\textbf{30.95/0.96}&\textbf{24.00/0.83}&\textbf{26.77/0.89}&\textbf{27.41/0.94}&\textbf{28.56/0.97}&\textbf{28.46/0.97}\\
\cline{2-12}
\multicolumn{1}{|c|}{($\chi$=0.7)}&RGB&25.59/0.72&22.53/0.80&23.23/0.76&29.27/0.90&28.42/0.89&22.10/0.80&24.86/0.85&24.72/0.87&25.85/0.78&25.53/0.94\\
\hline
\multirow{2}*{SLRQA-NF-3}&QR&\textbf{26.08/0.84}&\textbf{23.36/0.86}&\textbf{25.11/0.91}&\textbf{30.86/0.95}&\textbf{28.70/0.93}&\textbf{21.83/0.76}&\textbf{25.43/0.86}&\textbf{25.58/0.92}&\textbf{26.74/0.96}&\textbf{26.75/0.96}\\
\cline{2-12}
\multicolumn{1}{|c|}{($\chi$=0.8)}&RGB&22.60/0.56&20.84/0.70&21.24/0.67&26.72/0.81&25.12/0.79&21.83/0.74&25.43/0.82&22.60/0.80&23.97/0.61&22.89/0.91\\
\hline
\end{tabular}
}
\end{table*}

The parameter $\gamma$ is set as 0.7, 1, and 0.7 respectively for SLRQA-NF-1, SLRQA-NF-2, and SLRQA-NF-3. The parameter $w$ in SLRQA-NF-3 is set as $w_{i,k} = 10 /(\sigma_i(\X_{k-1})+10^{-4})$.

It is shown in Table \ref{TAB1} that the PSNR and SSIM values of the  SLRQA-NF with quaternion representation are higher than  SLRQA-NF with RGB representation for all different missing rate~$\chi$.
 Furthermore, the PSNR and SSIM values of different $\chi$ and different images are given in Figure \ref{fig:357} which shows that SLRQA-NF-1 with quaternion representation method outperforms SLRQA-NF-1 with RGB representation by a large margin. Moreover, the improvements in PSNR and SSIM increase with the rate of corruption.
 \begin{figure}[h]
	\subfigure{
	\includegraphics[width=0.45\textwidth]{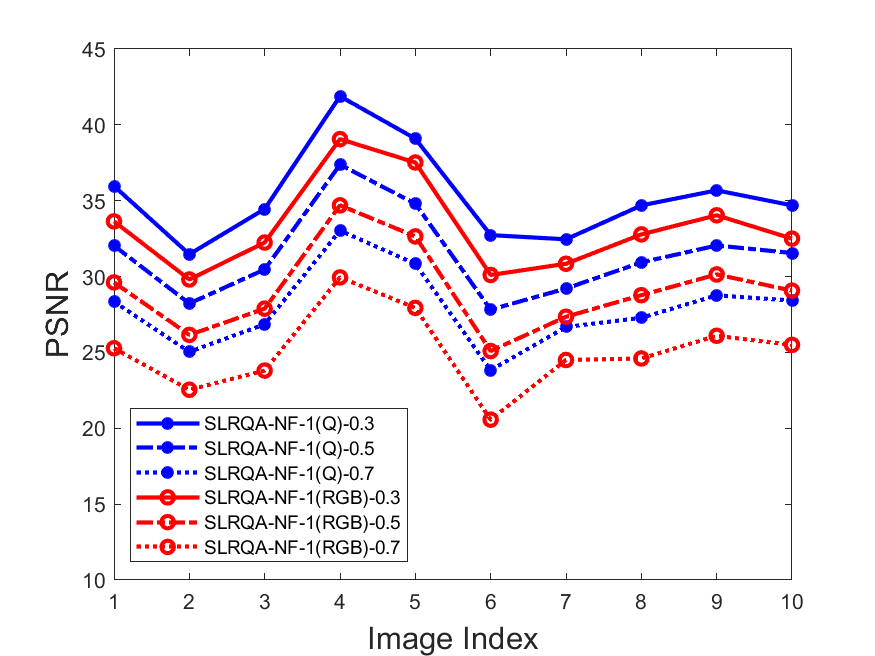}
	}
	\subfigure{
	\includegraphics[width=0.45\textwidth]{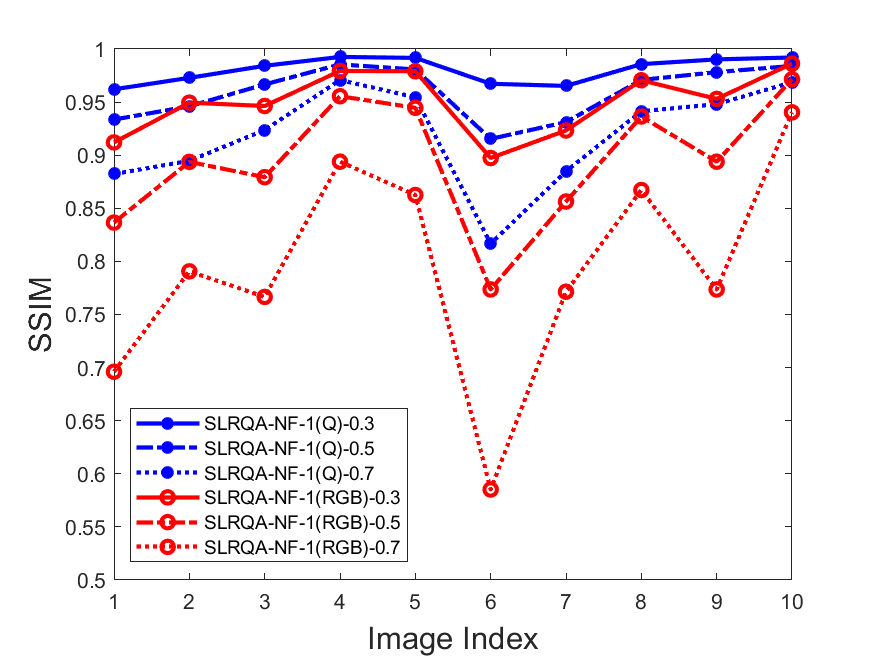}
	}
  \caption{PSNR and SSIM comparisons of quaternion and RGB based SLRQA-NF-1. In the legend, the suffix (Q) means that quaternion representation is used, and (RGB) means that RGB representation is used. The numbers 0.3, 0.5, and 0.7 denote the missing rate of the images respectively. }
 \label{fig:357}
\end{figure}

It is concluded that the quaternion-based model 
achieves better denoising and inpainting performance. Hence, it is preferable to use the quaternion model in color image processing.

  


\subsection{Color image denoising} \label{sec:ColorImagDe}
Non-local self-similarity~(NSS)  is a significant technique in color image denoising problems. 
The comparison without NSS directly reflects the superiority of the compared models, and the comparison with NSS better reflects the robustness of the models in practice. 
To fully compare the SLRQA with other state-of-the-art color image denoising methods, we tested their performance without and with the usage of NSS respectively in Section~\ref{6-3-1} and Section~\ref{6-3-2}.



\subsubsection{Color image denoising without NSS} \label{6-3-1}
The compared methods include pQSTV~\cite{TTW2022pQSTV}, QWSNM~\cite{QHZ2024QWSMN}, MCNNFNM~\cite{YS2023MCNNFNM}, SRRC~\cite{zhang2023SRRC}, QWNNM~\cite{huang2022quaternion}, LRQA~\cite{chen2019low}, and DWT~\cite{cai2016image}\footnote{
The implementations that we use are available at \\
pQSTV: \href{https://github.com/Huang-chao-yan/pQSTV.git}{ https://github.com/Huang-chao-yan/pQSTV.git}\\
QWSNM: 
\href{https://github.com/qiuxuanzhizi/QWSNM.git}{https://github.com/qiuxuanzhizi/QWSNM.git}\\
MCNNFNM: \href{https://github.com/wangzhi-swu/MCNNFNM.git}{https://github.com/wangzhi-swu/MCNNFNM.git}\\
SRRC: \href{https://github.com/zt9877/SRRC.git}{https://github.com/zt9877/SRRC.git} \\
LRQA: \href{https://www.fst.um.edu.mo/personal/wp-content/uploads/2021/05/LRQA.zip}{https://www.fst.um.edu.mo/personal/wp-content/uploads/2021/05/LRQA.zip}}. QWSNM, MCNNFNM, SRRC, QWNNM, and  LRQA are low-rank minimization models. QWSNM  extends the WSNM into the quaternion domain for color restoration. MCNNFNM is a multi-channel nuclear norm minus Frobenius norm minimization model. SRRC  exploits the weighted Schatten $p$-norm as the regularizer for the rank residual to obtain a new rank minimization model. QWNNM uses the weighted nuclear norm as an alternative function to the rank function. LRQA is a quaternion-based model that makes use of the low-rankness property by a nonconvex surrogate function. Compared with SLRQA, none of SRRC, QWNNM, and LRQA exploit the sparsity feature of the image. pQSTV is a dictionary method for color image denoising that combines total variation regularization and pure quaternion representation.
DWT~\cite{cai2016image} is a vanilla analysis-based approach \eqref{model:denoising-analysis} where low-rankness prior is not used and is implemented on our own. It is noted that MCNNFNM, SRRC, and DWT are monochromatic methods, while QWNNM, pQSTV, QWSNM, LRQA, and SLRQA are quaternion-based models. LRQA-1 and LRQA-2 denote the LRQA models using Schatten-$\gamma$ and Laplace function, respectively. The comparisons with LRQA and QWNNM can be regarded as ablation experiments to present the significance of sparsity and low-rankness prior respectively. For the compared methods, we use the default parameters stated in the corresponding papers.




The PSNR and SSIM results of the compared methods are shown in Table~$\ref{TAB5}$. It can be seen that SLRQA  has better results than the other methods in terms of PSNR and SSIM in most cases. Furthermore, the PSNR values of SRRC, QWNNM, and DWT are much lower than that of SLRQA when $\tau$ is large.
Note that since NSS is not used in these methods, the results in Table~$\ref{TAB5}$ underscore the effectiveness and robustness of the SLRQA and further validate the advanced nature and adaptability of its intrinsic architecture. We also give the visual comparison between all competing methods and SLRQA as shown in Figure~\ref{result1}. It can be seen in the highlighted red rectangles that when compared with other denoising methods, SLRQA can preserve more details of images. In particular, in terms of hair details, the images denoised by the DWT and LRQA-1 methods are still noisy. The images recovered by LRQA and SRRC methods contain blurry parts, and hence the detailed features of the images are lost. In contrast, the images recovered by SLRQAs have better results both numerically and visually.

To further analyze the details of PL-ADMM, we plot the empirical convergence of PL-ADMM for Problem~\eqref{model:imagedenoising}
on ten tested images with $\tau = 30$. The results are shown in Figure~\ref{fig:dn1}.
It can be observed that the error curve has a downward trend, illustrating the empirical convergence of Algorithm \ref{algorithm:denoise}. We can observe that $\|\Delta \X_{k+1} \|_{\mathrm{F}} + \|\Delta \W_{k+1} \|_{\mathrm{F}}   + \|\Delta \La_{k+1} \|_{\mathrm{F}}  \rightarrow 0$.
Hence, 
it follows that $(\X,\W,\La)$ converges to a stationary point of $\mathcal{L}_{\beta}$ and hence a KKT point of the original problem in~\eqref{model:imagedenoising} \cite[Theorem 2]{yashtini2022convergence}.


\begin{table*}[h]\scriptsize
	\setlength{\abovecaptionskip}{0.cm}
	\setlength{\belowcaptionskip}{-0.cm}
        \setlength{\tabcolsep}{2pt}
	\caption{\footnotesize{PSNR/SSIM results on color image denoising problem  of different methods without NSS.}}
	\label{TAB5}
	\centering
	\resizebox{\linewidth}{!}{
		\begin{tabular}{|c|c|c|c|c|c|c|c|c|c|c|c|c|}
			\hline
			\multicolumn{2}{|c|}{\diagbox[innerwidth=2.2cm]{Image}{Algorithm}}&pQSTV\cite{TTW2022pQSTV}&QWSNM\cite{QHZ2024QWSMN}&MCNNFNM\cite{YS2023MCNNFNM}&SRRC$\cite{zhang2023SRRC}$&QWNNM$\cite{gu2014weighted}$&DWT$\cite{cai2016image}$&LRQA-1\cite{chen2019low} &LRQA-2\cite{chen2019low}&SLRQA-1&SLRQA-2&SLRQA-3\\ 
			\hline
			\multirow{8}*{$\tau = 10$}&Image1&30.32/0.71&29.00/0.79&30.09/0.79&29.23/0.70&28.06/0.73&27.67/0.61&28.48/0.61&29.21/0.65&31.55/0.76&31.45/0.78&\textbf{32.86/0.81}\\
			\cline{2-13}
			\multicolumn{1}{|c|}{}&Image2&29.21/0.80&30.10/0.79 &28.10/0.75 &28.56/0.73&28.01/0.70&27.98/0.71&28.44/0.72&28.42/0.72&30.36/0.80&30.19/0.80&\textbf{31.20/0.82}\\
			
                \cline{2-13}
			\multicolumn{1}{|c|}{}&Image3&30.20/0.78&29.02/0.75&30.12/0.77 &28.03/0.60&27.81/0.61&28.21/0.61&28.49/0.64&28.56/0.65&31.33/0.77&31.79/0.78 &\textbf{32.03/0.84}\\
			
               \cline{2-13}
			\multicolumn{1}{|c|}{}&Image4&28.56/0.70 &29.20/0.72 &28.90/0.69 &28.12/0.60&27.90/0.65&28.29/0.66&28.53/0.64 &28.55/0.68&30.09/0.74&\textbf{31.45/0.75}&30.36/\textbf{0.75}\\

               \cline{2-13}
			\multicolumn{1}{|c|}{}&Image5&30.92/0.80&30.08/0.78 &30.81/0.80 &28.82/0.66&28.72/0.64&28.67/0.64&28.85/0.65&28.71/0.67&32.13/0.85&31.18/0.80&\textbf{32.50/0.87}\\
			
			\cline{2-13}
			\multicolumn{1}{|c|}{}&Image6&28.09/0.70&30.01/0.76&29.40/0.71 &28.32/0.60&27.89/0.56&28.02/0.61&28.72/0.65&28.66/0.65&30.75/0.76&\textbf{30.92/0.80}&30.77/0.76\\

               \cline{2-13}
			\multicolumn{1}{|c|}{}&Image7&30.08/0.75&29.60/0.74&30.23/0.75 &28.45/0.62&27.32/0.58&28.02/0.59&28.65/0.60&29.04/0.94 &30.95/0.73&30.82/0.75&\textbf{31.10/0.80}\\

               \cline{2-13}
			\multicolumn{1}{|c|}{}&Image8&29.20/0.70&30.91/0.71 &30.70/0.75 &28.01/0.60&28.11/0.59&28.46/0.58&28.56/0.63&28.24/ 0.63&31.23/0.76&30.75/0.74&\textbf{31.92/0.76}\\
      \cline{2-13}
			\multicolumn{1}{|c|}{}&Image9&30.00/0.75&29.80/0.71&29.32/0.72 &28.02/0.74&27.98/0.73&27.62/0.71&28.82/0.76&29.10/0.76&30.14/0.70&29.66/0.78&\textbf{30.80/0.79}\\
        \cline{2-13}
			\multicolumn{1}{|c|}{}&Image10&30.86/0.75&30.90/0.75&31.02/0.76 &30.52/0.74&29.72/0.72&29.62/0.70&30.90/0.76&31.10/0.79&31.14/0.78&31.26/0.80&\textbf{31.32/0.81}\\
			\hline
                \hline
			\multirow{8}*{$\tau = 30$}&Image1&23.98/0.56&24.80/0.55&25.00/0.57 &20.32/0.30&20.41/0.25&20.36/0.28&20.49/0.27&20.80/0.30&25.47/0.66&24.30/0.61&\textbf{26.40/0.67}\\
		\cline{2-13}
			\multicolumn{1}{|c|}{}&Image2&22.80/0.55 &23.01/0.57&21.32/0.55 &20.01/0.32&20.73/0.32&20.09/0.33&20.44/0.37&20.40/0.37&23.90/0.61&22.77/0.64&\textbf{24.02/0.64}\\
			
                \cline{2-13}
			\multicolumn{1}{|c|}{}&Image3&22.90/0.54&23.43/0.58 &22.03/0.52 &20.12/0.21&20.24/0.22&20.55/0.25 &20.50/0.27&21.10/0.31&26.15/0.65&25.21/0.64&\textbf{26.22/0.68}\\
			
               \cline{2-13}
			\multicolumn{1}{|c|}{}&Image4&22.39/0.43&23.09/0.55&23.10/0.55 &20.22/0.29&20.10/0.26& 20.01/0.29&20.68/0.29&20.20/0.31&23.81/0.59&22.74/0.57&\textbf{24.21/0.64}\\

               \cline{2-13}
			\multicolumn{1}{|c|}{}&Image5&23.03/0.59&24.43/0.60 &24.60/0.60 &21.02/0.32&20.58/0.33& 20.65/0.37
&21.23/0.30&21.30/0.33&\textbf{26.60/0.68}&25.32/0.58&26.34/\textbf{0.68}\\
			
			\cline{2-13}
			\multicolumn{1}{|c|}{}&Image6&21.30/0.59&21.01/0.55&22.30/0.60 &20.87/0.33&21.01/0.37& 20.37/0.37&20.76/0.36&20.61/0.35&\textbf{23.95/0.66}&21.72/0.64&22.50/0.63\\

               \cline{2-13}
			\multicolumn{1}{|c|}{}&Image7&23.09/0.56&24.33/0.60&24.90/0.61 &20.33/0.29&20.71/0.29& 20.94/0.22&20.93/0.26&20.21/0.21&25.12/0.63&24.39/0.68&\textbf{26.42/0.68}\\

               \cline{2-13}
			\multicolumn{1}{|c|}{}&Image8&23.33/0.57&24.01/0.58&24.40/0.61 &20.56/0.22&20.42/0.24& 21.08/0.28&20.77/0.29&20.90/0.29&25.14/0.64&24.06/0.62&\textbf{26.12/0.67}\\
   \cline{2-13}
			\multicolumn{1}{|c|}{}&Image9&24.90/0.60&24.23/0.59&24.89/0.61 &24.82/0.58&24.22/0.58&23.62/0.55&24.90/0.58&25.10/0.61&25.30/0.64&25.45/0.68&\textbf{25.80/0.68}\\
        \cline{2-13}
			\multicolumn{1}{|c|}{}&Image10&23.90/0.60 &23.84/0.55&24.01/0.62&22.32/0.60&23.82/0.62&23.62/0.61&24.34/0.66&24.90/0.68&24.97/0.66&25.02/0.70&\textbf{25.68/0.70}\\

                \hline
			\hline
			\multirow{8}*{$\tau = 50$}&Image1&18.02/0.32 &19.03/0.45&20.30/0.48 &17.92/0.17&17.01/0.14&18.01/0.18&17.66/0.17&17.30/0.18&23.33/0.50&23.50/0.51&\textbf{23.90}/0.51\\
			\cline{2-13}
			\multicolumn{1}{|c|}{}&Image2&18.30/0.40&19.02/0.45&18.46/0.42 &16.02/0.21&17.03/0.20&17.94/0.25&17.53/0.25&17.10/0.26&22.04/0.48&\textbf{22.77/0.52}&22.06/0.45\\
			
                \cline{2-13}
			\multicolumn{1}{|c|}{}&Image3&20.01/0.49&20.30/0.52&21.33/0.50 &17.75/0.15&17.40/0.21&17.24/0.19&17.56/0.17&17.90/0.18&24.21/0.51&24.91/0.58&\textbf{24.99/0.60}\\
			
               \cline{2-13}
			\multicolumn{1}{|c|}{}&Image4&20.30/0.39&21.02/0.42&20.80/0.40&17.08/0.14&16.60/0.16&16.87/0.17&17.82/0.19&17.20/0.14&22.23/0.46&22.34/0.47&\textbf{23.51/0.49}\\

               \cline{2-13}
			\multicolumn{1}{|c|}{}&Image5&20.20/0.45&21.03/0.50&20.39/0.51 &18.04/0.17&18.08/0.17& 18.31/0.18&18.39/0.19&18.30/0.17&21.84/0.52&21.62/\textbf{0.67}&\textbf{22.72}/0.65\\
			
			\cline{2-13}
			\multicolumn{1}{|c|}{}&Image6&19.32/0.48&19.80/0.50 &20.20/0.48 &17.53/0.25&18.01/0.26& 17.11/0.22&17.75/0.26&18.61/0.29&21.24/0.50&21.02/0.51&\textbf{22.00/0.53}\\

               \cline{2-13}
			\multicolumn{1}{|c|}{}&Image7&20.90/0.49&21.02/0.50&21.30/0.50 &17.60/0.17&18.21/0.18&17.11/0.13&18.08/0.17&18.71/0.17&23.57/0.48&23.09/0.45&\textbf{23.92/0.53}\\

               \cline{2-13}
			\multicolumn{1}{|c|}{}&Image8&21.93/0.50&22.03/0.49&22.30/0.51 &17.02/0.14&17.42/0.18&17.62/0.19&17.90/0.18&18.10/0.29&23.14/0.48&23.26/0.52&\textbf{24.32/0.56}\\

         \cline{2-13}
			\multicolumn{1}{|c|}{}&Image9&22.30/0.50 &21.79/0.51&22.80/0.52 &22.02/0.51&21..92/0.49&21.62/0.46&22.34/0.52&22.67/0.55&22.78/0.56&23.02/0.62&\textbf{23.04/0.63}\\
        \cline{2-13}
			\multicolumn{1}{|c|}{}&Image10&22.01/0.50&21.90/0.48 &22.34/0.55 &21.82/0.50&21.42/0.48&21.32/0.45&22.32/0.51&22.10/0.51&22.75/0.54&22.59/0.52&\textbf{23.11/0.58}\\
			\hline
	\end{tabular}
	}
\end{table*}

\begin{figure}[h]
	\centering
	\begin{minipage}{0.17\linewidth}
		\centering
		\includegraphics[width=1\linewidth]{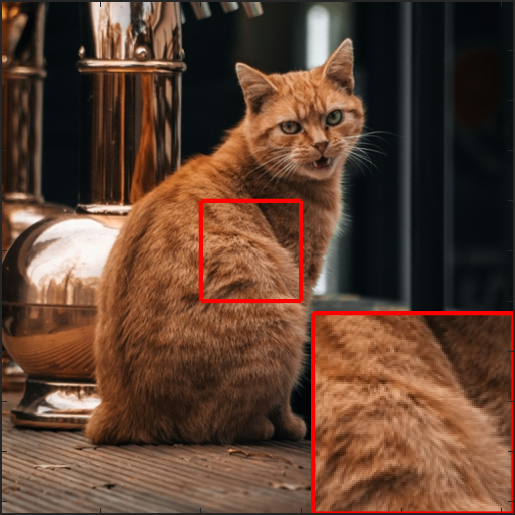}
		\centerline{\footnotesize{(a) Image5}}
	\end{minipage}
	\begin{minipage}{0.17\linewidth}
		\centering
		\includegraphics[width=1\linewidth]{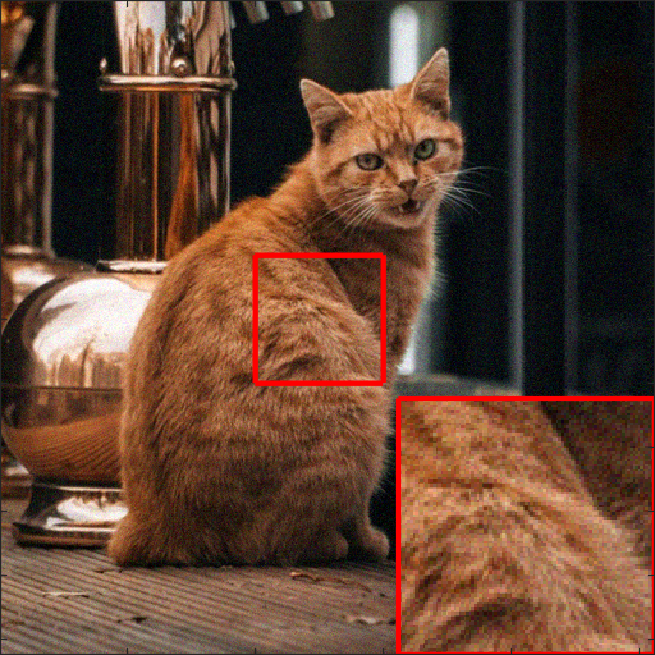}
		\centerline{\footnotesize{(b) Noisy Image}}
	\end{minipage}
    \begin{minipage}{0.17\linewidth}
		\centering
		\includegraphics[width=1\linewidth]{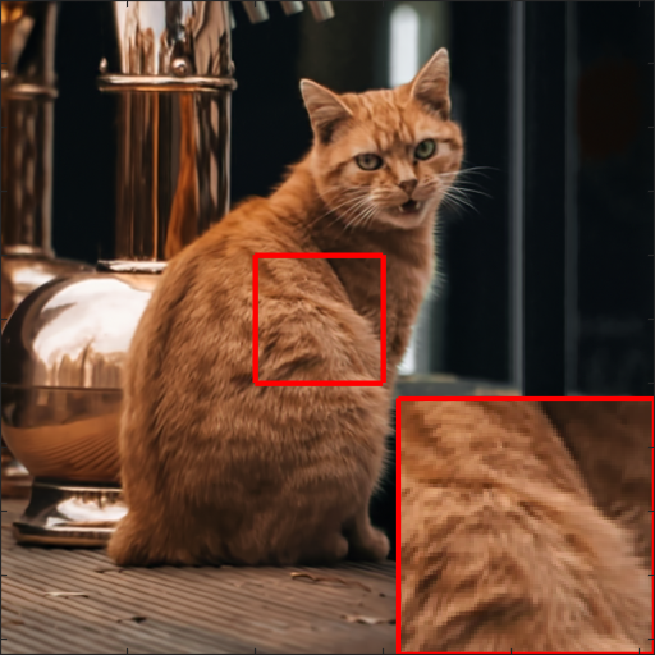}
		\centerline{\footnotesize{(c) pQSTV}}
	\end{minipage}
    \begin{minipage}{0.17\linewidth}
		\centering
		\includegraphics[width=1\linewidth]{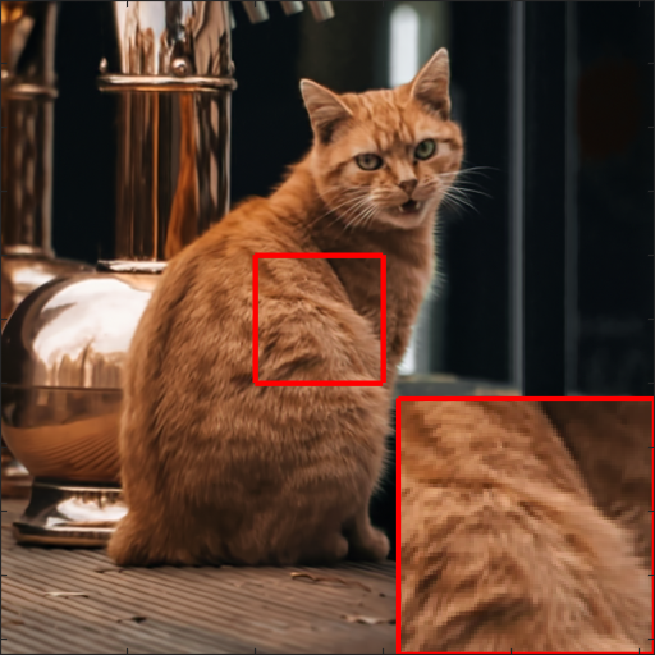}
		\centerline{\footnotesize{(d) QWSNM}}
	\end{minipage}
    \begin{minipage}{0.17\linewidth}
		\centering
		\includegraphics[width=1\linewidth]{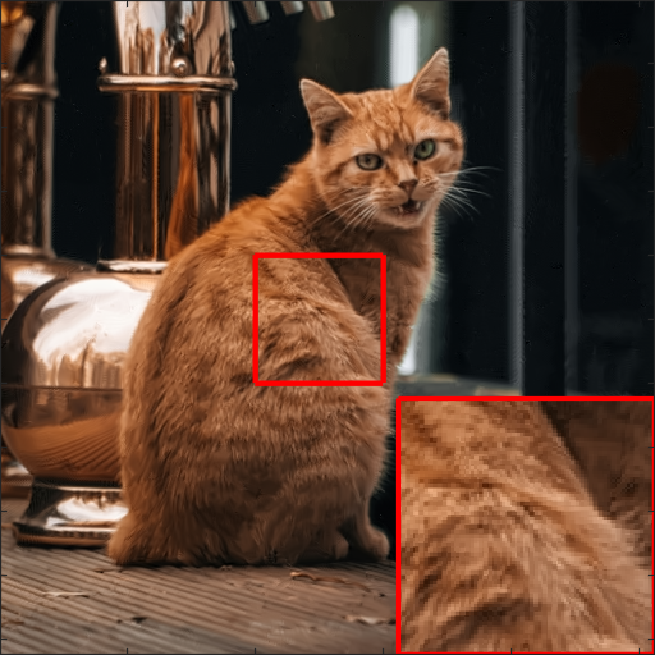}
		\centerline{\footnotesize{(e) MCNNFNM}}
	\end{minipage}
 \begin{minipage}{0.17\linewidth}
		\centering
		\includegraphics[width=1\linewidth]{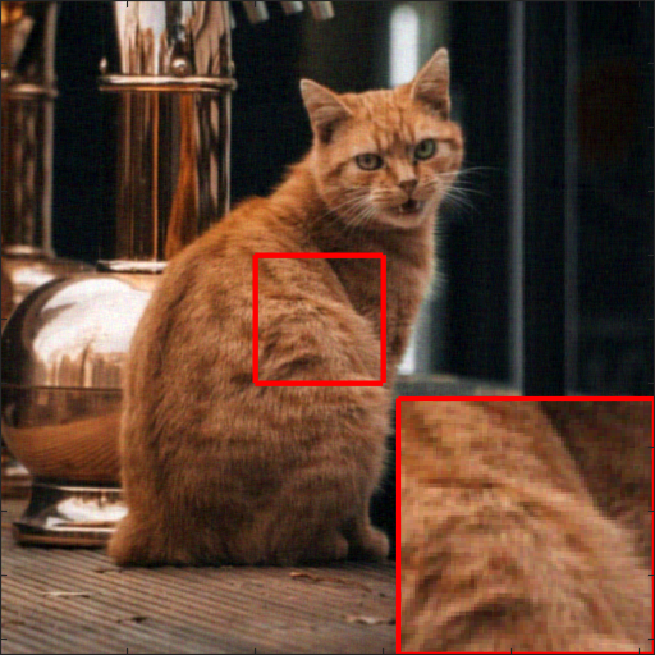}
		\centerline{\footnotesize{(f) SRRC}}
	\end{minipage}
 \begin{minipage}{0.17\linewidth}
		\centering
		\includegraphics[width=1\linewidth]{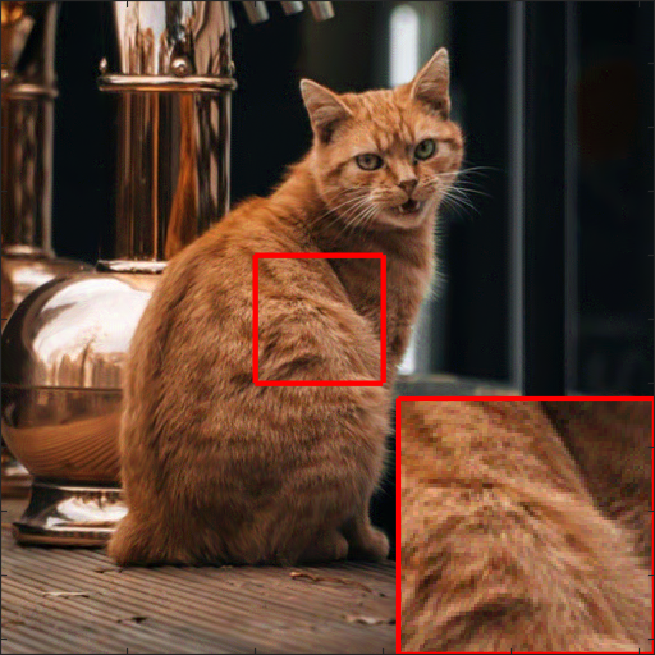}
		\centerline{\footnotesize{(g) QWNNM}}
	\end{minipage}
 \begin{minipage}{0.17\linewidth}
		\centering
		\includegraphics[width=1\linewidth]{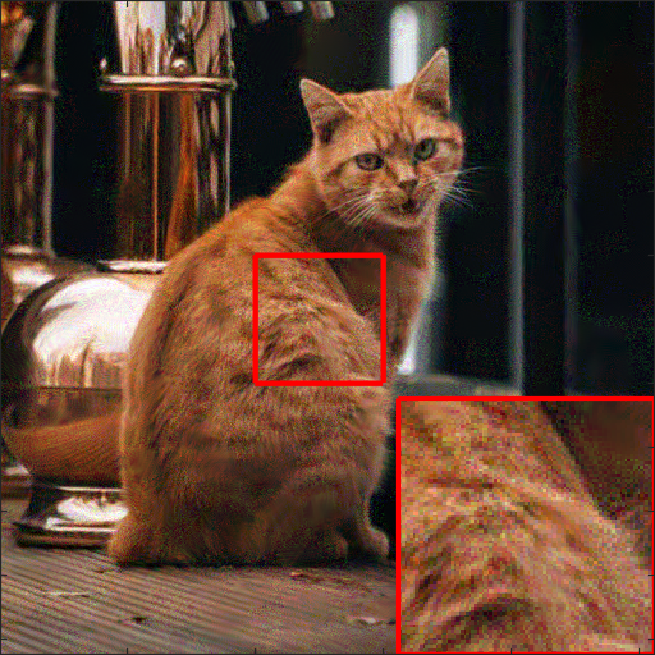}
		\centerline{\footnotesize{(h) DWT}}
	\end{minipage}
	\begin{minipage}{0.17\linewidth}
		\centering
		\includegraphics[width=1\linewidth]{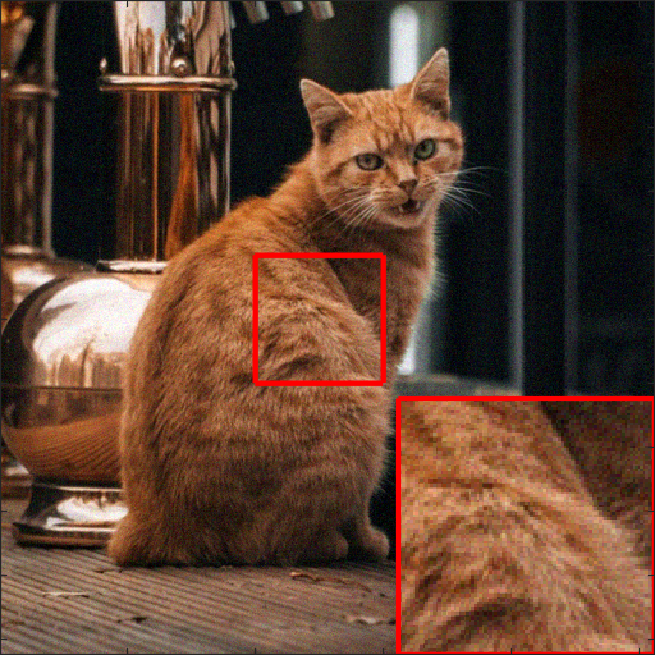}
	   \centerline{\footnotesize{(i) LRQA-1}}
	\end{minipage}
\begin{minipage}{0.17\linewidth}
		\centering
		\includegraphics[width=1\linewidth]{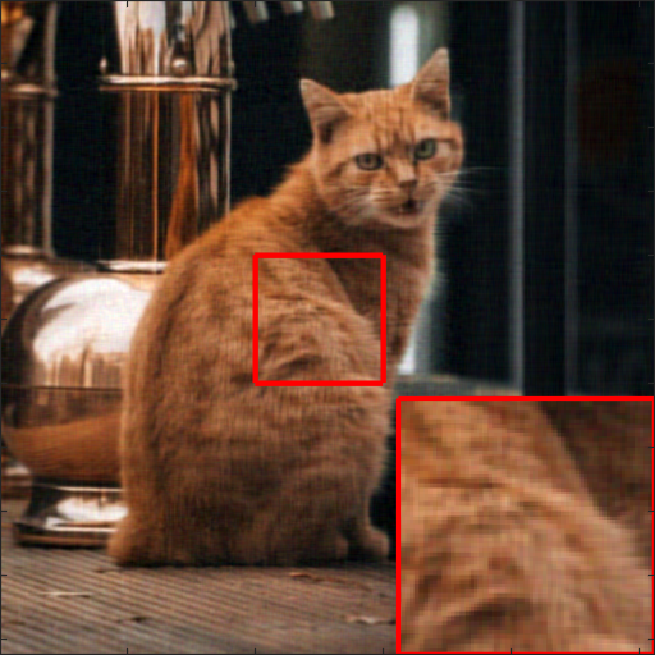}
	   \centerline{\footnotesize{(j) LRQA-2}}
	\end{minipage}
 \begin{minipage}{0.17\linewidth}
		\centering
		\includegraphics[width=1\linewidth]{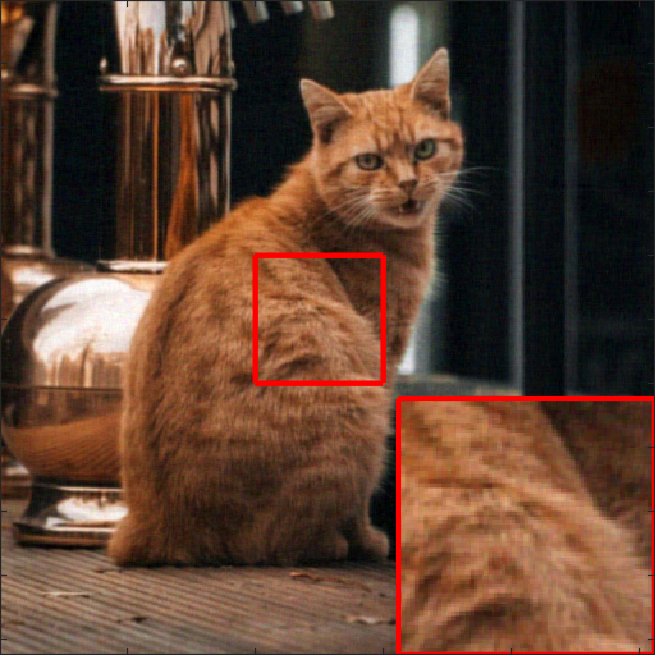}
	   \centerline{\footnotesize{(k) SLRQA-1}}
	\end{minipage}
 \begin{minipage}{0.17\linewidth}
		\centering
		\includegraphics[width=1\linewidth]{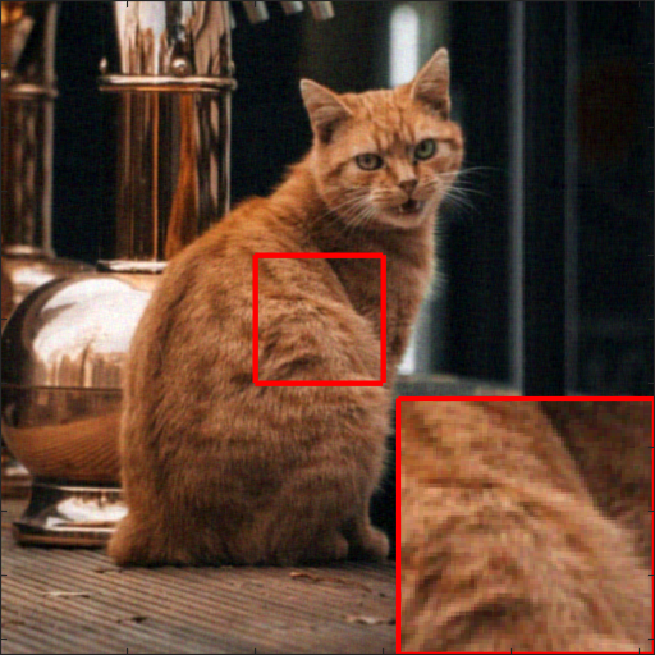}
	   \centerline{\footnotesize{(l) SLRQA-2}}
	\end{minipage}
 \begin{minipage}{0.17\linewidth}
		\centering
		\includegraphics[width=1\linewidth]{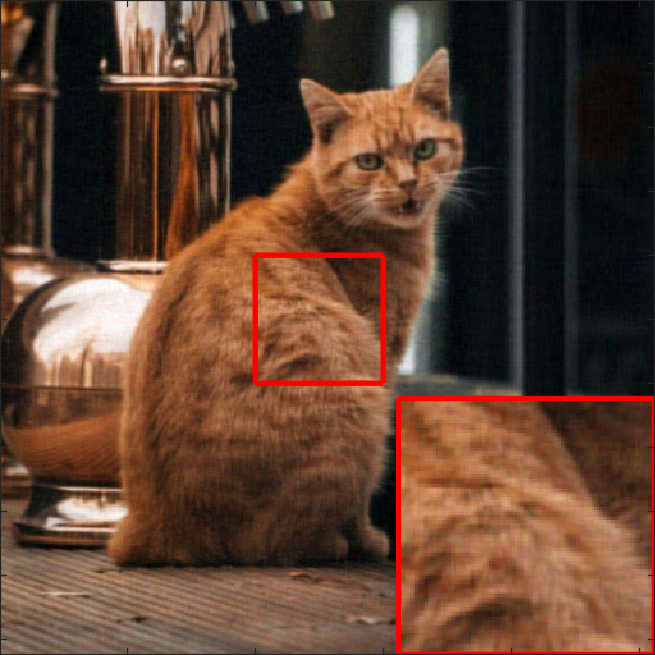}
	   \centerline{\footnotesize{(m) SLRQA-3}}
	\end{minipage}
 \caption{Color image denoising results of different methods on "Image5" with  $\tau = 10$.}

 \label{result1}
\end{figure}
\begin{figure}[h]
	\subfigure{
\includegraphics[width=0.32\textwidth]{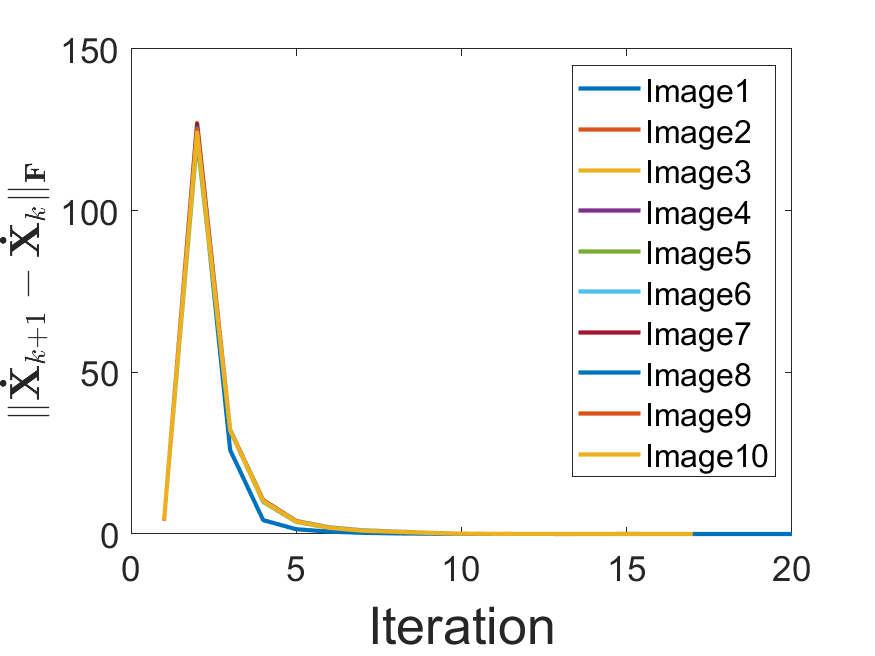}
	}
	\subfigure{
\includegraphics[width=0.32\textwidth]{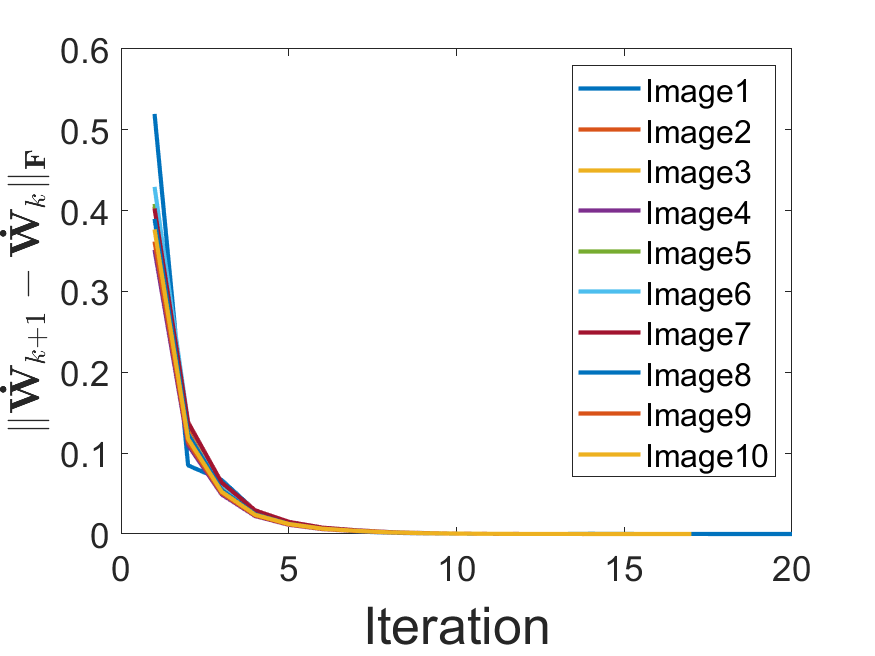}
	}
 \subfigure{
\includegraphics[width=0.32\textwidth]{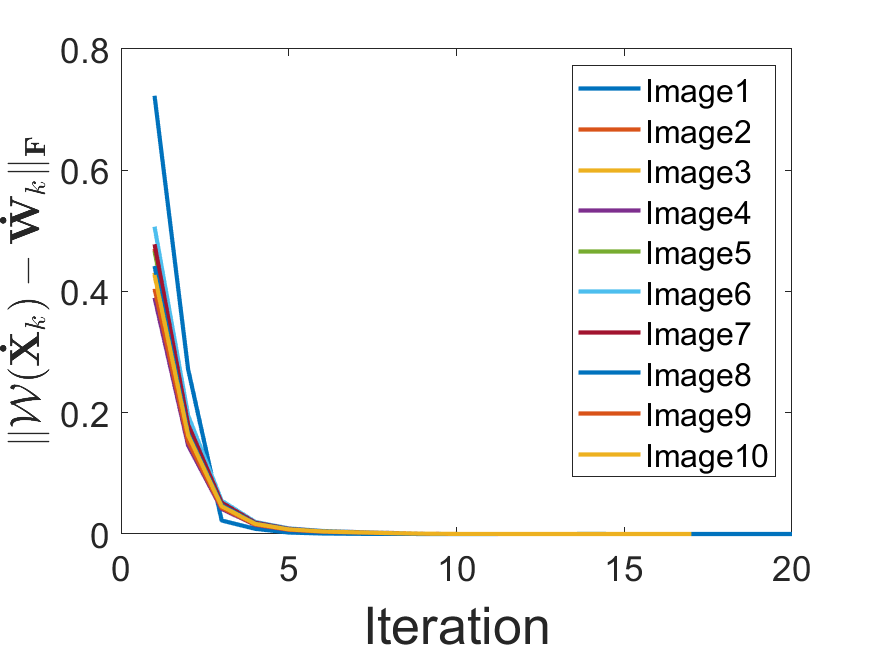}}
 \caption{The error convergence of SLRQA-1 for image denoising problems.}
  \label{fig:dn1}
 \end{figure}

\subsubsection{Color image denoising with NSS} \label{6-3-2}
For color image denoising problems with NSS, in addition to the methods compared in Section \ref{6-3-1}, we also compare SLRQA with CBM3D~\cite{MY2020CBMD}\footnote{The implementation is available from \url{https://webpages.tuni.fi/foi/GCF-BM3D/
} released 30 June 2024.}. Instead of constructing a denoising model, CBM3D employs block matching and sparse 3D collaborative filtering techniques to denoise the image, a process that cannot be achieved without using NSS. Hence we do not conduct a comparison with the CBM3D algorithm in the experiments of Section~\ref{6-3-1}.

The NSS prior is based on the fact that for a given local patch in a natural image, one can find many similar patches across the whole image. This technique has been widely used for image processing problems, such as image denoising, image deblurring, and image repairing. Here, the NSS process in~\cite{chen2019low} is used. Specifically, a noisy image $\mathbf{\dot{Y}}$ is divided into $n$ overlapping patches of size $\sqrt{d} \times \sqrt{d}$. Each patch is transformed into a vector $\mathbf{\dot{y}}^i \in \mathbb{H}^{d}, i = 1, 2,...,n$. Then for each patch, $s$ nearest neighbor patches are selected from a searching window with $L \times L$ pixels to form a set $K_i$. All patches in $K_i$ are stacked into a matrix $\mathbf{\dot{Y}}^i \in \mathbb{H}^{d\times s}$. We denoise $\mathbf{\dot{Y}}^i$ and aggregate all denoised patches together to form the clean color image $\mathbf{\dot{X}}$. Since all patches in each data matrix have similar structures, the constructed data matrix $\mathbf{\dot{Y}}^i$ is low rank. We also adopt the iterative regularization scheme in~\cite{chen2019low} and the iterative relaxation parameter is fixed to 0.1. 
For the parameter of NSS, we set patch size to 10$\times $10, 12$\times$12, and 14$\times$14 for $\tau$ = 10, 30, and 50, respectively. 
The number of selected non-local similar patches $n$ is set as 70, 80, and 90,  respectively. The searching window $L$ is fixed at 30.


The parameter $\gamma$ is set as 0.3, 0.8, and 0.3 respectively for SLRQA-1, SLRQA-2, and SLRQA-3. The parameter $w$ in SLRQA-3 is set as $w_{i,k} = 10 /(\sigma_i(\X_{k-1})+10^{-4})$.
The multiple values of $(\lambda, \tau)$ are used, i.e., $(\lambda, \tau) = (0.001, 10), (0.01, 30)$, and $(0.01, 50)$.
Intuitively, NSS enhances the low-rank property of the image. Hence, the parameter $\lambda$ is chosen to be smaller when compared to that without using NSS.


 The PSNR and SSIM values of each image between the competing methods under different noise levels are reported in Table~$\ref{TAB7}$. Most proposed SLRQA methods achieve higher PSNR and SSIM values than other methods. Moreover, SLRQA-3 achieves the best results in most cases. The visual comparisons of Image5 with $\tau = 30$ are shown in Figure~$\ref{result2}$ as a typical example. 
It is shown in Figure~\ref{result2} that the images recovered by CBM3D and QWNNM methods are still noisy. In addition, the images recovered by DWT, LRQA-1, and LRQA-2 are overly smooth, resulting in the loss of the features of the original images. In contrast, the image recovered by SLRQA-3 can not only remove the noise but also keep the features of the original images. 



From the above numerical experiments, we can conclude that the SLRQA has better denoising performance compared to other methods with or without using NSS in the sense of PSNR and SSIM. 

\begin{table*}[h]\scriptsize
	\setlength{\abovecaptionskip}{0.cm}
	\setlength{\belowcaptionskip}{-0.cm}
        \setlength{\tabcolsep}{2pt}
	\caption{PSNR/SSIM results on color image denoising problem of different methods with NSS. (To highlight the gap between methods, we keep three decimal digits for SSIM.)}
	\label{TAB7}
	\centering
	\resizebox{\linewidth}{!}{
		\begin{tabular}{|c|c|c|c|c|c|c|c|c|c|c|c|c|c|}
			\hline
			\multicolumn{2}{|c|}{\diagbox[innerwidth=2.2cm]{Image}{Algorithm}}&pQSTV~\cite{TTW2022pQSTV}&QWSNM~\cite{QHZ2024QWSMN}&MCNNFNM~\cite{YS2023MCNNFNM}&CBM3D$~\cite{MY2020CBMD}$&SRRC$~\cite{zhang2023SRRC}$&QWNNM$~\cite{huang2022quaternion}$&DWT$~\cite{cai2016image}$&LRQA-1~\cite{chen2019low} &LRQA-2~\cite{chen2019low}&SLRQA-1&SLRQA-2&SLRQA-3\\
			\hline
			\multirow{8}*{$\tau = 10$}&Image1&34.78/0.917 &35.48/0.921 & 35.10/0.913&35.71/0.920&35.67/0.919&34.16/0.895&34.62/0.913&35.19/0.919&35.34/0.922&35.68/0.924&35.58/0.921&\textbf{35.86/0.926}\\
			\cline{2-14}
			\multicolumn{1}{|c|}{}&Image2&30.84/0.862 & 31.98/0.844&31.25/0.839&32.22/0.855&32.02/0.849&31.37/0.862&30.89/0.850&31.94/0.868&31.88/0.866&32.13/0.865&32.19/0.869&\textbf{32.40/0.872}\\
			
                \cline{2-14}
			\multicolumn{1}{|c|}{}&Image3&36.85/0.934 &38.50/0.975 &37.89/0.960&37.67/0.948&37.72/0.943&36.37/0.940&36.22/0.934&37.51/0.955&38.62/0.972&38.71/0.973&38.79/0.978 &\textbf{38.83/0.984}\\
			
               \cline{2-14}
			\multicolumn{1}{|c|}{}&Image4&39.79/0.969 & 40.10/0.975&39.51/0.956&39.36/0.952&39.12/0.948&38.24/0.936&38.89/0.946&39.67/0.955 &40.87/0.972&40.99/0.974&\textbf{41.05/0.976}&40.36/\textbf{0.976}\\

               \cline{2-14}
			\multicolumn{1}{|c|}{}&Image5&36.49/0.920 & 37.05/0.951&36.29/0.943&37.08/0.951&36.82/0.944&35.00/0.921&35.87/0.942&36.48/0.947&36.71/0.945&37.01/0.949&37.18/0.951&\textbf{37.50/0.954}\\
			
			\cline{2-14}
			\multicolumn{1}{|c|}{}&Image6&35.56/0.924 & 37.12/0.962&37.08/0.963&37.02/0.969&36.32/0.962&35.42/0.942&36.02/0.959&36.65/0.966&36.66/0.966&\textbf{37.37/0.969}&37.22/0.968&37.27/0.969\\

               \cline{2-14}
			\multicolumn{1}{|c|}{}&Image7&34.11/0.901 &\textbf{ 35.08/0.916}&34.66/0.907&34.25/0.905&34.01/0.894&33.69/0.882&33.92/0.899&34.25/0.904&34.24/0.904 &34.68/0.904&34.82/0.905&35.00/0.906\\

               \cline{2-14}
			\multicolumn{1}{|c|}{}&Image8&34.27/0.923 & 35.77/0.938&35.46/0.934&\textbf{35.94/0.942}&34.82/0.932&34.25/0.910&34.66/0.928&35.14/0.935&35.24/ 0.935&35.84/0.941&35.75/0.940&35.92/\textbf{0.942}\\

           \cline{2-14}
			\multicolumn{1}{|c|}{}&Image9&34.15/0.829 &34.83/0.870 &33.37/0.826&33.25/0.805&32.81/0.794&32.60/0.788&32.32/0.749&33.75/0.804&34.14/0.884 &34.24/0.895&34.92/0.892&\textbf{35.01/0.902}\\

               \cline{2-14}
			\multicolumn{1}{|c|}{}&Image10&33.05/0.812 & 33.32/0.826&32.80/0.811&32.94/0.780&33.52/0.802&32.84/0.833&32.85/0.818&32.84/0.893&32.24/ 0.804&33.50/0.836&35.45/0.820&\textbf{34.08/0.842}\\

			\hline
                \hline
			\multirow{8}*{$\tau = 30$}&Image1&30.35/0.801 & 30.21/0.767&29.90/0.817&30.82/0.854&30.76/0.844&30.41/0.854&30.36/0.848&30.66/0.859&31.00/0.862&31.30/0.871&31.30/0.871&\textbf{31.40/0.877}\\
			\cline{2-14}
			\multicolumn{1}{|c|}{}&Image2&27.28/0.722 & 27.71/0.743&26.67/0.710&27.76/0.740&27.01/0.729&26.73/0.729&26.99/0.737&27.30/0.730&27.40/0.730&27.67/0.740&27.77/0.741&\textbf{28.00/0.743}\\
			
                \cline{2-14}
			\multicolumn{1}{|c|}{}&Image3&30.05/0.818 & 30.12/0.802&30.67/0.801&31.02/0.810&30.82/0.801&30.74/0.802&30.55/0.801 &30.93/0.811&31.10/0.813&31.15/0.813&31.21/0.814&\textbf{31.22/0.814}\\

                \cline{2-14}
			 \multicolumn{1}{|c|}{}&Image4&29.26/0.764& 29.52/0.770&28.85/0.767&29.45/0.802&29.22/0.791&29.10/0.806&29.00/0.798&29.37/0.805&29.203/0.81&29.57/0.806&29.74/0.807&\textbf{29.81/0.809}\\

               \cline{2-14}
			\multicolumn{1}{|c|}{}&Image5&30.64/0.846 & 30.02/0.832&30.35/0.816&32.21/0.847&31.02/0.846&31.08/0.846& 30.65/0.837
&31.34/0.853&31.30/0.853&\textbf{31.50/0.859}&31.32/0.853&31.34/0.854\\
			
			\cline{2-14}
			\multicolumn{1}{|c|}{}&Image6&31.35/0.902 & 30.19/0.823&30.90/0.898&31.55/0.932&30.87/0.923&31.01/0.922& 30.37/0.927&31.56/0.941&31.61/0.941&31.68/0.940&\textbf{31.72/0.941}&31.50/0.937\\

               \cline{2-14}
			\multicolumn{1}{|c|}{}&Image7&28.72/0.762&29.02/0.773 & 28.99/0.778&29.23/0.801&28.81/0.792&28.71/0.790& 28.94/0.794&29.09/0.799&29.21/0.801&29.32/0.801&29.39/0.801&\textbf{29.42/0.802}\\

               \cline{2-14}
			\multicolumn{1}{|c|}{}&Image8&30.04/0.808 & \textbf{30.29/0.834}&29.36/0.802&29.92/0.820&29.56/0.815&29.42/0.814& 28.68/0.817&29.82/0.819&29.90/0.822&29.80/0.818&30.06/0.823&30.12/0.827\\

        \cline{2-14}
			\multicolumn{1}{|c|}{}&Image9&29.00/0.770 &29.12/0.782 &28.87/0.712&29.25/0.795&29.01/0.794&28.80/0.738&28.12/0.690&29.05/0.799&29.24/0.804 &29.74/0.806&29.82/0.805&\textbf{30.02/0.816}\\

               \cline{2-14}
			\multicolumn{1}{|c|}{}&Image10&29.95/0.925 & 28.71/0.845&29.19/0.713&30.94/0.942&30.82/0.932&30.71/0.929&30.66/0.928&31.14/0.935&31.34/ 0.935&\textbf{31.93/0.942}&31.75/0.940&31.92/\textbf{0.942}\\

			\hline
                \hline
			\multirow{8}*{$\tau = 50$}&Image1&28.07/0.816 &28.21/0.767 &27.33/0.731&28.32/0.815&27.92/0.810&28.01/0.814&28.15/0.815&28.31/0.818&28.30/0.818&\textbf{28.60/0.821}&28.50/0.819&\textbf{28.60}/0.819\\
			\cline{2-14}
			\multicolumn{1}{|c|}{}&Image2&25.57/0.631 & 25.36/0.625&24.66/0.630&25.23/0.670&25.02/0.661&25.03/0.660&24.94/0.653&25.22/0.666&25.10/0.666&25.47/0.669&\textbf{25.77/0.674}&25.60/0.670\\
                \cline{2-14}
			\multicolumn{1}{|c|}{}&Image3&28.55/0.728 & 28.95/0.782&28.57/0.779&28.82/0.778&28.75/0.770&28.74/0.761&28.24/0.754&28.85/0.769&28.90/0.773&28.95/0.781&28.91/0.781&\textbf{28.99/0.784}\\
			
               \cline{2-14}
			\multicolumn{1}{|c|}{}&Image4&27.01/0.736 & 27.15/0.712&26.67/0.707&27.12/0.739&27.08/0.740&26.60/0.736&26.87/0.737&27.06/0.742&27.20/0.741&27.27/0.746&27.34/0.746&\textbf{27.51/0.749}\\

               \cline{2-14}
			\multicolumn{1}{|c|}{}&Image5&27.06/0.604 & 27.82/0.760&27.47/0.697&28.56/0.762&28.34/0.760&28.08/0.756& 28.01/0.755&28.41/0.769&28.30/0.767&28.50/0.769&28.62/\textbf{0.776}&\textbf{28.72}/0.775\\
			
			\cline{2-14}
			\multicolumn{1}{|c|}{}&Image6&28.90/0.907 &28.84/0.912 &27.81/0.823&28.77/0.900&28.53/0.895&28.01/0.892& 28.11/0.899&28.52/0.916&28.61/0.918&28.98/0.930&\textbf{29.02/0.931}&29.00/0.931\\

               \cline{2-14}
			\multicolumn{1}{|c|}{}&Image7&26.04/0.681 & 26.19/0.692&26.55/0.683&26.53/0.749&26.43/0.745&26.21/0.738&26.11/0.737&26.74/0.747&26.71/0.744&26.82/0.750&26.89/0.751&\textbf{26.92/0.753}\\

               \cline{2-14}
			\multicolumn{1}{|c|}{}&Image8&27.91/0.712 & 27.74/0.721&26.73/0.702&28.12/0.723&27.90/0.714&27.42/0.714&27.62/0.719&28.02/0.723&28.10/0.729&28.10/0.730&28.26/0.732&\textbf{28.32/0.732}\\

      \cline{2-14}
			\multicolumn{1}{|c|}{}&Image9&27.63/0.664 & 27.90/0.687&26.19/0.650&26.65/0.655&26.71/0.664&26.80/0.682&26.02/0.629&27.05/0.694&27.24/0.700 &27.34/0.701&27.62/0.705&\textbf{28.70/0.736}\\

               \cline{2-14}
			\multicolumn{1}{|c|}{}&Image10&27.07/0.665 &27.71/0.702 &27.06/0.646&27.34/0.702&27.02/0.682&26.81/0.639&26.67/0.618&27.34/0.652&27.64/ 0.701&27.73/0.687&27.75/0.702&\textbf{28.02/0.721}\\
			\hline
	\end{tabular}
	}
\end{table*}

\begin{figure}[h]
	\centering
	\begin{minipage}{0.17\linewidth}
		\centering
\includegraphics[width=0.9\linewidth]{denoising/image5_o.png}
		\centerline{(a) Image5}
	\end{minipage}    
	\begin{minipage}{0.17\linewidth}
		\centering		\includegraphics[width=0.9\linewidth]{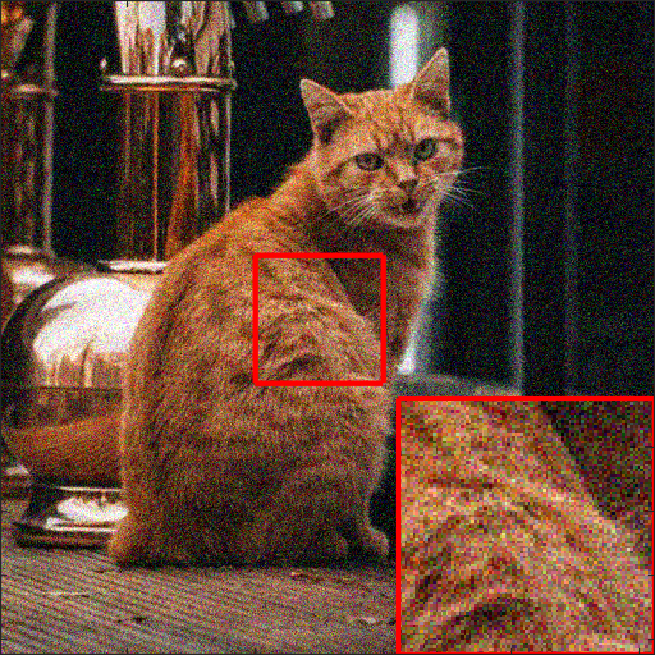}
		\centerline{(b) Noisy Image}
	\end{minipage}
\begin{minipage}{0.17\linewidth}
		\centering		\includegraphics[width=0.9\linewidth]{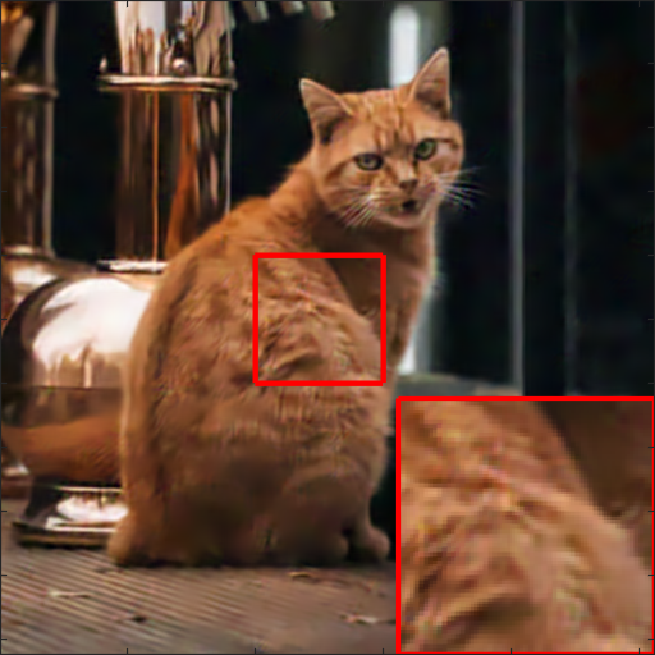}
		\centerline{(c) pQSTV}
	\end{minipage}    
    \begin{minipage}{0.17\linewidth}
		\centering		\includegraphics[width=0.9\linewidth]{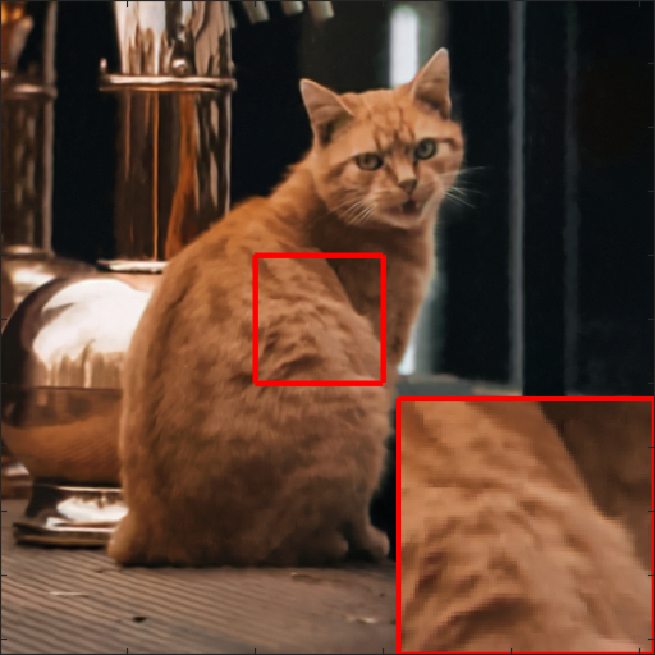}
		\centerline{(d) QWSNM}
	\end{minipage}
    \begin{minipage}{0.17\linewidth}
		\centering
		\includegraphics[width=0.9\linewidth]{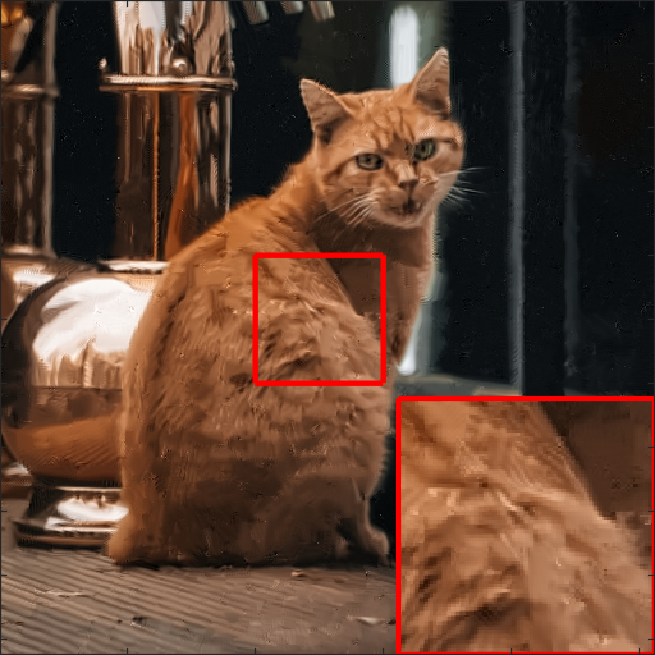}
		\centerline{(e) MCNNFNM}
	\end{minipage}   
	\begin{minipage}{0.17\linewidth}
		\centering
		\includegraphics[width=0.9\linewidth]{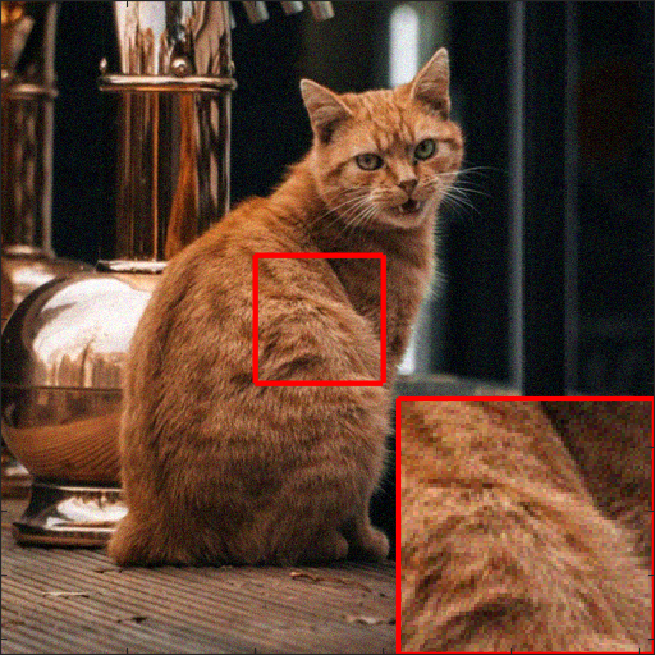}
		\centerline{(f) CBM3D}
	\end{minipage}
 \begin{minipage}{0.17\linewidth}
		\centering
		\includegraphics[width=0.9\linewidth]{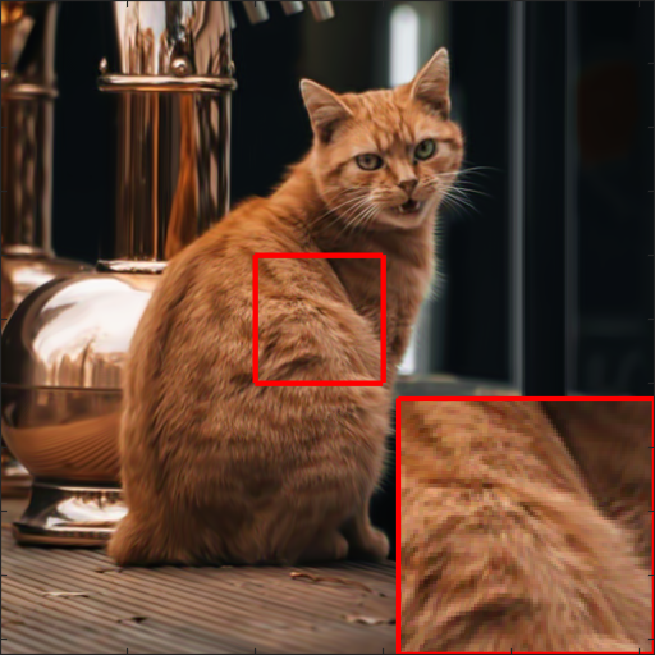}
		\centerline{(g) SRRC}
	\end{minipage}
 \begin{minipage}{0.17\linewidth}
		\centering
		\includegraphics[width=0.9\linewidth]{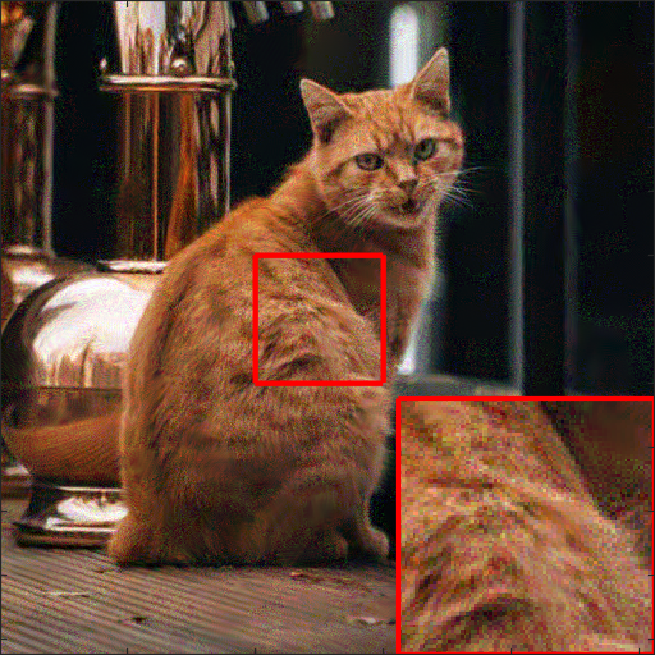}
		\centerline{(h) QWNNM}
	\end{minipage}
 \begin{minipage}{0.17\linewidth}
		\centering
		\includegraphics[width=0.9\linewidth]{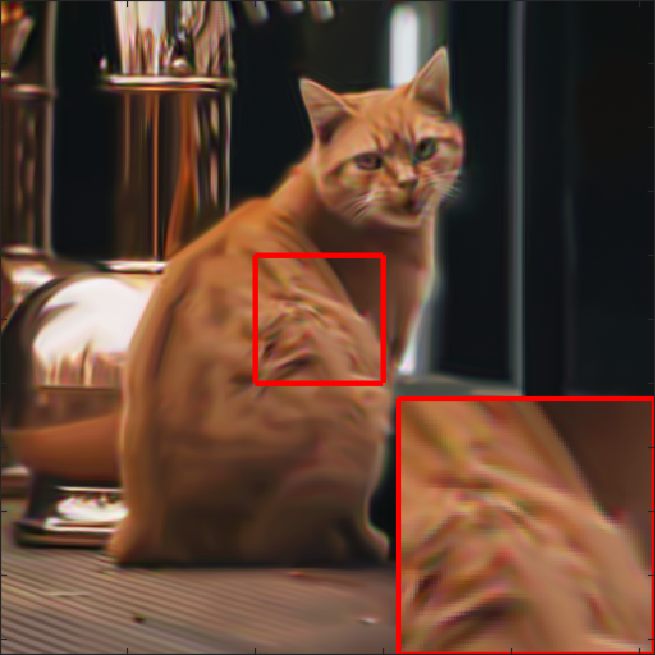}
		\centerline{(i) DWT}
	\end{minipage}
	\begin{minipage}{0.17\linewidth}
		\centering
		\includegraphics[width=0.9\linewidth]{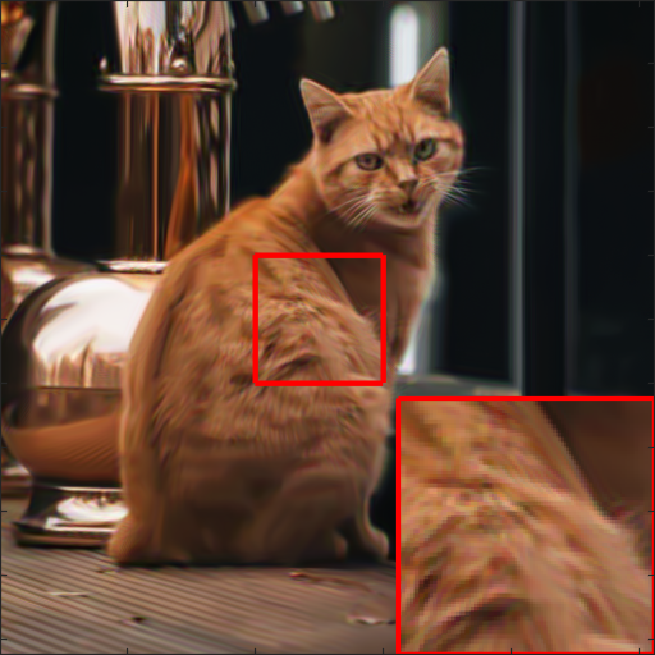}
	   \centerline{(j) LRQA-1}
	\end{minipage}
\begin{minipage}{0.17\linewidth}
		\centering
		\includegraphics[width=0.9\linewidth]{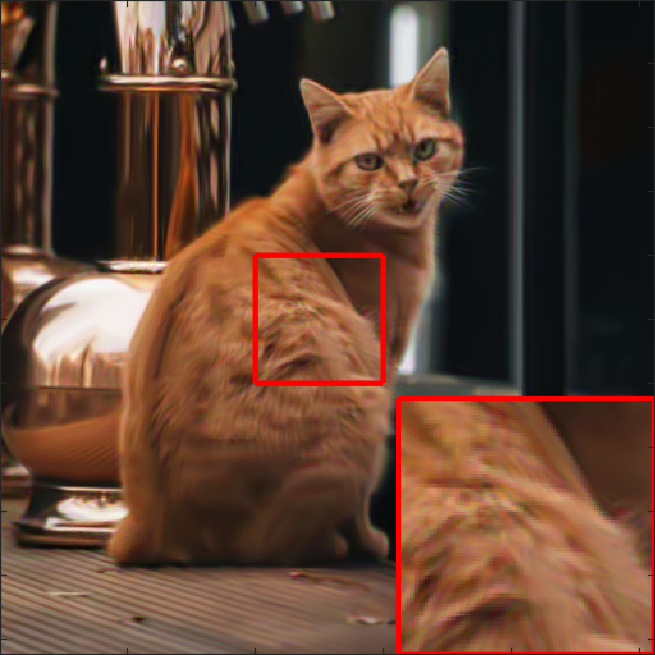}
	   \centerline{(k) LRQA-2}
	\end{minipage}
 \begin{minipage}{0.17\linewidth}
		\centering
		\includegraphics[width=0.9\linewidth]{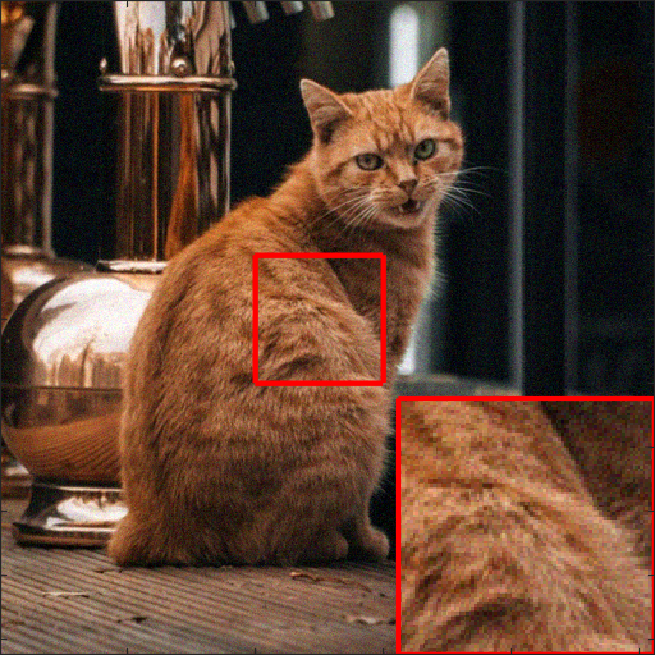}
	   \centerline{(l) SLRQA-1}
	\end{minipage}
 \begin{minipage}{0.17\linewidth}
		\centering
		\includegraphics[width=0.9\linewidth]{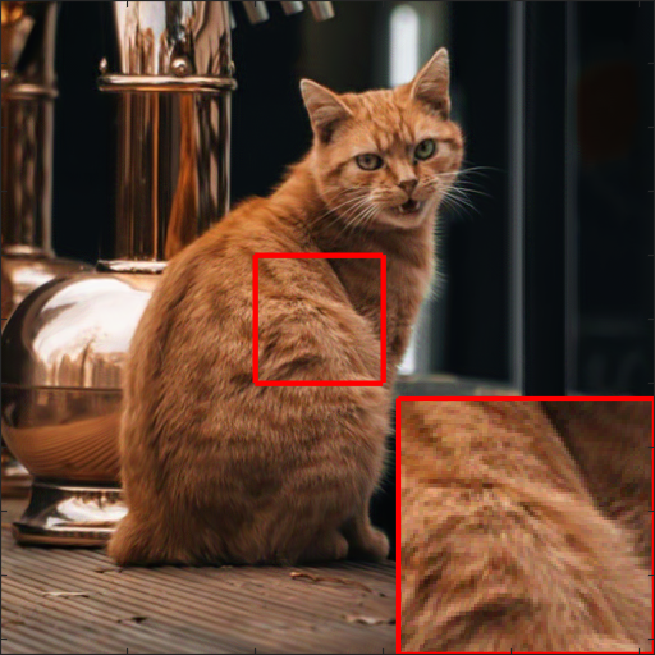}
	   \centerline{(m) SLRQA-2}		 
	\end{minipage}
 \begin{minipage}{0.17\linewidth}
		\centering
		\includegraphics[width=0.9\linewidth]{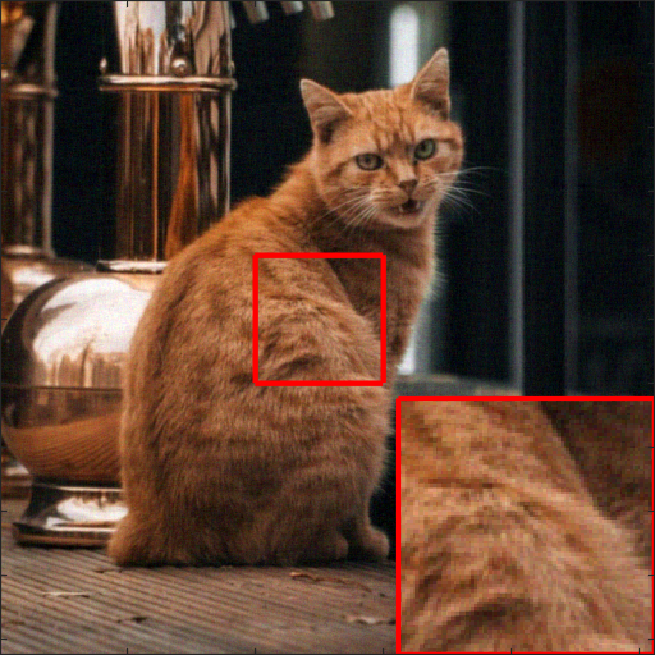}
	   \centerline{(n) SLRQA-3}
	\end{minipage}
 \caption{Color image denoising results of tested methods on ``Image5"  with  $\tau = 30$.}
 \label{result2}
\end{figure}

\subsection{ Color Image Inpainting} \label{sec:ColorImagInp}

For color image inpainting problems, the compared methods include Bilinear Factorization (BF) methods such as NCALR \cite{NCARL}, and Rank Minimization (RM) methods such as SLRI~\cite{liang2012repairing}, LRQA \cite{chen2019low}, and QWNNM \cite{huang2022quaternion}. \footnote{
The implementations that we use are available from: \\
LRQA: \href{https://www.fst.um.edu.mo/personal/wp-content/uploads/2021/05/LRQA.zip}{https://www.fst.um.edu.mo/personal/wp-content/uploads/2021/05/LRQA.zip} \\
QWNNM: \href{https://github.com/Huang-chao-yan/QWNNM}{https://github.com/Huang-chao-yan/QWNNM}\\
NCARL: \href{https://github.com/hyzhang98/NCARL}{https://github.com/hyzhang98/NCARL}.}.

We note that the BF methods require a rank estimate as a prior.
For fairness, the SLRI is extended to quaternion representation, which is referred to as QSLRI and performs better than the SLRI method. To fully demonstrate the performance of NCARL, we choose the best PSNR and SSIM from the initial rank {\rm range} $[50,\, 80,\, 100,\, 120,\, 150].$ The comparisons with LRQA and QSLRI can be regarded as ablation experiments to present the significance of sparsity and low-rankness prior respectively.

The results of all tested algorithms are shown in Table \ref{TAB:inpaint}. It is shown therein that SLRQA-NF-3 always enjoys the best performance compared with other methods in terms of PSNR and SSIM. 
The visual comparisons
between SLRQA-NF and all competing inpainting methods on the
 Image2 with $\chi = 0.7$ is shown in Figure \ref{result:inpaint} as a typical example. Although the NCARL method completes the image, the details of the recovered image are lost, and the recovered images are blurry. The images recovered by LRQA-1 and LRQA-2 methods have a relatively good effect, but the images are still slightly noisy or blurry.
 In addition, the images recovered by QWNNM are still little noisy. In contrast,  SLRQA-NF methods all demonstrate considerable inpainting results and have the highest PSNR and SSIM values. 
 

 \begin{figure}[H]
	\centering
	\begin{minipage}{0.16\linewidth}
		\centering
		\includegraphics[width=0.9\linewidth]{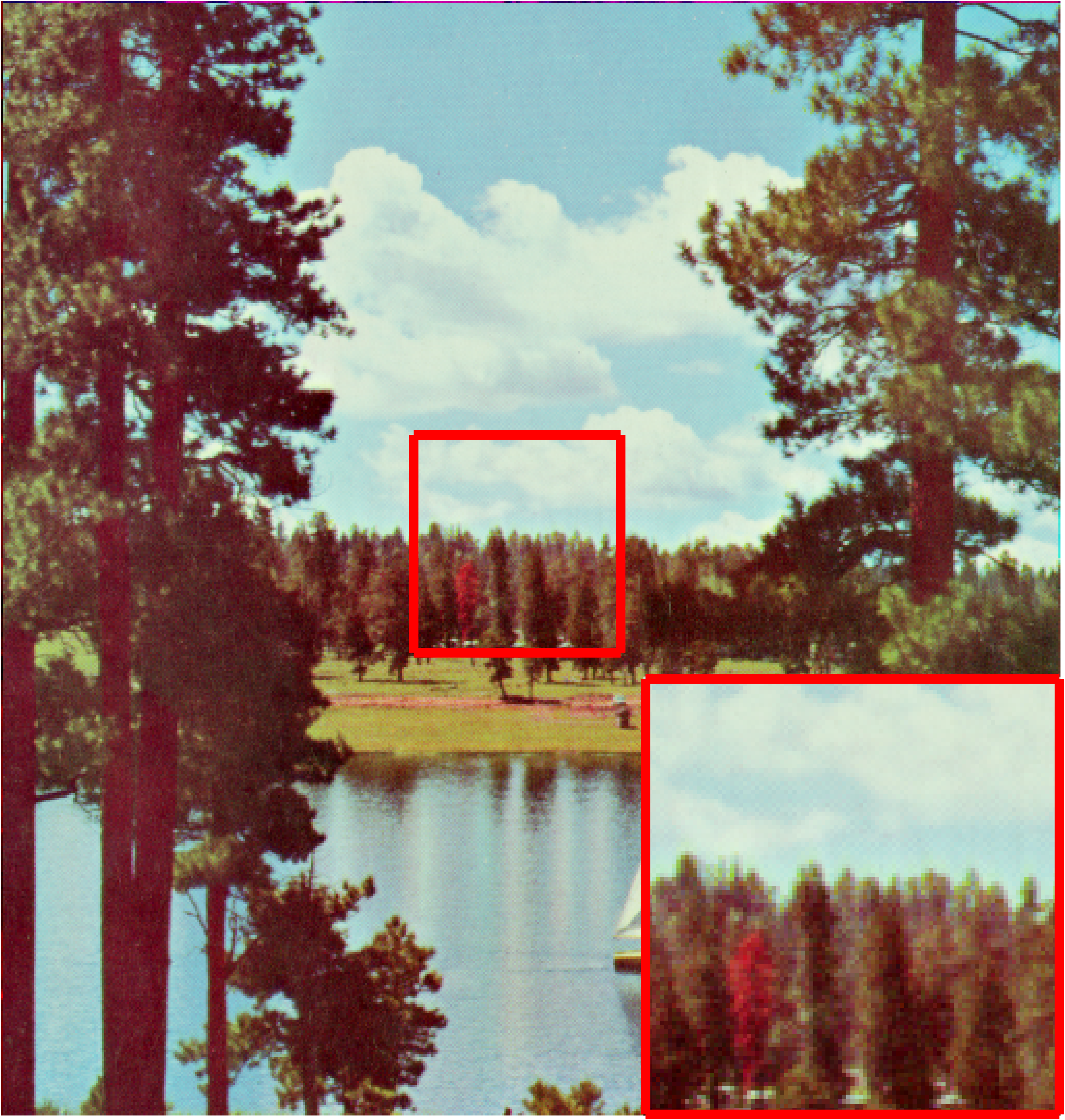}
		\centerline{(a)  Image2}
	\end{minipage}
	\begin{minipage}{0.16\linewidth}
		\centering
		\includegraphics[width=0.9\linewidth]{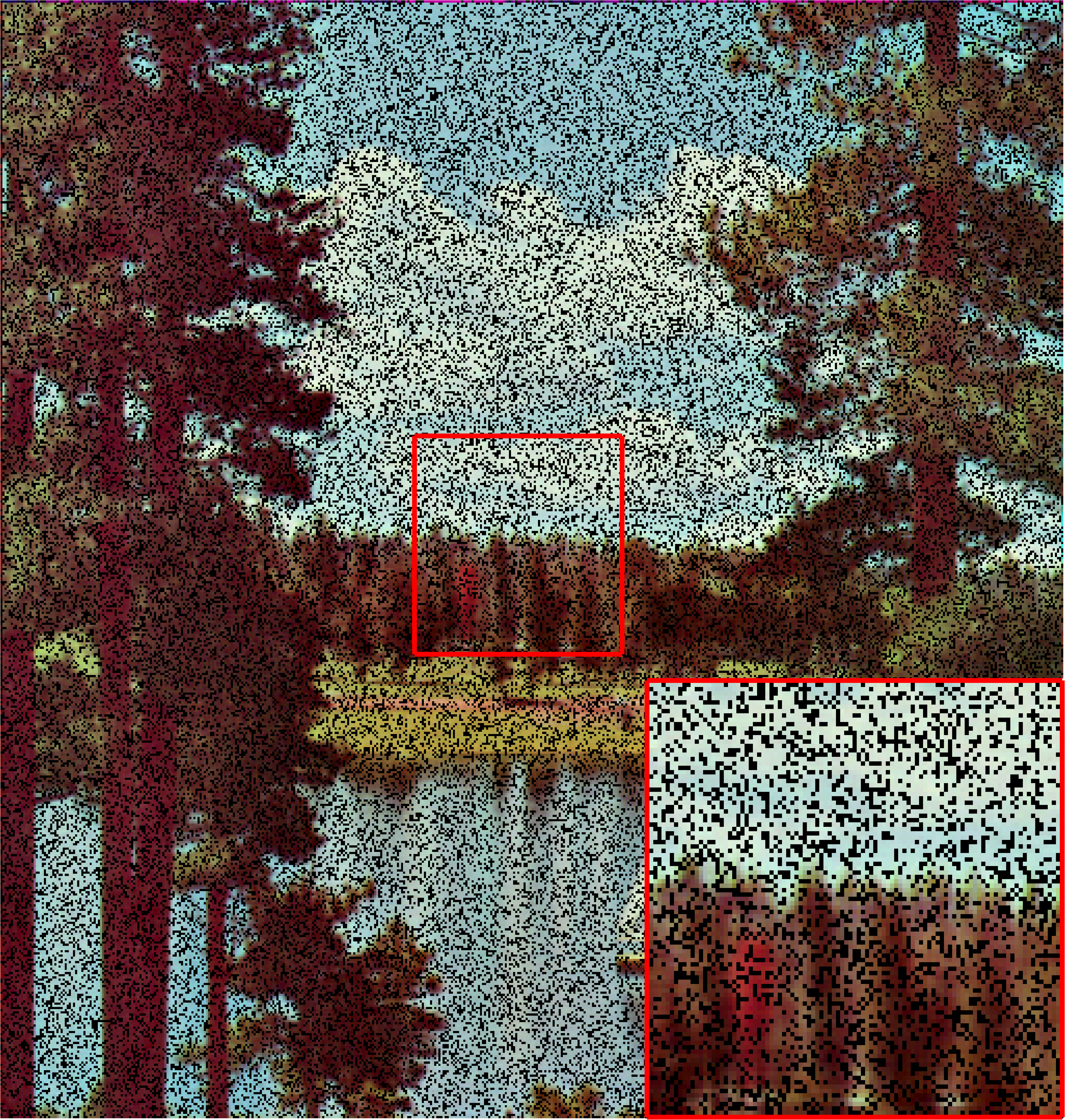}
		\centerline{(b) Corrupted}
	\end{minipage}
	\begin{minipage}{0.16\linewidth}
		\centering
		\includegraphics[width=0.9\linewidth]{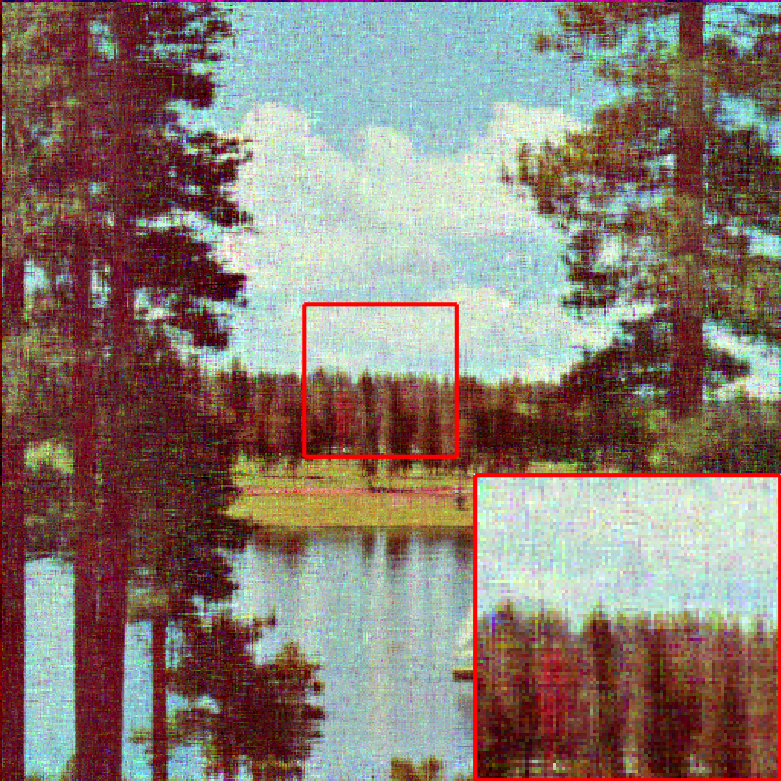}
                 \centerline{(c) NCARL}
	\end{minipage}
 \begin{minipage}{0.16\linewidth}
		\centering
		\includegraphics[width=0.9\linewidth]{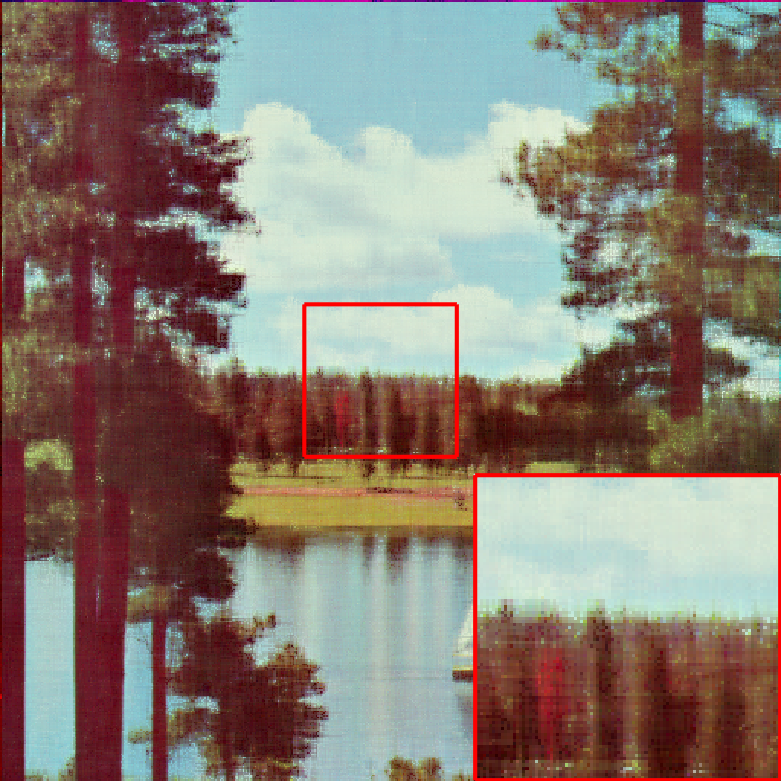}
        	   \centerline{(d) QWSNM}
	\end{minipage}
 \begin{minipage}{0.16\linewidth}
		\centering
		\includegraphics[width=0.9\linewidth]{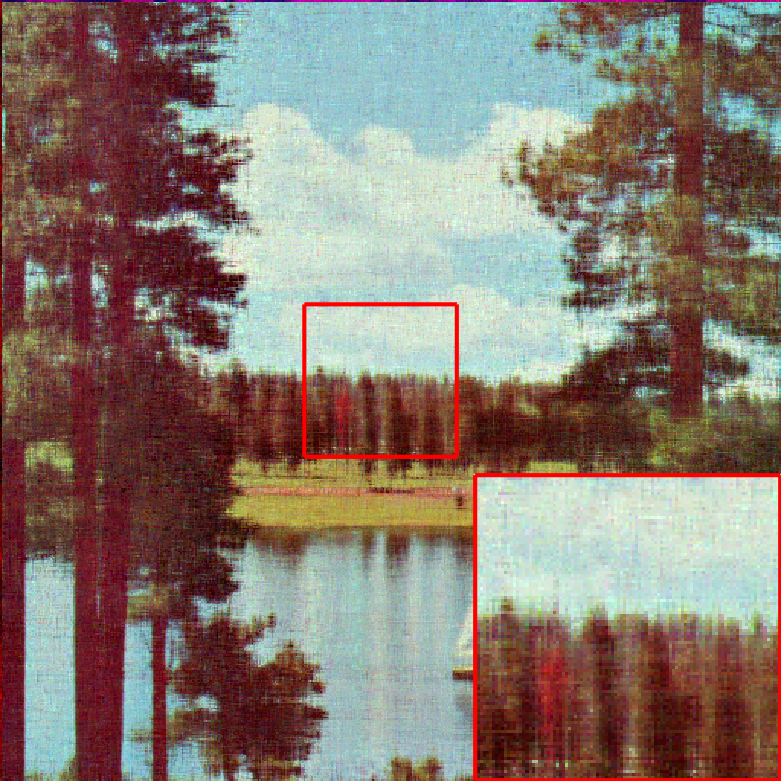}
		\centerline{(e) QWNNM}
	\end{minipage}
	\begin{minipage}{0.16\linewidth}
		\centering
		\includegraphics[width=0.9\linewidth]{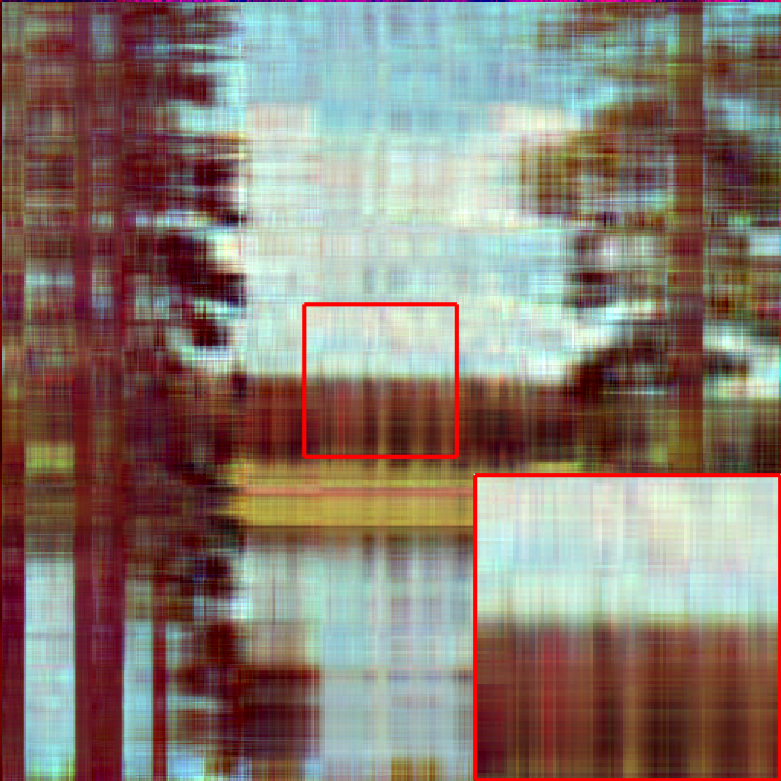}
     \centerline{(f) QSLRI}
	\end{minipage}
 \begin{minipage}{0.16\linewidth}
		\centering
		\includegraphics[width=0.9\linewidth]
        {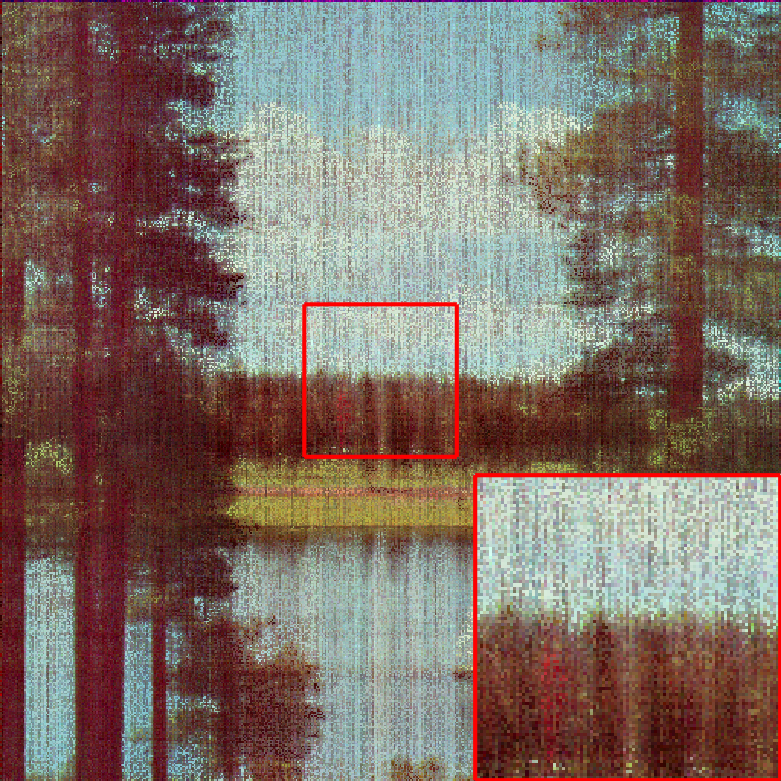}
	   \centerline{(g) LRQA-1}
	\end{minipage}
 \begin{minipage}{0.16\linewidth}
		\centering
		\includegraphics[width=0.9\linewidth]{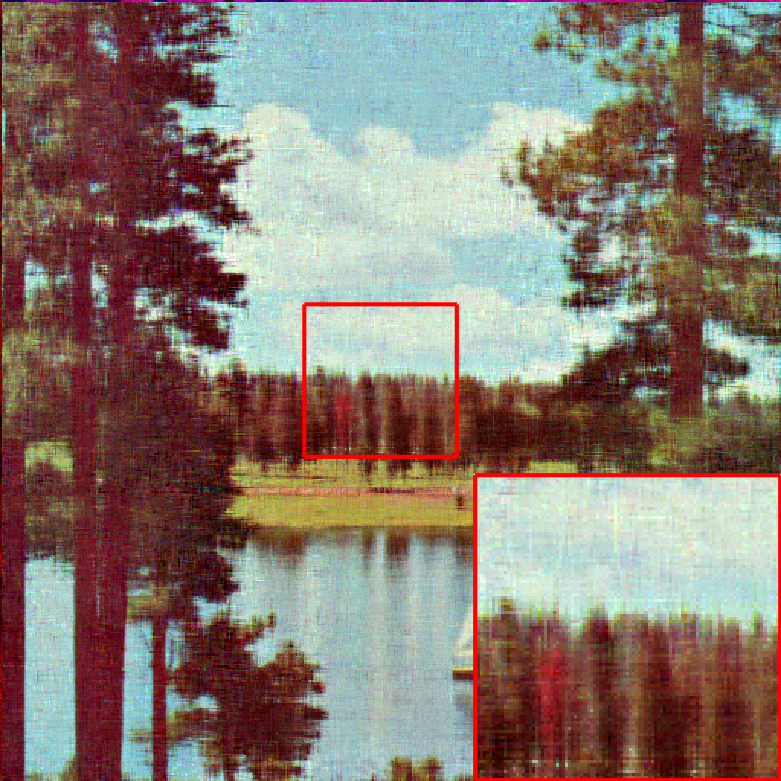}
	   \centerline{(h) LRQA-2}
	\end{minipage}
 \begin{minipage}{0.16\linewidth}
		\centering
		\includegraphics[width=0.9\linewidth]{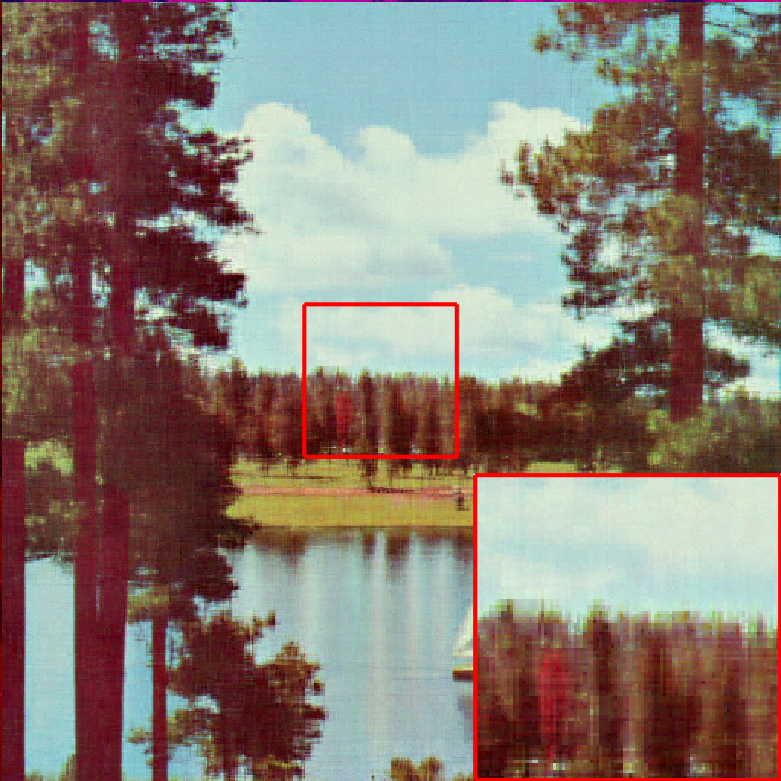}
	   \centerline{\small(i) SLRQA-NF1}
	\end{minipage}
  \begin{minipage}{0.16\linewidth}
		\centering
		\includegraphics[width=0.9\linewidth]{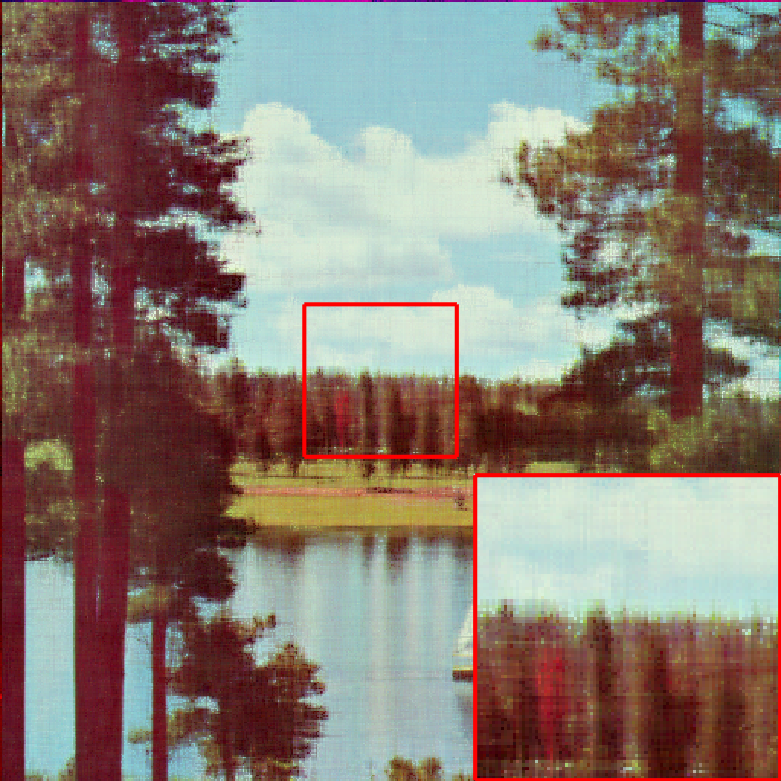}
	   \centerline{(j) SLRQA-NF2}
	\end{minipage}
\hspace{0.3cm}
  \begin{minipage}{0.16\linewidth}
		\centering
		\includegraphics[width=0.9\linewidth]{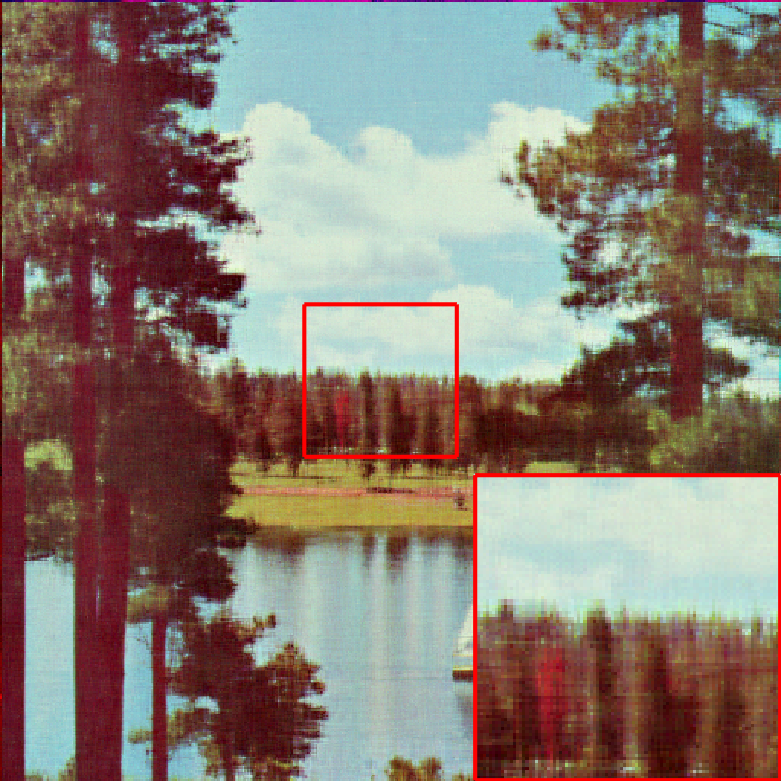}
	   \centerline{(k) SLRQA-NF3}

	\end{minipage}
 
	\caption{Color image inpainting result of tested methods on ``Image2'' with $\chi = 0.7$.  }
 \label{result:inpaint}
\end{figure}



To give a more detailed analysis, the errors of PL-ADMM-NF are also demonstrated for tested images. 
Figure \ref{fig:SLRQAcor} shows the empirical convergence results of SLRQA-NF-1, which indicate that even though the errors fluctuate at the early stage of iteration, they all converge in later iterations.
We can observe that $\|\Delta \X_{k+1} \|_{\mathrm{F}} + \|\Delta \W_{k+1} \|_{\mathrm{F}}   + \|\Delta \La_{k+1,1} \|_{\mathrm{F}} + \|\Delta \La_{k+1,2} \|_{\mathrm{F}} \rightarrow 0$. Hence, it follows from \eqref{eqn2:bound4} that $| \|\widetilde{\qd}_{k+1}\| | \rightarrow 0$.
 According to the update of $\La_{k+1}$ and Lemma \ref{crit2}, it follows that $(\X,\W,\La_1,\La_2)$ converges to a stationary point of $\mathcal{L}_{\beta_1,\beta_2}$ and hence a KKT point of the original problem in~\eqref{model:imageinpainting}.

\begin{table}[h]\scriptsize
\setlength{\abovecaptionskip}{0.cm}
\setlength{\belowcaptionskip}{-0.cm}
\setlength{\tabcolsep}{2pt}
\caption{PSNR/SSIM results on image inpainting problem of different methods.}
\label{TAB:inpaint}
\centering
\resizebox{\linewidth}{!}{
\begin{tabular}{|c|c|c|c|c|c|c|c|c|c|c|}
\hline
\multicolumn{2}{|c|}{\diagbox[innerwidth=2cm]{Image}{Algorithm}}& NCALR~\cite{NCARL} &QWSNM~\cite{QHZ2024QWSMN} &QWNNM ~\cite{huang2022quaternion} &QSLRI~\cite{liang2012repairing}&LRQA-1~ \cite{chen2019low} &LRQA-2~\cite{chen2019low}&SLRQA-NF-1&SLRQA-NF-2&SLRQA-NF-3\\
\hline
\multirow{10}*{$\chi = 0.5$}&Image1&26.48/0.85&28.16/0.86&25.25/0.60&31.16/0.95&30.87/0.89&30.25/0.87&32.32/0.94&31.26/0.94&\textbf{32.48/0.94}\\
\cline{2-11}
\multicolumn{1}{|c|}{}&Image2&25.55/0.90&24.56/0.87&23.55/0.81&30.89/0.85&26.63/0.91&26.20/0.90&28.04/0.94&27.54/0.94&\textbf{28.17/0.94}\\
\cline{2-11}
\multicolumn{1}{|c|}{}&Image3&26.75/0.91& 27.87/0.90&25.23/0.75&29.65/0.97&28.37/0.90&27.94/0.89&30.40/0.96&29.65/0.97&\textbf{30.61/0.97}\\
\cline{2-11}
\multicolumn{1}{|c|}{}&Image4&28.47/0.93&28.56/0.92&31.16/0.85&36.39/0.99&35.99/0.97&35.51/0.97&37.69/0.99&36.27/0.98&\textbf{37.91/0.99}\\
\cline{2-11}
\multicolumn{1}{|c|}{}&Image5&33.46/0.94&34.23/0.91&30.09/0.87&33.73/0.98&33.27/0.96&32.72/0.95&34.82/0.98&33.75/0.98&\textbf{34.98/0.98}\\
\cline{2-11}
\multicolumn{1}{|c|}{}&Image6&30.26/0.90&31.03/0.87&22.01/0.61&26.56/0.92&25.27/0.79&25.72/0.80&\textbf{28.25/0.93}&26.62/0.91&27.97/0.92\\
\cline{2-11}
\multicolumn{1}{|c|}{}&Image7&27.48/0.91&28.65/0.92&24.96/0.75&28.84/0.94&27.59/0.87&27.86/0.88&27.82/0.93&\textbf{28.92/0.94}&29.33/0.94\\
\cline{2-11}
\multicolumn{1}{|c|}{}&Image8&28.12/0.93&29.42/0.94&25.17/0.86&30.03/0.97&29.62/0.95&29.06/0.94&31.06/0.97&30.16/0.97&\textbf{31.22/0.97}\\
\cline{2-11}
\multicolumn{1}{|c|}{}&Image9&29.83/0.94&28.89/0.95&26.02/0.94&30.99/0.98&30.12/0.98&29.56/0.97&31.58/0.98&31.09/0.98&\textbf{31.71/0.98}\\
\cline{2-11}
\multicolumn{1}{|c|}{}&Image10&29.47/0.95&28.43/0.93&26.18/0.94&30.69/0.98&29.86/0.98&29.20/0.97&31.28/0.98&30.84/0.98&\textbf{31.42/0.98}\\
\hline
\hline
\multirow{10}*{$\chi = 0.7$}&Image1&27.39/0.84&27.41/0.85&23.05/0.52&27.29/0.89&25.97/0.74&25.47/0.72&28.40/0.88&27.30/0.88&\textbf{28.62/0.89}\\
\cline{2-11}
\multicolumn{1}{|c|}{}&Image2&25.66/0.87&24.11/0.85&20.13/0.67&24.19/0.89&22.51/0.79&22.44/0.79&25.05/0.90&24.30/0.89&\textbf{25.22/0.90}\\
\cline{2-11}
\multicolumn{1}{|c|}{}&Image3&26.44/0.93&24.12/0.92&21.50/0.63&26.09/0.93&23.69/0.77&23.67/0.77&26.87/0.93&26.00/0.92&\textbf{27.10/0.94}\\
\cline{2-11}
\multicolumn{1}{|c|}{}&Image4&31.32/0.96&30.47/0.95&28.34/0.80&31.89/0.97&31.07/0.93&30.44/0.92&33.37/0.97&31.80/0.97&\textbf{33.63/0.98}\\
\cline{2-11}
\multicolumn{1}{|c|}{}&Image5&26.55/0.91&27.46/0.92&25.75/0.77&29.58/0.95&28.53/0.88&27.85/0.86&30.95/0.96&29.68/0.95&\textbf{31.14/0.96}\\
\cline{2-11}
\multicolumn{1}{|c|}{}&Image6&23.77/0.74&20.57/0.75&18.13/0.44&22.78/0.82&20.28/0.57&19.86/0.55&23.78/0.81&22.70/0.80&\textbf{24.00/0.83}\\
\cline{2-11}
\multicolumn{1}{|c|}{}&Image7&24.01/0.83&25.43/0.85&22.40/0.67&26.24/0.90&24.62/0.78&24.58/0.78&\textbf{26.63/0.88}&26.19/0.89&26.57/0.88\\
\cline{2-11}
\multicolumn{1}{|c|}{}&Image8&26.08/0.93&25.45/0.92&22.34/0.79&26.34/0.94&24.87/0.87&24.54/0.87&27.41/0.94&26.42/0.94&\textbf{27.60/0.95}\\
\cline{2-11}
\multicolumn{1}{|c|}{}&Image9&22.36/0.87&24.15/0.93&23.33/0.90&27.83/0.97&26.13/0.94&25.69/0.94&28.56/0.97&27.81/0.97&\textbf{28.74/0.97}\\
\cline{2-11}
\multicolumn{1}{|c|}{}&Image10&23.14/0.91&23.76/0.92&23.22/0.90&27.54/0.96&25.95/0.95&25.34/0.94&28.46/0.97&27.58/0.96&\textbf{28.64/0.97}\\
\hline
\hline
\multirow{10}*{$\chi = 0.8$}&Image1&19.43/0.79&23.16/0.81&20.77/0.43&25.05/0.85&22.87/0.61&22.76/0.60&26.08/0.84&24.54/0.79&\textbf{26.30/0.85}\\
\cline{2-11}
\multicolumn{1}{|c|}{}&Image2&17.65/0.80&22.47/0.85&17.72/0.54&22.34/0.84&19.44/0.65&20.27/0.69&23.36/0.86&22.25/0.84&\textbf{23.53/0.86}\\
\cline{2-11}
\multicolumn{1}{|c|}{}&Image3&22.65/0.84&23.45/0.86&18.71/0.53&23.98/0.90&20.52/0.65&21.35/0.69&24.91/0.90&23.61/0.88&\textbf{25.13/0.91}\\
\cline{2-11}
\multicolumn{1}{|c|}{}&Image4&28.41/0.93&27.48/0.92&25.48/0.72&29.09/0.95&27.75/0.87&27.20/0.85&30.74/0.96&28.73/0.94&\textbf{31.04/0.96}\\
\cline{2-11}
\multicolumn{1}{|c|}{}&Image5&26.08/0.88&26.45/0.89&22.52/0.65&27.10/0.93&25.53/0.80&25.05/0.78&28.70/0.93&27.09/0.92&\textbf{28.97/0.94}\\
\cline{2-11}
\multicolumn{1}{|c|}{}&Image6&18.04/0.65&19.76/0.74&15.41/0.30&20.82/0.74&17.93/0.45&14.65/0.31&21.61/0.74&20.10/0.70&\textbf{21.82/0.76}\\
\cline{2-11}
\multicolumn{1}{|c|}{}&Image7&22.18/0.80&23.16/0.79&20.67/0.60&24.79/0.87&22.81/0.71&22.64/0.71&\textbf{25.60}/0.85&24.53/0.86&25.42/\textbf{0.86}\\
\cline{2-11}
\multicolumn{1}{|c|}{}&Image8&21.38/0.84&24.74/0.90&19.88/0.71&24.26/0.91&21.69/0.79&21.93/0.80&25.31/0.92&23.91/0.90&\textbf{25.50/0.92}\\
\cline{2-11}
\multicolumn{1}{|c|}{}&Image9&24.07/0.89&25.19/0.91&20.78/0.84&25.89/0.95&23.43/0.91&23.24/0.91&26.74/0.96&25.56/0.95&\textbf{26.97/0.96}\\
\cline{2-11}
\multicolumn{1}{|c|}{}&Image10&22.92/0.90&23.79/0.92&20.48/0.84&25.52/0.95&23.10/0.90&22.85/0.90&27.67/0.96&25.33/0.95&\textbf{26.91/0.96}\\
\hline
\end{tabular}
}
\end{table}

\begin{figure}[h]
	\subfigure{
		\includegraphics[width=0.45\textwidth]{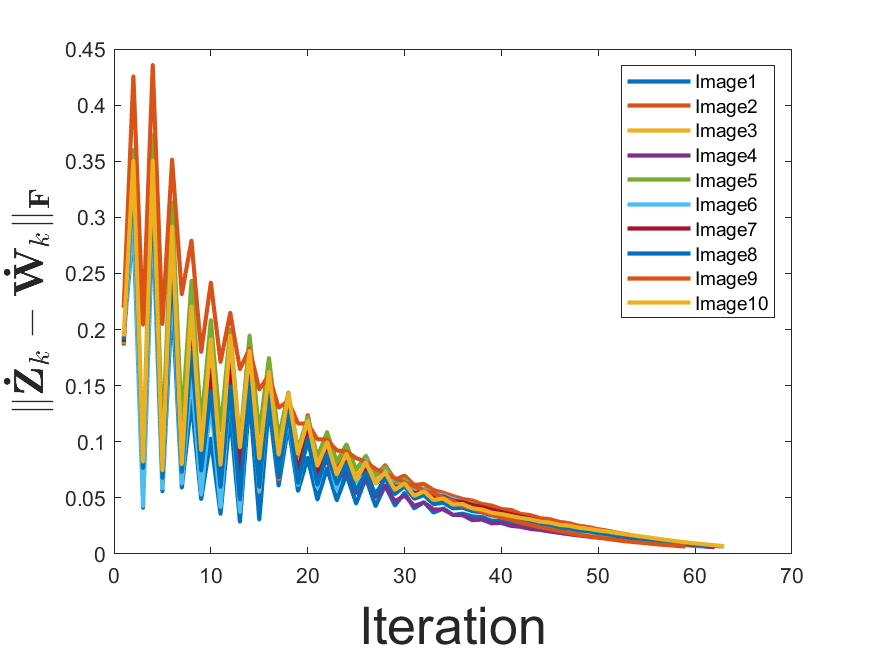}
	}
	\subfigure{
		\includegraphics[width=0.45\textwidth]{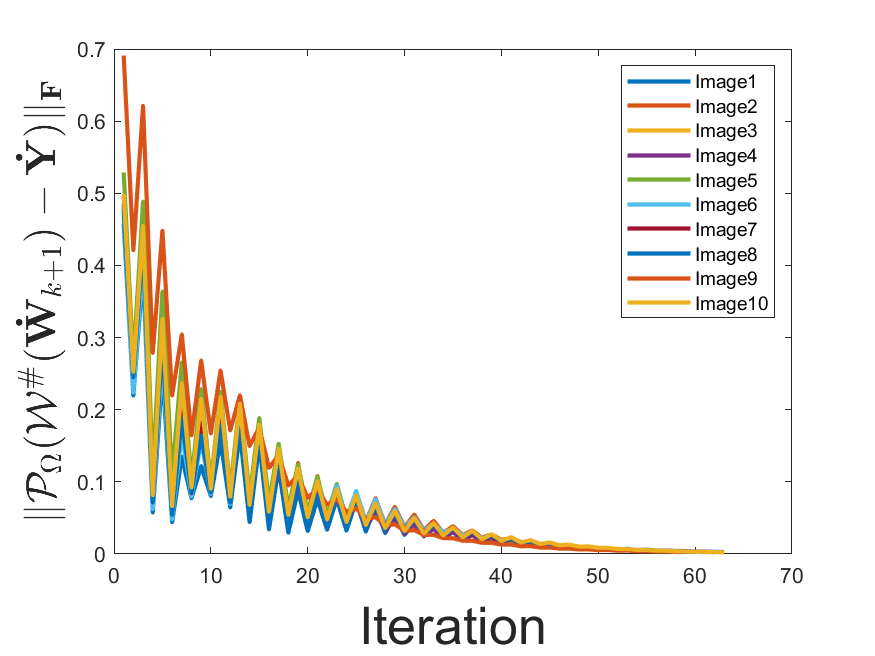}
  }
	\caption{The error convergence of SLRQA-NF for image inpainting problems.} \label{fig:SLRQAcor}
\end{figure}




\section{Conclusion} \label{6}
In this paper, we propose a novel SLRQA for the color image process problems.  Different from most existing models which only consider one or two properties from low-rankness, sparsity, and quaternion representation, SLRQA uses these properties all. Furthermore, SLRQA does not need an initial rank estimate. A PL-ADMM algorithm is proposed to solve the SLRQA and its global convergence is guaranteed.
  When the observation is noise-free, an SLRQA-NF of the limiting case of SLRQA is proposed. Subsequently, a  PL-ADMM-NF algorithm is also proposed to solve the SLRQA-NF. Under a newly proposed assumption, the global convergence of PL-ADMM-NF is established. To the best of our knowledge, this is the first ADMM-type algorithm without the ``range assumption'' and still guarantees global convergence. Extensive experiments for color image denoising and inpainting problems have demonstrated the robustness and effectiveness of  SLRQA and SLRQA-NF.


\section*{Data availability statement}
The data that support the findings of this study are openly available at the following URL/DOI: \url{https://github.com/dengzhanwang/SLRQA/tree/main}.

\bibliography{WHlibrary}

\end{document}